\documentclass[bibliography=totoc, 12pt]{amsart}
\usepackage[utf8]{inputenc}
\usepackage[margin=1in]{geometry} 
\usepackage{amsmath,amsthm,amssymb,amsfonts, dsfont, graphicx,mathtools,mathrsfs,tikz,hyperref,physics, stmaryrd, bm, tikz-cd,tkz-euclide,comment, enumerate, appendix, todonotes}
\usepackage{subcaption}
\usepackage[backend=biber,bibencoding=utf8,style=numeric,url=false,
    isbn=false,     
    doi=false,           
    abbreviate=true,            
    giveninits=true, 
    natbib=true,
    hyperref=true]{biblatex}
\usetikzlibrary{decorations.pathreplacing,decorations.markings}

\newcommand{\ZZ}{\mathbb{Z}}
\newcommand{\QQ}{\mathbb{Q}}
\newcommand{\CC}{\mathbb{C}}

\newcommand{\FF}{\mathbb{F}}
\newcommand{\A}{\mathbb{A}}
\newcommand{\PP}{\mathbb{P}}

\DeclareMathOperator{\supp}{supp}

\DeclareMathOperator{\Frob}{Frob}

\DeclareMathOperator{\GL}{GL}

\DeclareMathOperator{\diag}{diag}

\DeclareMathOperator{\cha}{char}

\DeclareMathOperator{\rk}{rank}

\def\O{\mathcal{O}}
\def\a{\mathfrak{a}}

\def\RR{\mathbb{R}}
\def\TT{\mathbb{T}}

\theoremstyle{plain}
\newtheorem{theorem}{Theorem}[section] 
\newtheorem{corollary}[theorem]{Corollary}
\newtheorem{lemma}[theorem]{Lemma}
\newtheorem{prop}[theorem]{Proposition}

\numberwithin{equation}{section}

\theoremstyle{definition}

\theoremstyle{remark}
\newtheorem*{remark}{Remark}

\title[Complete intersections of cubic and quadric hypersurfaces over $\FF_q(t)$]{Rational points on complete intersections of cubic and quadric hypersurfaces over $\FF_q(t)$}
\author[J. Glas]{Jakob Glas}
\address{IST Austria, Am Campus 1, 3400 Klosterneuburg, Austria.}

\email{jakob.glas@ist.ac.at}


\begin{document}

\maketitle
\begin{abstract}
    Using a two-dimensional version of the delta method, we establish an asymptotic formula for the number of rational points of bounded height on non-singular complete intersections of cubic and quadric hypersurfaces of dimension at least $23$ over $\FF_q(t)$, provided $\cha(\FF_q)>3$. Under the same hypotheses, we also verify weak approximation.
\end{abstract}
\tableofcontents
\section{Introduction}
Understanding the arithmetic of varieties over global fields constitutes one of the most fundamental and difficult ambitions of number theory. If $X\subset \PP^{n-1}$ is a variety over a global field $K$, then one key aspect entails understanding the counting function 
\[
N_X(P)\coloneqq \#\{\bm{x}\in X(K)\colon H(\bm{x})<P\}
\]
for a suitable height function $H\colon X(K)\to \RR_{\geq 0}$. The Hardy--Littlewood circle method is a versatile tool to study $N_X(P)$ and has been successful in producing an asymptotic formula in many situations. The circle method works particularly well when $n$ is large compared to the degree of $X$ and so the main challenge lies in reducing the number of admissible variables. In this work we shall focus on the case when $X=V(F_1,F_2)\subset \PP^{n-1}$ is a non-singular complete intersection of a cubic and a quadric hypersurface over $K=\FF_q(t)$. To state our main result, we take a smooth weight function $w\colon \FF_q((t^{-1}))^n\to \RR_{\geq 0}$ that is supported around a suitable point $\bm{x}_0\in \FF_q((t^{-1}))^n$ with $F_1(\bm{x}_0)=F_2(\bm{x}_0)=0$ and consider the counting function 
\[
N(P)\coloneqq \sum_{\substack{\bm{x}\in \FF_q[t]^n \\ F_1(\bm{x})=F_2(\bm{x})=0}}w\left(\frac{\bm{x}}{t^{P}}\right)\quad \text{as }P\to \infty.
\]
In addition, we set  $|\bm{x}|=\max q^{\deg x_i}$ for $\bm{x}\in\FF_q[t]^n$ and for $P\in\RR$, we write $\widehat{P}=q^P$.
\begin{theorem}\label{Th: TheTheorem}
    Let $X\subset \PP^{n-1}$ be a non-singular complete intersection of a cubic and a quadric hypersurface over $\FF_q(t)$. If $n\geq 26$ and $\cha(\FF_q)>3$, then there exists $\delta>0$ such that 
    \[
    N(P)= c\widehat{P}^{n-5}+O\left(\widehat{P}^{n-5-\delta}\right),
    \]
    for some $c>0$.
\end{theorem}
To put our result into context, Browning--Dietmann--Heath-Brown \cite{browning2015rational} proved the analogue of Theorem \ref{Th: TheTheorem} over $K=\QQ$ for $n\geq 29$. Another result in this direction when both $F_1$ and $F_2$ are diagonal is due to Wooley \cite{Wooley_Add1, wooley_Add2}. Again working over $\QQ$, he restricted the range of possible integer solutions to those having only few prime factors. Appealing to the theory of smooth Weyl sums this allowed him to provide an asymptotic formula for the number of such restricted solutions for $n\geq 13$ whenever $F_1$ and $F_2$ have at least $7$ and $5$ non-zero coefficients respectively.

Let us now review some results in the general situation of non-singular complete intersections $X=V(F_1,\dots, F_R)\subset \PP^{n-1}$ over a global field $K$, where $d_i=\deg F_i$ for $i=1,\dots, R$. Provided $n>d_1+\cdots +d_R$, the variety $X$ is Fano --- a class of varieties for which Manin and his collaborators \cite{franke1989rational} made a precise conjecture about the asymptotic behaviour of the counting function $N_X(P)$. Originally the conjecture was only stated over number fields, but it was later generalised to global function fields by Peyre \cite{peyrePosGlob}. In our case, the order of growth established in Theorem \ref{Th: TheTheorem} agrees with this prediction. The constant $c$ appearing in Theorem \ref{Th: TheTheorem} plays an important role in the qualitative understanding of the set of rational points $X(K)$ and has obtained a conjectural interpretation by Peyre \cite{peyre1995hauteurs}. Recall that $X$ is said to satisfy the Hasse principle if the existence of $K_\nu$-rational points for every place $\nu$ of $K$ is sufficient to guarantee that $X(K)$ is non-empty. In addition, we say that $X$ satisfies weak approximation if the image under the diagonal embedding 
\[
X(K)\hookrightarrow \prod_\nu X(K_\nu)
\]
is dense, where we endow the right hand side with the product topology coming from the analytic topology on $X(K_\nu)$. Whenever the circle method yields the Hasse principle, we can often adapt the argument to give weak approximation. 

The application of the circle method to the study of rational points goes back to pioneering work of Birch \cite{birch1962forms}. For $K=\QQ$, he gave an asymptotic formula for $N_X(P)$ provided \break $d\coloneqq d_1= \cdots = d_R$ and $n-\sigma^*>(d-1)2^{d-1}R(R+1)$, where $\sigma^*$ is the dimension of the Birch singular locus, and thereby verified that $X$ satisfies the Hasse principle. Schmidt \cite{Schmidt} showed how to handle to complete intersections defined by forms of differing degrees. This was subsequently improved by Browning and Heath-Brown \cite{browning2017forms} and by Myerson \cite{myerson2017systems} for generic complete intersections. Myerson was able to drop the genericity assumption for quadratic $F_i$ \cite{myerson2018quadratic} and cubic $F_i$ \cite{myerson2019systems}. However, his work only improves Birch's result for large values of $R$. In a different direction, Northey and Vishe \cite{northey2021hasse} developed a procedure that opens up the possibility for a Kloostermann refinement for systems of two forms with $d_1+d_2>5$ and used it to give an asymptotic for $N_X(P)$ and verify the Hasse principle for non-singular intersections of two cubic forms if $n\geq 39$.

In the function field setting the Hasse principle is a straightforward consequence of Lang--Tsen theory if $n>d_1^2+\cdots +d_R^2$. Indeed, with this condition \cite[Theorem 3.6]{greenberg1969lectures} shows the existence of a rational point of $X$, so that the Hasse principle is vacuously true. However, establishing weak approximation or an asymptotic formula for the number of rational points of bounded height remains a substantial challenge. Building on work of Kubota \cite{R1974}, Lee \cite{lee2011birch}  showed that Birch's work can be translated to the function field setting and further demonstrated that weak approximation holds under the same constraints on the number of variables provided $\cha(K)>d$. Specialising again to the case when $X$ is the non-singular intersection of a cubic and a quadric hypersurface, Tsen's Theorem shows that $X$ posseses a $K$-rational point as soon as $n >13$. Although there should be no problem in translating the work of Browning--Dietmann--Heath-Brown \cite{browning2015rational} and of Wooley \cite{Wooley_Add1}, \cite{ wooley_Add2}, currently there are no results available regarding weak approximation for complete intersections of cubic and quadric hypersurfaces.  Our second main result remedies this deficiency.
\begin{theorem}\label{Th: WA}
    Let $X\subset \PP^{n-1}$ be a non-singular complete intersection of a cubic and a quadric hypersurface over $\FF_q(t)$. If $n\geq 26$ and $\cha(\FF_q)>3$, then $X$ satisfies weak approximation.
\end{theorem}
The restriction on the characteristic in both Theorems \ref{Th: TheTheorem} and \ref{Th: WA}  arises naturally in applications of the circle method. Typically, it comes from applications of Weyl differencing, which renders any estimates trivial when the characteristic is smaller than the degrees of the equations. In our situation we have to study both quadratic and cubic exponential sums that we can only bound satisfactorily when $\cha(K)>3$. 

Browning and Vishe \cite{BrownVisheCurves} have found a way to use the circle method over $\FF_q(t)$ to obtain crude geometric information about the space of rational curves of fixed degree inside a hypersurface in sufficiently many variables compared to its degree. If one is willing to make all the estimates uniform in $q$, then our work is likely to give access to the analogous properties when $X$ is the intersection of a cubic and a quadric hypersurface. 

Using a geometric approach, Tian \cite{tian2017hasse} has verified the Hasse principle for non-singular cubic hypersurfaces $X\subset \PP^{n-1}$ when $\cha(K)>5$ and $n\geq 6$ and weak approximation for non-singular intersections of quadratic forms when $\cha(K)>2$ and $n\geq 6$. It would be interesting to see whether his methods carry over to say something useful about intersections of cubic and quadric hypersurfaces. 

When the degree of a form $F$ is small, the delta method developed by Duke, Friedlander and Iwaniec \cite{duke1993bounds} and further refined by Heath-Brown \cite{heath1996new} is capable of dealing with significantly fewer variables than the classical circle method. In particular, it has been successfully applied to quadratic forms \cite{heath1996new} and cubic forms \cite{hooley2014octonary}. Over $\FF_q(t)$ the delta method turns out to be much simpler thanks to the non-archimedean nature and was successfully incorporated by Browning and Vishe \cite{browning2015rational} . However, until recently it was unclear how to construct an analogue of the delta method for systems of equations. Vishe \cite{vishe2019rational} made substantial progress by developing a 2-dimensional analogue of the delta method over \break $K=\FF_q(t)$ that enabled him to produce an asymptotic formula for the number of rational points of bounded height on non-singular intersections of two quadratic forms in $n\geq 9$ variables when $\cha(K)>2$. His innovation serves as the main input for our work and we shall now proceed to outline the main steps of our proof.

\subsection*{Outline of proof}
In \cite{vishe2019rational} the main new input is the development of a two-dimensional version of a Farey dissection over $\FF_q(t)$. While Vishe's version only allows one to put squares around the approximating rationals, our application requires lopsided boxes in order to take into account the different degrees of the forms $F_1$ and $F_2$. In Section \ref{Sec: Fary} we shall modify his development to accommodate our needs. We expect that the argument would carry over inductively to higher dimensions. In the case of hypersurfaces the delta method is particularly useful when the degree is at most 3. Unless one appeals to a similar strategy as devised by Marmon and Vishe \cite{marmon} to deal with quartic hypersurfaces, it seems that when one considers intersections of two hypersurfaces, our situation is just at the barrier. That is, when the sum of the degrees of the individual hypersurfaces exceeds 5, it does not seem to give any improvements compared to the classical circle method.

Once we have achieved the Farey dissection, a standard application of the Poisson summation formula leads us to study certain oscillatory integrals and exponential sums.  We provide upper bounds for the oscillatory integrals in Section \ref{Sec: ExpIntegrals} and  for the exponential sums and averages thereof in Sections \ref{Se: Exp1} and \ref{Se: Exp2}. When the modulus is square-free, we estimate the exponential sums by appealing to work of Katz \cite{katz1999estimates}, that ultimately relies on Deligne's resolution of the Riemann hypothesis over finite fields \cite{DeligneWII} and obtain cancellations when summing over the numerators. This is usually referred to as a ``Kloosterman refinement''. As in \cite{vishe2019rational} and \cite{CubicHypersurfacesBV} it would have been desirable to obtain a double Kloosterman refinement,  in which we extract cancellations when summing over both the numerators and denominators. In \cite{vishe2019rational} and \cite{CubicHypersurfacesBV} the corresponding exponential sums are multiplicative and their averages over the denominator can be studied via the associated $L$-functions that satisfy a suitable version of the Riemann hypothesis. In our setting, we consider exponential sums associated to linear combinations of the cubic and  quadratic form. That these are not homogeneous only allows for a ``twisted'' form of multiplicativity and it is not clear how to associate an $L$-function to study their averages. There remains the substantial task of providing estimates for exponential sums when the modulus is not square-free. We are unable to give upper bounds directly, but rather study averages of them over the dual variable in Section \ref{Se: Exp2}. The underlying arguments go back to work of Heath-Brown \cite{heath1983cubic}, but are significantly more complicated in our situation.
\subsection*{Acknowledgements} The author would like to thank his supervisor Tim Browning for suggesting this project and many helpful conversations and Pankaj Vishe for useful comments. Moreover, he is grateful to Dante Bonolis and Julian Lyczak for sharing their expertise in exponential sums and geometry. While working on this paper the author was supported by FWF grant P 36278.
\subsection*{Notation} We use the standard Vinogradov notation $\ll, \gg$ and the big-$O$ notation interchangeably.  All our implied constants are allowed to depend $q$. Any further dependence will be indicated by a subscript, unless mentioned otherwise. Moreover, $\varepsilon$ always denotes an arbitrarily small number whose value might change from one occurrence to the next. Its presence in an inequality implies that the constant may also depend on $\varepsilon$. 
\section{Farey dissection}\label{Sec: Fary}
Vishe's strategy is to find a suitable family of lines in the unit square so that when we consider rational points on these lines they cover the whole square and at the same time stay sufficiently far away from each other to ensure an exact partition. Before reviewing his results in more detail, we need some notation. Let $K=\FF_q(t)$ and $\O=\FF_q[t]$ be its ring of integers. We denote by $K_\infty \coloneqq \FF_q((t^{-1}))$ the completion of $K$ with respect to the absolute value induced by $t^{-1}$. It naturally comes with a norm extending the absolute value $|f|=q^{\deg f}$ on $\O$. More explicitly, any non-zero $\alpha\in K_\infty$ can be written uniquely as $\alpha =\sum_{i\leq N}\alpha_it^i$ with  $\alpha_N\neq 0$ and $\alpha_i\in\FF_q$, in which case $|\alpha|=q^N$. We also define $\{\alpha\}\coloneqq \sum_{i\leq -1}\alpha_it^i$ to be the \emph{fractional part} of $\alpha$ and $[\alpha]\coloneqq \alpha - \{\alpha\}$ to be the \emph{integer part} of $\alpha$. 

Furthermore, $K_\infty$ is a locally compact Hausdorff group with respect to addition and so comes with a Haar measure $\dd \alpha$ that we normalise in such a way that the unit interval $\TT\coloneqq \{\alpha \in K_\infty \colon |\alpha| <1\}$ has measure 1. We extend the norm and the Haar measure to $K_\infty^n$ via 
\[
|\underline{\alpha}|\coloneqq \max_{i=1,\dots, n}|\alpha_i|\quad \text{and}\quad \dd\underline{\alpha}\coloneqq \dd\alpha_1\dots \dd\alpha_n
\]
for $\underline{\alpha}=(\alpha_1,\dots, \alpha_n)\in K_\infty^n$. If $\underline{c}=(c_1,c_2)\in \O^2$, then we say that $\underline{c}$ is \emph{primitive} if $ (c_1,c_2)=1$ and either $c_1$ is monic or $c_1=0$ and $c_2$ is monic. For $d, k\in\O$ with $(d,k)=1$ and $\underline{c}\in\O^2$ primitive we define the affine line 
\begin{equation}
    L_1(d\underline{c},k)\coloneqq \{\underline{x}\in K_\infty^2\colon d\underline{c}\cdot\underline{x}=k\}
\end{equation}
and the \emph{generalised line}
\begin{equation}\label{Eq: GenLine}
    L(d\underline{c})\coloneqq \{\underline{a}/r \in \TT^2\cap K^2\colon (\underline{a},r)=1,  \underline{a}/r \in L_1(d\underline{c},k) \text{ for some }k\in\O \text{ with }(k,d)=1\}.
\end{equation}
Note that since $(k,d)=1$, we must have $d\mid r$ if $\underline{a}/r \in L(d\underline{c})$ with $(\underline{a},r)=1$. We refer to $|d\underline{c}|$ as the \emph{height} of $L(d\underline{c})$. For $\underline{x}\in \TT^2$ and $R\in \ZZ$, we let $B(\underline{x}, \widehat{R})\coloneqq \{\underline{\theta}\in \TT^2\colon |\underline{\theta}-\underline{x}|<\widehat{R}\}$ be the ball of radius $\widehat{R}$ centered around $\underline{x}$. Similarly, for $\underline{\widehat{R}}=(\widehat{R}_1,\widehat{R}_2)\in \ZZ$ we set 
\[
R(\underline{x}, \underline{\widehat{R}})\coloneqq \{\underline{\theta}\in \TT^2 \colon |\theta_i-x_i|<\widehat{R}_i \text{ for }i=1,2\}\]
to be a rectangle of sidelengths $\widehat{R}_1$ and $\widehat{R}_2$ centered around $\underline{x}$. In addition, $\underline{\widehat{R}}^{-1}$ denotes the vector $(\widehat{R}_1^{-1},\widehat{R}_2^{-1})$. We are now in a position to state Vishe's partition of the unit square.
\begin{theorem}\label{Th: FareyPankaj}
Let $Q\geq 1$. Then 
\[
\TT^2=\bigsqcup_{\substack{r \text{ monic}\\ |r|\leq \widehat{Q}}} \bigsqcup_{\substack{d\mid r\text{ monic}\\ \underline{c}\in \O^2\text{ primitive}\\ |r|\widehat{Q}^{-1/2}\leq |d\underline{c}|\leq |r|^{1/2}\\ |dc_2|<|r|^{1/2}}}\sideset{}{'}\bigsqcup_{\substack{|\underline{a}|<|r|\\ \underline{a}/r \in L(d\underline{c})}}B(\underline{a}/r, |r|^{-1}\widehat{Q}^{-1/2}),
\]
where $\bigsqcup'$ indicates that we only consider vectors $\underline{a}$ such that $(\underline{a},r)=1$. 
\end{theorem}
A few remarks are in order to explain the conditions on the lines in the theorem. First, from a standard application of Dirichlet's approximation theorem one obtains
\begin{equation}\label{Eq: Dirichlet}
\TT^k = \bigcup_{\substack{|r|\leq \widehat{Q}\\ r\text{ monic}}}\sideset{}{'}\bigcup_{|\underline{a}|<|r|}B(\underline{a}/r, |r|^{-1}\widehat{Q}^{-1/k})
\end{equation}
for any $Q>0$ and $k\geq 1$. Furthermore, using the pigeon-hole principle one can show that that any $\underline{a}/r\in K^2$ with $|\underline{a}|<|r|$ and $(\underline{a},r)=1$ lies on a line $L(d\underline{c})$ with $|d\underline{c}|\leq |r|^{1/2}$ and $|dc_2|<|r|^{1/2}$. It also clear from the definition \eqref{Eq: GenLine} that we must have $d\mid r$. So the key condition in Theorem \ref{Th: FareyPankaj} is $|r|\widehat{Q}^{-1/2}\leq |d\underline{c}|$. This guarantees that 
\begin{enumerate}[(i)]
    \item rational points on an individual line stay sufficiently far away from each other,
    \item rational points on distinct lines stay sufficiently far away from each other,
    \item distinct lines don't intersect at rationals with small denominator.
\end{enumerate}
With (i)--(iii) at hand it only remains to show that we can still cover $\TT^2$ with balls centered on rationals $\underline{a}/r$ on lines $L(d\underline{c})$ such that $|r|\widehat{Q}^{-1/2}\leq |d\underline{c}|$. This is a consequence of one-dimensional Diophantine approximation, where in fact \eqref{Eq: Dirichlet} already provides an exact partition of the unit interval. \\

We will follow this blueprint closely to obtain an analogue of Theorem \ref{Th: FareyPankaj} that allows for lopsided boxes. This requires us to go through most of Vishe's steps again, since we have to modify some of the proofs to get control over the distance between the individual coordinates of rational vectors. We begin with a two-dimensional version of Dirichlet's approximation theorem with rectangles. 
\begin{lemma}
let $R_1\geq R_2 \geq 1$ be integers. Then 
\[
\TT^2=\bigcup_{\substack{|r|\leq \widehat{R}_1\widehat{R}_2\\ r\text{ monic}}}\sideset{}{'}\bigcup_{\substack{|\underline{a}|<|r|}}R(\underline{a}/r, |r|^{-1}\underline{\widehat{R}}^{-1}).
\]
\end{lemma}
\begin{proof}
For any $\underline{x}\in\TT^2$ the  rectangle $R(\underline{x},\underline{\widehat{R}}^{-1})$ has volume $(\widehat{R}_1\widehat{R}_2)^{-1}$. We can therefore write
\begin{equation}
\TT^2=\bigsqcup_{i=1}^{\widehat{R}_1\widehat{R}_2}R(\underline{x}_i,\underline{\widehat{R}}^{-1})
\end{equation}
for some $\underline{x}_i\in\TT^2$. Observe that there are $\widehat{R}_1\widehat{R}_2q$ polynomials $r\in\O$ with $|r|\leq \widehat{R}_1\widehat{R}_2$. In particular, if $\underline{x}\in\TT^2$ and $r$ runs through all $r\in\O$ with $|r|\leq \widehat{R}_1\widehat{R}_2$, then two of them, say $r_1\neq r_2$, must satisfy 
\[
\{r_i\underline{x}\}\in R(\underline{x}_j,\underline{\widehat{R}}^{-1}),
\]
for $i=1,2$ and some $1\leq j\leq \widehat{R}_1\widehat{R}_2$, where $\{\cdot\}$ denotes the fractional part. If we let $r=r_1-r_2$, then this implies 
\[
r\underline{x}-\underline{a}\in R(0, (\widehat{R}_1^{-1},\widehat{R}_2^{-2})),
\]
where $\underline{a}$ is the integer part of $(r_1-r_2)\underline{x}$. We can divide through by $(\underline{a},r)$ to ensure that $(\underline{a},r)=1$ and also multiply by a unit if necessary to guarantee that $r$ is monic. 
\end{proof}
Next we show that every rational lies on a generalised line of suitable height. This is the analogue of \cite[Lemma 3.1]{vishe2019rational}, where the difference is that we allow the extra parameter $T$.
\begin{lemma}\label{Le: LinesCover}
Let $T\geq 1$. Then for any $\underline{a}/r\in\TT^2$ with $|\underline{a}|<|r|$ and $(\underline{a},r)=1$, there exists $d\mid r$ monic and $\underline{c}\in\O^2$ primitive such that $|dc_1|\leq \widehat{T}|r|^{1/2}$, $|dc_2|< \widehat{T}^{-1}|r|^{1/2}$ and $\underline{a}/r\in L(d\underline{c})$. 
\end{lemma}
\begin{proof}
The set 
\[
\{\underline{c}\in\O^2\colon |c_1|\leq \widehat{T}|r|^{1/2}, |c_2|<\widehat{T}^{-1}|r|^{1/2}\}
\]
has cardinality strictly bigger than $|r|$. In particular, there exists two distinct vectors $\underline{c}_1, \underline{c}_2$ in this set such that $\underline{c}_1\cdot \underline{a}\equiv \underline{c}_2\cdot \underline{a} \mod r$. Let $\underline{c}'=\underline{c}_1-\underline{c}_2$. It then follows that $\underline{c}'\cdot\underline{a}=k'r$ for some $k'\in \O$. Now let $d'=\gcd(c_1,c_2)$, $d= d'/\gcd(d',k')$, $\underline{c}=\underline{c}'/d'$ and $k=k'/\gcd(d',k')$. We then have $d\underline{c}\cdot \underline{a}=kr$. Moreover, by construction $(d,k)=1$, which also implies $d\mid r$. We can further guarantee that $\underline{c}$ is primitive and $d$ monic by multiplying with a unit and changing $k$ if necessary. 
\end{proof}
For any $\underline{c}=(c_1,c_2)\in\O^2$ we let $\underline{c}^\bot = (-c_1,c_1)$. We also need the following result, which is \cite[Lemma 3.5]{vishe2019rational}, about the distribution of rational points on an individual line. 
\begin{lemma}\label{Le: StructureRatLines}
Let $d\in\O$ be monic and $\underline{c}\in\O^2$ be primitive. Then for every $\underline{a}/r\in L(\underline{c})$ there exists a unique $a\in \O$ with $|a|<|r|$, $(a,r)=1$ and a unique $\underline{d}\in \O^2$ with  $|\underline{d}|<|\underline{c}|$ such that $\underline{a}/r=\frac{a}{r}\underline{c}^\bot+\underline{d}$. Moreover, for every $\underline{a}/r\in L(d\underline{c})$ there exists a unique $\underline{a}'/(r/d)\in L(\underline{c})$ with $|\underline{a}'|<|r/d|$ and a unique $\underline{d}'\in\O^2$ with $|\underline{d}|'<|d|$ such that $\underline{a}/r=\underline{a}'/r+\underline{d}'/d$.
\end{lemma}
Next we turn to studying the distance between rational points on lines of the form $L(\underline{c})$. The following result is the analogue of \cite[Lemma 3.6]{vishe2019rational}.
\begin{lemma}\label{Le: DistanceRatL(c)}
Let $\underline{c}\in\O^2$ be primitive and $\underline{a}/r\neq \underline{a}'/r'\in L(\underline{c})$. Then 
\[
\left|\frac{a_i}{r}-\frac{a_i'}{r'}\right|\geq \frac{|c_i^\bot|}{|rr'|}\text{ for both }i=1,2 \quad\text{or}\quad \max_{i=1,2}\left\{|c_i|\left|\frac{a_i}{r}-\frac{a_i'}{r'}\right|\right\}\geq 1.
\]
The first case happens if $\underline{a}/r, \underline{a}'/r'\in L_1(\underline{c},k)$ for some $k\in \O$. Moreover, if 
$\underline{a}/r$ is an element of $L(\underline{c})\cap L_1(\underline{c},k)$ with $|\underline{c}|\leq |r|^2$, there exists $\underline{b}/r_1\in L(\underline{c})\cap L_1(\underline{c},k)$ such that $|r_1|=q|r|$ and $|a_i/r-b_i/r_1|=|c_i^\bot|/|rr_1|$.
\end{lemma}
\begin{proof}
We begin with the first part of the lemma. Since $\underline{a}/r, \underline{a}'/r'\in L(\underline{c})$, it follows from the definition of $L(\underline{c})$ that there exist $k, k'\in\O$ such that $(\underline{a}/r-\underline{a}'/r')\cdot \underline{c}=k-k'$. If $k\neq k'$, then by the ultrametric property we have $\max_{i=1,2}{|a_i/r-a_i'/r'||c_i|}\geq 1$, which is sufficient. Thanks to Lemma \ref{Le: StructureRatLines} we can also write $\underline{a}/r=\frac{a}{r}\underline{c}^\bot +\underline{d}$ and $\underline{a}'/r'=\frac{a'}{r'}\underline{c}^\bot+\underline{d}'$. If $k=k'$, then the conditions $|\underline{d}|, |\underline{d}'|<|\underline{c}|$ and the fact that $\underline{c}$ is primitive imply that $\underline{d}=\underline{d}'$. Therefore, we have $|a_i/r-a'_i/r'|=|a/r-a'/r||c_i^\bot|\geq |c_i^\bot|/|rr'|$. Furthermore, we have $k=k'$ if and only if $\underline{a}/r, \underline{a}'/r'\in L_1(\underline{c},k)$.

For the second part of the lemma, one can check that Vishe's proof of \cite[Lemma 3.6]{vishe2019rational} in fact gives control over the distance of both coordinates of $\underline{a}/r-\underline{b}/r_1$. Moreover, he requires $|\underline{c}|^2\leq |r|$, but his proof shows that $|\underline{c}|\leq|r|^2$ is in fact sufficient.

\end{proof}
We can also extend this result to arbitrary lines.
\begin{lemma}\label{Le: DistnaceRatL(dc)}
Let $d\in \O$ be monic and $\underline{c}\in\O^2$ be primitive. Then for $\underline{a}/r\neq \underline{a}'/r'\in L(d\underline{c})$, at least one of the following must hold:
\begin{enumerate}[(i)]
    \item $\left|\frac{a_i}{r}-\frac{a'_i}{r'}\right|\geq \frac{|dc_i^\bot|}{|rr'|}\text{ for both }i=1,2$,
    \item $\left|\frac{\underline{a}}{r}-\frac{\underline{a}'}{r'}\right| \geq \max\{|r|^{-1}, |r'|^{-1}\}$,
    \item $\max_{i=1,2}\left\{ |dc_i|\left |\frac{a_i}{r}-\frac{a'_i}{r'}\right|\right\}\geq 1.$
\end{enumerate}
Moreover, if $\underline{a}/r\in L_1(d\underline{c},k)$, then there exists $\underline{a}/r\neq \underline{b}/r_2\in L_1(d\underline{c},k)$ such that 
\[
{|a_i/r-b_i/r_2|\leq |dc_i^\bot|/|rr_2|}\]
for both $i=1,2$.
\end{lemma}
\begin{proof}
We begin with the first part of the statement. Recall from Lemma \ref{Le: StructureRatLines} that we can write $\underline{a}/r=\underline{a}_1/r+\underline{d}/d$ and $\underline{a}'/r'=\underline{a}_2/r'+\underline{d}'/d$ where $\underline{a}_1/(r/d), \underline{a}_2/(r'/d)\in L(\underline{c})$ and $\underline{d},\underline{d}'\in\O^2$. We thus have 
\[
\left|\frac{a_i}{r}-\frac{a'_i}{r'}\right|=\left|\frac{a_{1,i}}{r}-\frac{a_{2,i}}{r'}+\frac{d_i-d_i'}{d}\right|
\]
for $i=1,2$. If $\underline{d}\neq \underline{d}'$, then this is clearly at least $1/|d|\geq \max\{|r|^{-1},|r'|^{-1}\}$ for one of $i=1,2$, since $d\mid r,r'$ . One the other hand, if $\underline{d}=\underline{d}'$, we can use Lemma \ref{Le: DistanceRatL(c)}: In its first case we obtain $|a_{1,i}/(r/d)-a_{2,i}/(r'/d)|\geq |d^2||c_i^\bot|/|rr'|$ for both $i=1,2$, which implies \[|a_i/r-a_i'/r'|\geq |dc_i^\bot|/|rr'|\quad \text{for}\quad i=1,2,\] whereas in the second case $\max_{i=1,2}\{ |c_i||a_{1,i}/(r/d)-a_{2,i}/(r/d)|\}\geq 1$, which implies 
\[\max_{i=1,2}|dc_i||a_i/r-a'_i/r'|\geq 1.\]

For the second part of the lemma, we use Lemma \ref{Le: StructureRatLines} to write $\underline{a}/r = \underline{b}'/r+\underline{d}'/d$, where $\underline{b}'/(r/d)\in L_1(\underline{c},k)$ for some $k\in\O$.  It follows from the second part of Lemma \ref{Le: DistanceRatL(c)} that there exists $\underline{b}''/r_1\in L_1(\underline{c},k)$ with $|r_1|=q|r/d|$ and $|b'_i/(r/d)-b''_i/r_1|=|dc_i^\bot|/|rr_1|$ for $i=1,2$. Now set $\underline{b}/r_2= \underline{b}''/(r_1d)+\underline{d}'/d$. We then have $d\underline{c}\cdot \underline{b}/r_2=k + \underline{c}\cdot\underline{d}'= d\underline{c}\cdot \underline{a}/r$, so that $\underline{b}/r_2\in L_1(d\underline{c},k)$. Moreover, we also have $|b_i/r_2-a_i/r|=|b_i''/r_1-b'_i/(r/d)|/|d|=|c_i^\bot|/|rr_1|\leq |dc_i^\bot|/|rr_2|$, where we used that $|r_2|\leq |r_1d|$. 
\end{proof}
We also need the following lemma, which is the second part of \cite[Lemma 3.9]{vishe2019rational}.
\begin{lemma}\label{Le:LinesdontIntersect}
Let $\underline{a}/r\in L(d\underline{c})\cap L(d'\underline{c}')$, where $d,d'\in \O$ are monic and $\underline{c},\underline{c}'\in\O^2$ are primitive. If $|c_1c_2'|, |c_2c_1'|<|r/(dd')|$, then $d\underline{c}=d'\underline{c}'$.
\end{lemma}
Note that in \cite[Lemma 3.9]{vishe2019rational} the extra condition $|d\underline{c}|^2, |d'\underline{c}'|^2\leq |r|$ is required, but this is in fact not used in the proof.
The next lemma is concerned with the distance between rational points that lie on distinct lines.
\begin{lemma}\label{Le:DistnaceBetweenLines}
Let $\underline{a}/r\in L(d\underline{c})$ and $\underline{a}'/r'\in L(d'\underline{c}')$ with $d\underline{c}\neq d'\underline{c}'$, $|dd'c_1c_2'|<|rr'|^{1/2}$ and $|dd'c_2c_1'|<|rr'|^{1/2}$.
Then we have $|a_i/r-a'_i/r'|\geq |dc_ir'|^{-1}$ for one of $i=1,2$ and \break $|a_i/r-a'_i/r'|\geq |d'c'_ir|^{-1}$ for one of $i=1,2$.
\end{lemma}
\begin{proof}
First note that $\underline{a}/r$ and $\underline{a}'/r'$ must be distinct by Lemma \ref{Le:LinesdontIntersect} . By the second part of Lemma \ref{Le: DistnaceRatL(dc)}  there exists $\underline{b}/r_1\in L(d\underline{c})$ such that $|a_i/r-b_i/r_1|\leq |dc_i^\bot|/|rr_1|$. Let \[
C=\begin{pmatrix} \underline{a}/r - \underline{b}/r_1 \\ \underline{a}/r- \underline{a}'/r'\end{pmatrix}.\] 
Since $\underline{a}'/r'\not\in L(d\underline{c})$, it follows that $\det(C)\neq 0$. It is therefore clear that $|\det(C)|\geq |rr'r_1|$. This implies that $|a_i/r-a'_i/r'|\geq 1/|dc_ir'|$ for one of $i=1,2$. Finally we can replace the role of $r$ and $r'$ to obtain the second inequality of the statement.
\end{proof}
We now have all ingredients at hand to prove the main result of this section.
\begin{theorem}\label{Th:Farey}
Let $R_1\geq R_2\geq 1$ be integers. If we set $T=({R}_1-{R}_2)/2$, then
\begin{equation}\label{Eq: FareyRect}
    \TT^2=\bigsqcup_{\substack{|r|\leq \widehat{R}_1\widehat{R}_2\\ r\text{ monic}}}\bigsqcup_{\substack{d\mid r \text{ monic}\\
    \underline{c}\in\O^2\text{ primitive}\\ |dc_1|\leq \widehat{T}|r|^{1/2}, |dc_2|<\widehat{T}^{-1}|r|^{1/2}\\ \max\{\widehat{R}_i|dc_i^\bot|\}\geq |r|}}\sideset{}{'}\bigsqcup_{\substack{|\underline{a}|<|r|\\ \underline{a}/r\in L(d\underline{c})}}R(\underline{a}/r, |r|^{-1}\underline{\widehat{R}}^{-1}). 
\end{equation}
\end{theorem}
\begin{proof}
We first show the union on the right hand side of \eqref{Eq: FareyRect} is disjoint. Let $\underline{a}/r \neq  \underline{a}'/r'$ appear on the right hand side of \eqref{Eq: FareyRect} and suppose $|r'|\geq |r|$. We now have to distinguish a few cases. First, if $\underline{a}/r, \underline{a}'/r' \in L(d\underline{c})$ for some $L(d\underline{c})$ appearing on the right hand side of \eqref{Eq: FareyRect}, then in case (i) of Lemma \ref{Le: DistnaceRatL(dc)} we have  
\begin{align*}
    \left|\frac{a_i}{r}-\frac{a_i'}{r'}\right|  \geq \frac{|dc_i^\bot|}{|rr'|} \geq \frac{1}{\widehat{R}_i|r|}
\end{align*}
for one of $i=1,2$, where we used that $|dc_i^\bot|\geq \widehat{R}_i^{-1}|r'|$ for one of $i=1,2$. This is clearly sufficient to show $R(\underline{a}/r, |r|^{-1}\underline{\widehat{R}}^{-1})$ is disjoint from $R(\underline{a'}/r', |r'|^{-1}\underline{\widehat{R}}^{-1})$. On the other hand, case (ii) of Lemma \ref{Le: DistnaceRatL(dc)} yields  $\left |\underline{a}/r-\underline{a}'/r'\right|\geq  \max\{|r|^{-1}, |r'|^{-1}\}$. However, then the rectangles around $\underline{a}/r$ and $\underline{a}'/r'$ must be disjoint since $\widehat{R}_1\geq \widehat{R}_2\geq 1$. If case (iii) in Lemma \ref{Le: DistnaceRatL(dc)} holds for $i=1$, then $dc_1\neq 0$ and
\begin{align*}
        \left|\frac{a_1}{r}-\frac{a_1'}{r'}\right|  \geq \frac{1}{|dc_1|}\geq \frac{\widehat{R}_2|dc_1|}{\widehat{R}_1|r|}\geq \frac{1}{\widehat{R}_1|r|}.
\end{align*}
A similar calculation shows that if the inequality holds for $i=2$, then we obtain \break $|a_2/r-a_2'/r'|\geq \widehat{R}_2^{-1}|r|^{-1}$. This finishes the case $\underline{a}/r, \underline{a}'/r'\in L(d\underline{c})$. Next we are concerned about the case $\underline{a}/r\in L(d\underline{c})$, $\underline{a}'/r'\in L(d'\underline{c}')$ with $d\underline{c}\neq d'\underline{c}'$. Our constraints on $d\underline{c}, d'\underline{c}'$ guarantee that the requirements in Lemma \ref{Le:DistnaceBetweenLines} are met, and so we get $|a_i/r-a_i/r'|\geq |dc_ir|^{-1}$ for one of $i=1,2$. Note that $|dc_1|\leq (\widehat{R}_1/\widehat{R}_2)^{1/2}|r|^{1/2}\leq \widehat{R}_1$, since $|r|\leq \widehat{R}_1\widehat{R}_2$. Similarly we get $|dc_2|\leq \widehat{R}_2$, so that $|a_i/r-a_i/r'|\geq |dc_ir|^{-1}\geq \widehat{R}_i^{-1}|r|^{-1}$, which is sufficient. Finally, it remains to show that every rational $\underline{a}/r$ in the right hand side of \eqref{Eq: FareyRect} appears precisely once. This is a consequence of Lemma \ref{Le:LinesdontIntersect}. Therefore, we have established that the right hand side of \eqref{Eq: FareyRect} is a disjoint union. \\

Now we show that if $\underline{a}/r\in\TT^2$ is a rational with $|\underline{a}|<|r|\leq \widehat{R}_1\widehat{R}_2$ and $(\underline{a},r)=1$, then there exists a rational $\underline{a}'/r'$ appearing on the right hand side of \eqref{Eq: FareyRect} such that \break $R(\underline{a}/r, |r|^{-1}\underline{\widehat{R}}^{-1})\subset R(\underline{a}'/r', |r'|^{-1}\underline{\widehat{R}}^{-1})$. By Lemma \ref{Eq: Dirichlet} this is enough to show the equality of sets in \eqref{Eq: FareyRect}.  It follows from Lemma \ref{Le: LinesCover} that there exist $d\in \O$ monic and $\underline{c}\in\O^2$ primitive such that $d\mid r$, $\underline{a}/r\in L(d\underline{c})$ and $|dc_1|\leq T|r|^{1/2}$, $|dc_2|<T^{-1}|r|^{1/2}$. If $\max\{\widehat{R}_i|dc_i^\bot|\}\geq |r|$ we are done. Otherwise, let $M=\max\{\widehat{R}_i|c_i^\bot|\}$ so that $M<|r|$. Lemma \ref{Le: StructureRatLines} allows to write $\underline{a}/r = \frac{a}{r}\underline{c}^\bot +\underline{d}/d$ for some $a\in \O$ and $\underline{d}\in \O^2$, where $a/(r/d)\in L(\underline{c})$. The one dimensional Dirichlet approximation theorem implies the existence of a rational $a'/r_1$ such that $|a'|<|r_1|\leq M|d|^{-1}$, $(a',r_1)=1$ and $|a/(r/d)-a'/r_1|<|r_1|^{-1}|d|M^{-1}$. Now set $\underline{a}'/r'= \frac{a'}{r_1d}\underline{c}^\bot +\underline{d}/d$. We then have $d\underline{c}\cdot \underline{a}'/r'=\underline{c}\cdot \underline{d}=d\underline{c}\cdot \underline{a}/r$, which implies $\underline{a}'/r'\in L(d\underline{c})$ and $d\mid r'$. Moreover, by construction we have $|r'|\leq |dr_1|\leq M$. We also have $|a_i/r-a'_i/r'| = |d|^{-1}|a_i/(r/d)-a'/r_1|< |r_1|^{-1}M^{-1}\leq  |r'|^{-1}\widehat{R}_i^{-1}$ for both $i=1,2$. This completes the proof of Theorem \ref{Th:Farey}.
\end{proof}
\begin{remark}
Note that in the particular case $R_1=R_2$ we recover Theorem \ref{Th: FareyPankaj} with $Q=2R_1$ from Theorem \ref{Th:Farey}.
\end{remark}
The following corollary will be useful when evaluating the main contribution to our asymptotic formula. 
\begin{corollary}\label{Cor: FaryMajor}
    Let $R_1\geq R_2\geq 1$ be integers. Then 
    \begin{align*}
   \bigsqcup_{\substack{|r|\leq \widehat{R}_2\\ r\text{ monic}}}\bigsqcup_{\substack{d\mid r \text{ monic}\\
    \underline{c}\in\O^2\text{ primitive}\\ |dc_1|\leq \widehat{T}|r|^{1/2} \\ |dc_2|<\widehat{T}^{-1}|r|^{1/2}}}\sideset{}{'}\bigsqcup_{\substack{|\underline{a}|<|r|\\ \underline{a}/r\in L(d\underline{c})}}R(\underline{a}/r, |r|^{-1}\underline{\widehat{R}}^{-1})
    = \bigsqcup_{\substack{|r|\leq \widehat{R}_2 \\ r \text{ monic}}}\sideset{}{'}\bigsqcup_{\substack{|\underline{a}|<|r|}}R(\underline{a}/r, |r|^{-1}\underline{\widehat{R}}^{-1})   
    \end{align*}   
\end{corollary}
\begin{proof}
    Clearly the left hand side of the claimed equality is contained in the right hand side. Moreover, by Lemma \ref{Le: LinesCover} the right hand side is contained in the left hand side. Theorem \ref{Th:Farey} implies that the left hand side is disjoint, since the condition $\max\{\widehat{R}_i|dc_i^\bot|\}\geq |r|$ is vacuously true for $|r|\leq \widehat{R}_2$. It remains to prove that the right hand side is disjoint. Suppose that $\underline{\alpha}\in R(\underline{a}_1/r_1, |r_1|^{-1}\underline{\widehat{R}}^{-1})\cap R(\underline{a}_2/r_2, |r_2|^{-1}\underline{\widehat{R}}^{-1})$ with $\underline{a}_1/r_1\neq \underline{a}_2/r_2$. We then have 
    \[
    \frac{1}{|r_1r_2|}\leq \left|\frac{\underline{a}_1}{r_1}-\frac{\underline{a}_2}{r_2}\right|= \left|\left(\frac{\underline{a}_1}{r_1}-\underline{\alpha}\right)+\left(\underline{\alpha}-\frac{\underline{a}_2}{r_2}\right)\right|  <\max\left\{\frac{1}{|r_1|\widehat{R}_2},\frac{1}{|r_2|\widehat{R}_2}\right\},
    \]
    which is impossible since $|r_1|,|r_2|\leq \widehat{R}_2$.
\end{proof}
\section{Geometry}\label{Se: geometry}
If the complete intersection $X\subset \PP^{n-1}$ is defined by a cubic form $F_1$ and a quadratic form $F_2$, then it is also defined by $F_1+LF_2$ and $F_2$ for any linear form $L\in\O[x_1,\dots,x_n]$. In particular, given the degree of freedom we have, it is reasonable to expect that we can define $X$ as the intersection of a non-singular cubic hypersurface and a quadratic hypersurface. This is indeed the case as demonstrated in Lemmas 3.2--3.3 of \cite{browning2015rational}, whose proof we adjust to cope with positive characteristic.
\begin{lemma}\label{Le: F1nonsingular}
Let $X\subset \PP^{n-1}$ be a non-singular complete intersection of a cubic and a quadratic hypersurface over $K$. Then $X=V(F_1,F_2)$, where $F_1\in\O[x_1,\dots,x_n]$ is a non-singular cubic form and $F_2$ is a quadratic form of rank at least $n-1$.
\end{lemma}
\begin{proof}
    Suppose $X=V(G_1,G_2)$, where $G_1$ is a cubic and $G_2$ a quadratic form respectively. For $U=\PP^{n-1}\setminus V(G_2)$ define the morphism 
    \[
    \varphi\colon U \to \PP^n,\quad (x_1\colon \cdots \colon x_n)\mapsto (G_1(\bm{x})\colon x_1G_2(\bm{x})\cdots \colon x_nG_2(\bm{x})).
    \]
    Assume for a moment that there exists a hyperplane $H\subset \PP^n$ defined over $K$ such that $\varphi^{-1}(H)$ is smooth. This means that there exist $\lambda_0,\dots, \lambda_n\in K$ such that 
    \[
    F_1= \lambda_0G_1+\lambda_1x_1G_2+\cdots +\lambda_nx_nG_n
    \]
    satisfies $U\cap \{F_1=0\}$ is smooth. However, $U\cap \{F_1=0\}=\{F_1=0\}\setminus\{G_2=0\}$, from which it follows that $F_1$ is non-singular since $X$ is non-singular. 

    To prove the existence of the claimed $\lambda_i$'s,  Browning--Dietmann--Heath-Brown appeal to Bertini's Theorem, which does not hold in general in positive characteristic. However, it follows from work of Spreafico \cite[Corollary 4.3]{luisa1998axiomatic} that the fiber above a general hyperplane $H\subset \PP^n$ is smooth provided the induced extension of residues fields $\kappa(x)/\kappa(\varphi(x))$ is separable for any $x\in U$. Let $Y=V(F_1-x_0F_2)\subset \PP^n$. Then a little thought reveals that $\varphi$ factors into 
    \[
    U\rightarrow Y\setminus V(F_2)\rightarrow Y \rightarrow \PP^n,
    \]
where the first arrow is an isomorphism, the second an open embedding and the third a closed immersion. Both open embeddings and closed immersions are unramified, as is the composition of unramified morphisms. It follows that $\varphi$ is unramified and hence all the residue field extensions are separable. 

    It remains to show that $G_2$ has rank at least $n-1$. Aiming at a contradiction, suppose the opposite holds. This implies that $G_2$ is singular along a line. However, this line will intersect $\{F_1=0\}$ in a point that will then be a singular point of $X$, which is impossible.
\end{proof}
To deal with the exponential integrals appearing in our work, we will have to concentrate our weight function near a point such that linear combinations of its Hessian associated with $F_1$ and the matrix underlying the quadratic form $F_2$ always have large rank. This is only possible when $\cha(K)>3$ and so we assume this holds for the rest of our work. Let us now fix a symmetric matrix $M\in \O^{n\times n}$ such that $F_2(\bm{x})=\bm{x}^tM\bm{x}$. Moreover, for $\bm{x}\in K_\infty^n$ we shall denote by $H(\bm{x})=(\frac{\partial F_1}{\partial x_i\partial x_j})_{1\leq i,j\leq n}$ the Hessian of $F_1$ evaluated at $\bm{x}$.
\begin{lemma}\label{Le: DimRankSmall}
For $1\leq k \leq n-1$, let $V_k$ be the Zariski closure of 
\[
V'_k\coloneqq \{\bm{x} \in \A^n\colon \text{rk}(t_1H(\bm{x})+t_2M)\leq k \text{ for some }(t_1,t_2)\in\A^2\setminus\{\bm{0}\}\}
\]
inside $\A^n$. Then $\dim V_k \leq k+1$.
\end{lemma}
\begin{proof}
For $1\leq k \leq n-1$ consider the incidence correspondence
\begin{align*}
    I&\coloneqq \{(\bm{x},\bm{y})\in \A^n\times \A^n\colon \text{rk}(H(\bm{x})\bm{y}, M\bm{y})\leq 1\}.
\end{align*}
Let $V$ be an irreducible component of $V_k$. Since $V\times\{\bm{0}\}\subset I$ and $V\times \{\bm{0}\}$ is irreducible, there exists an irreducible component $W$ of $I$ containing $V\times\{\bm{0}\}$. Then the projection of $I$ onto the first factor restricts to  a surjective morphism $\psi\colon W\to V$ of irreducible varieties. We can therefore apply Chevalley's theorem \cite[Proposition 14.109]{gortz2020algebraic} to deduce the existence of an open dense subset $U\subset V$ such that $\dim \psi^{-1}(x)=\dim W- \dim V$ for all $x\in V$. Note that since $V'_k$ is dense in $V_k$ and $V$ is an irreducible component of $V_k$, we must have $U\cap V_k \neq \emptyset$. In addition, by the definition of $V_k'$ we have $\dim \psi^{-1}(x)\geq n-k$ for all $x\in V_k'$, so that we must also have $\dim \psi^{-1}(x)\geq n-k$ for all $x\in U$. 

Our next task is to bound the dimension of $I$. For this, let $\Delta\colon \A^n\to\A^{n}\times \A^n$ the diagonal embedding, that is $\bm{x}\mapsto (\bm{x},\bm{x})$, and let $S=\Delta(\A^n)\cap I$. If $\bm{x}\in S\cap V(F_1)$, then there exists $(t_1,t_2)\in\A^2\setminus\{\bm{0}\}$ such that $t_1\nabla F_1(\bm{x})+t_2\nabla F_2(\bm{x})=0$. Suppose first that $t_2=0$. It follows immediately from Euler's identity that $F_1(\bm{x})=0$ and hence $\bm{x}=\bm{0}$, because $F_1$ is non-singular by assumption. On the other hand, if $t_2\neq 0$, then after taking the inner product with $\bm{x}$, we get $F_2(\bm{x})=0$, so that $\bm{x}$ is a singular point on the affine cone of the non-singular complete intersection of $F_1$ and $F_2$, which again implies $\bm{x}=\bm{0}$. Altogether we obtain $0\geq \dim S\cap V(F_1)\geq \dim S+ \dim V(F_1)-n$ and hence $\dim S\leq 1$. 

Having established an upper bound for $\dim S$, we are now in a position to get control over $\dim I$. From what we have just shown, it follows that 
\[
1\geq \dim I\cap \Delta(\A^n)\geq \dim I + \dim \A^n -2n = \dim I -n\]
and thus $\dim I\leq n+1$, from which we immediately deduce $\dim W\leq n+1$. Combining this with the information about the dimension of the fibers of $\psi$, we obtain for any $x\in U$ the inequality $n-k \leq \dim \psi^{-1}(x)=\dim W - \dim V_k$ and therefore $\dim V_k\leq k+1$ as claimed.
\end{proof}
\begin{corollary}\label{Cor: NicePointExists}
There exists $\bm{x}_0\in K_\infty^n$ such that $F_1(\bm{x}_0)=F_2(\bm{x}_0)=0$, $\text{rk}(H(\bm{x}_0))\geq n-2$ and $\text{rk}(t_1H(\bm{x})+t_2 M)\geq n-3$ for all $(t_1,t_2)\in K_\infty^2\setminus\{\bm{0}\}$.
\end{corollary}
\begin{proof}
Let $X'\subset \A^n$ be the affine cone of the non-singular complete intersection \break $X=V(F_1,F_2)\subset \PP^{n-1}$. It follows from Lemma \ref{Le: DimRankSmall} that $V_{n-4}$ is a Zariski closed subset in $\A^n$ of dimension at most $n-3$. Similarly, Heath-Brown 
\cite[Lemma 2]{heath1983cubic} has shown that
\[W=\{\bm{x}\in\A^n\colon \rk(H(\bm{x}))\leq n-3\}\]
has dimension at most $n-3$. Since $X$ is non-singular, it follows that $X'(K_\infty)$ is Zariski dense in $X'$. In particular, the fact that $\dim X'=n-2$ implies that $X'(K_\infty)\setminus(V_{n-3}\cup W)$ is non-empty and any point contained therein satisfies the conditions required in the statement of the corollary.
\end{proof}
We also need strong upper bounds for the number of integral points on an affine hypersurface, which are a special case of \cite[Theorem 1.10]{paredessasyk2021} in the $\FF_q[t]$ setting.
\begin{theorem}\label{Th: DimensionGrowth}
    Let $G\in \O[x_1,\dots, x_n]$ be a polynomial of degree $d\geq 5$ whose degree $d$ part is absolutely irreducible and let $B\geq 1$. Then there exists a constant $C>0$ depending only on $d$, $n$ and $q$ such that
    \[
    \#\{\bm{x}\in \O^n \colon |\bm{x}| < \widehat{B}, G(\bm{x})=0 \} \leq C \widehat{B}^{n-2}.
    \]
\end{theorem}
\section{Activation of the circle method}
In this section we collect the remaining facts needed to get the circle method started. Recall from Lemma \ref{Le: F1nonsingular} that we can assume $X=V(F_1,F_2)$ with $F_1\in\O[x_1,\dots, x_n]$ a non-singular cubic form and $F_2\in\O[x_1,\dots, x_n]$ a quadratic form of rank at least $n-1$. We shall fix such a choice of $F_1$ and $F_2$ once and for all and write $\underline{F}=(F_1,F_2)$. Moreover, we assume $M\in\text{Mat}_{n\times n}(\O)$ is a symmetric matrix such that $F_2(\bm{x})=\bm{x}^tM\bm{x}$. 
In what follows, for $F\in K_\infty[x_1,\dots, x_n]$ we refer to the maximum of the absolute values of the coefficients of $F$ as the height of $F$ and denote it by $H_F$. We extend this definition to pairs of polynomials by $H_{\underline{G}}\coloneqq \max\{H_{G_1},H_{G_2}\}$.

Corollary \ref{Cor: NicePointExists} implies that there exists $\bm{x}_0\in K_\infty^n$ such that
\begin{align}
\begin{split}\label{Eq: PropertiesNicePoint}
    &F_1(\bm{x}_0)=F_2(\bm{x}_0)=0,\\ 
    &\text{rk}(H(\bm{x}_0))\geq n-2\text{ and}\\ &\text{rk}(\gamma_1H(\bm{x}_0)+\gamma_2 M)\geq n-3\text{ for all }(\gamma_1,\gamma_2)\in K_\infty^2\setminus\{\bm{0}\},
    \end{split}
\end{align}
where $H(\bm{x}_0)$ denotes the Hessian of the cubic form $F_1$ evaluated at $\bm{x}_0$. These properties are clearly invariant under scaling and so we may additionally assume $|\bm{x}_0|<H_{\underline{F}}^{-1}$. We will then work with the weight function $w\colon K_\infty^n\to \RR_{\geq 0}$ defined by 
\begin{equation}\label{Eq: DefWeight}
w(\bm{x})\coloneqq \chi_{\TT}(t^L(\bm{x}-\bm{x}_0)),
\end{equation}
where $L$ is a large but fixed integer, whose exact value will be determined throughout our work. The non-archimedean nature of $K_\infty$ ensures that $\text{rk}(H(\bm{x}))\geq n-2$ and $|\bm{x}|< 1/H_{\underline{F}}$ whenever $w(\bm{x})\neq 0$ and $L$ is sufficiently large. Moreover, we have seen in the proof of Corollary \ref{Cor: NicePointExists} that the set of points $\bm{x}\in K_\infty^n$ satisfying $\rk(\gamma_1H(\bm{x})+\gamma_2M)\geq n-3$ for all $(\gamma_1,\gamma_2)\in K_\infty^2\setminus\{\bm{0}\}$ is Zariski dense in $K_\infty^n$. In particular, if $L$ is large enough, we can guarantee that any $\bm{x}\in \supp(w)$ satisfies the third property in \eqref{Eq: PropertiesNicePoint}. Let us now fix $\bm{b}\in\O^n$ and $N\in\O$ such that $N\mid F_1(\bm{b}), F_2(\bm{b})$. The counting function we consider is now given by 
\[
N(P)\coloneqq \sum_{\substack{\bm{x}\in\O^n \\F_1(\bm{x})=F_2(\bm{x})=0 \\ \bm{x}\equiv \bm{b}\:(N)}}w\left(\frac{\bm{x}}{t^P}\right).
\]

Any $\alpha\in K_\infty$ can be written as a Laurent series $\alpha =\sum \alpha_i t^i$, where $\alpha_i\neq 0$ for only finitely many $i>0$. We can now define a character 
\[
\psi\colon K_\infty \to \CC^\times, \quad \alpha \mapsto e_p\left(\Tr_{\FF_q/\FF_p}(\alpha_{-1})\right),
\]
where as usual $e_p(\cdot)=\exp(2\pi i \cdot /p)$. The starting point for the circle method is the orthogonality relation 
\begin{equation}\label{Eq: OrthoChars}
    \int_\TT \psi(\alpha x)\dd\alpha = \begin{cases} 1 & \text{if }x=0,\\ 0 &\text{else,}\end{cases}
\end{equation}
for any $x\in \O$ as proved in \cite[Corollary 5.6]{Browning2021}. Combined with Theorem \ref{Th:Farey} this immediately implies
\begin{equation}\label{Eq: N(P)Step1}
    N(P)= \sum_{\substack{|r|\leq \widehat{R}_1\widehat{R}_2\\ r\text{ monic}}}\sum_{\substack{d\mid r \text{ monic}\\
    \underline{c}\in\O^2\text{ primitive}\\ |dc_1|\leq \widehat{T}|r|^{1/2}, |dc_2|<\widehat{T}^{-1}|r|^{1/2}\\ \max\{\widehat{R}_i|dc_i^\bot|\}\geq |r|}}\sideset{}{'}\sum_{\substack{|\underline{a}|<|r|\\ \underline{a}/r\in L(d\underline{c})}}\int_{|\theta_1|<|r|^{-1}\widehat{R}_1^{-1}}\int_{|\theta_2|<|r|^{-1}\widehat{R}_2^{-1}}S\left(\frac{\underline{a}}{r}+\underline{\theta}\right)\dd\theta_2\dd\theta_1,
\end{equation}
where 
\begin{equation}\label{Def: S(alpha)}
S(\underline{\alpha})\coloneqq \sum_{\substack{\bm{x}\in\O^n \\ \bm{x}\equiv \bm{b}\:(N)}}\psi(\alpha_1F_1(\bm{x})+\alpha_2F_2(\bm{x}))w(\bm{x}/t^P),
\end{equation}
for $\underline{\alpha}\in \TT^2$ and $\sum'$ indicates the condition $(\underline{a},r)=1$. After splitting $\bm{x}$ into residue classes modulo $r_N$, where $r_N=rN/(r,N)$, it is a standard argument, see \cite[Lemma 4.4]{browning2015rational} for example, to use Poisson summation to evaluate $S(\underline{\theta}+\underline{a}/r)$ and transform \eqref{Eq: N(P)Step1} into 
\begin{equation}\label{Eq: N(P)Step2Delta}
    N(P)=\widehat{P}^n\sum_{\substack{|r|\leq \widehat{R}_1\widehat{R}_2\\ r\text{ monic}}}|r_N|^{-n}\sum_{\substack{d\mid r \text{ monic}\\
    \underline{c}\in\O^2\text{ primitive}\\ |dc_1|\leq \widehat{T}|r|^{1/2}, |dc_2|<\widehat{T}^{-1}|r|^{1/2}\\ \max\{\widehat{R}_i|dc_i^\bot|\}\geq |r|}}\int_{D(|r|\underline{\widehat{R}})}\sum_{\bm{v}\in \O^n} S_{d\underline{c},r,\bm{b},N}(\bm{v})I_{r_N}(\underline{\theta}, \bm{v})\dd\underline{\theta},
\end{equation}
where $D(|r|\underline{\widehat{R}})=\{\underline{\theta}\in\TT^2\colon |r\theta_i|<\widehat{R}_i^{-1}\text{ for }i=1,2\}$, 
\begin{equation}\label{Defi: SdcrbN}
    S_{d\underline{c},r,\bm{b},N}(\bm{v})\coloneqq \sideset{}{'}\sum_{\underline{a}/r\in L(d\underline{c})}\sum_{\substack{|\bm{x}|<|r_M| \\ \bm{x}\equiv \bm{b}\:(N)}}\psi\left(\frac{a_1F_1(\bm{x})+a_2F_2(\bm{x})}{r}\right)\psi\left(\frac{-\bm{v}\cdot\bm{x}}{r_N}\right),
\end{equation}
 and 
\begin{equation}
    I_s(\underline{\theta},\bm{v})\coloneqq \int_{K_\infty^n}w(\bm{x})\psi\left(t^{3P}\theta_1 F_1(\bm{x})+t^{2P}\theta_2F_2(\bm{x})+\frac{t^P \bm{v}\cdot \bm{x}}{s}\right)\dd\bm{x}
\end{equation}
for $s\in \O\setminus\{0\}$.

In our work we will assume throughout that $R_1$ and $R_2$ are chosen in such a way that 
\begin{equation}\label{Eq: ParameterSize}
     \widehat{R}_1 \asymp \widehat{P}^{4/3}\quad\text{and}\quad  \widehat{R}_2\asymp\widehat{P}^{1/3},
\end{equation}
so that $\widehat{T}\asymp \widehat{P}^{1/2}$. This ensures that $\text{meas}(D(|r|\widehat{\underline{R}}))\asymp \widehat{P}^{-5}$ and $|I_r(\underline{\theta},\bm{v})|\asymp 1$ when ${|r|=\widehat{R}_1\widehat{R}_2}$. Let us now separate the terms from \eqref{Eq: N(P)Step2Delta} that will go into the error term. For this, we write 
\begin{equation}\label{Eq: Decomp.N(P)1}
N(P)= M(P) + E_1(P)+E_2(P),
\end{equation}
where $E_2(P)$ consists of the contribution in \eqref{Eq: N(P)Step1} for which
\begin{enumerate}
    \item\label{Cond: thetasmall} $\bm{v}\neq \bm{0}$ with $\lvert \theta_1 \rvert < \widehat{P}^{-9}\widehat{R}_2$ or $|\theta_2| <\widehat{P}^{-9}\widehat{R}_1$ or 
    \item\label{Cond: Thetalarge.rsmall} $c_2=0$ with $|\theta_1|> |r|^{-1}\widehat{P}^{-\delta}$, where $\delta =8(n-16)/(3n-24)$, and $|r|\leq \widehat{P}^{1-\eta}$ with $\eta=2/n$,
\end{enumerate}
holds. Observe that the set of $\underline{\theta}'$ for which \eqref{Cond: Thetalarge.rsmall} holds is non-empty when $n>24$, because then $\delta >4/3$. The terms $M(P)$ and $E_1(P)$ comprise the contribution from all $r$'s and $\underline{\theta}$'s for which neither \eqref{Cond: thetasmall} nor \eqref{Cond: Thetalarge.rsmall} holds with the additional constraint that $\bm{v}=\bm{0}$ for $M(P)$ and $\bm{v}\neq \bm{0}$ for $E_1(P)$ respectively.\\

We begin with estimating the contribution to $E_2(P)$ defined by \eqref{Cond: thetasmall}. Note that for any $r\in \O$ the measure of $\underline{\theta}\in \TT^2$ for which \eqref{Cond: thetasmall} holds is $O(\widehat{P}^{-9}|r|^{-1})$. Estimating trivially $S(\underline{a}/r+\underline{\theta})\leq \widehat{P}^n$ and using Lemma \ref{Le: L(dc)RatPoints} to deduce that the number of $\underline{a}$ with $|\underline{a}|<|r|$ such that $\underline{a}/r\in L(d\underline{c})$ is at most $|dr|$, we see that the contribution from \eqref{Cond: thetasmall} in \eqref{Eq: N(P)Step1} is 
\begin{align*}
\ll \widehat{P}^{n-9}\sum_{|r|\leq \widehat{R}_1\widehat{R}_2} \sum_{\substack{|dc_1|\leq \widehat{T}|r|^{1/2}\\ |dc_2|<\widehat{T}^{-1}|r|^{1/2} \\ d\mid r}} |d|\ll \widehat{P}^{n-9}\sum_{|r|\leq \widehat{R}_1\widehat{R}_2} |r|^{1+\varepsilon} \ll \widehat{P}^{n-17/3+\varepsilon},
\end{align*}
upon recalling \eqref{Eq: ParameterSize} for our choice of $R_1$ and $R_2$. \\

The reason for separating the contribution coming from \eqref{Cond: Thetalarge.rsmall} is that the integral estimates we provide in Section \ref{Sec: ExpIntegrals} are insufficient when $|r|$ is small and $|\underline{\theta}|$ is large. We eliminate this shortfall by dealing with this contribution in a manner akin to the treatment of the minor arcs in a classical application of the circle method. To begin with, let us fix the absolute values of $r$ and $\theta_1$ in the definition of $E_2(P)$ to be $|r|=\widehat{Y}$ and $|\theta_1|=\widehat{\Theta}_1$ with $-Y -\delta P \leq \Theta_1 \leq - Y - 4P/3$. The main tool to deal with this contribution is Weyl's inequality, whose function field analogue is provided by Lemma 4.3.6 in Lee's PhD thesis \cite{lee2011birch}.
\begin{lemma}\label{Le: WeylInequality}
    Let $\underline{\alpha}\in \TT^2$ and $a_1/r \in K\cap \TT$ be such that $(a_1,r)=1$ and $\alpha_1=a_1/r+\theta_1$. Then for $S(\underline{\alpha})$ given by \eqref{Def: S(alpha)} we have 
    \[   S(\underline{\alpha})\ll_{\bm{b}, N, F_1}\widehat{P}^{n+\varepsilon}\left(\frac{\widehat{P}+|r|+\widehat{P}^3|r_1\theta_1|}{\widehat{P}^3}+\frac{1}{|r_1|+\widehat{P}^3|r_1\theta_1|}\right)^{n/8}.
    \]
\end{lemma}
\begin{remark}
    Lee states Lemma 4.3.6 without the appearance of the quadratic form that features in the definition of $S(\underline{\alpha})$. However, the proof of the lemma proceeds via estimating the quantity $|S(\underline{\alpha})|^4$, which only requires considering the bilinear forms associated to the leading cubic form. It is easy to see that the quadratic form disappears during this process, so that Lee's Lemma 4.3.6 continues to hold in our situation.
\end{remark}
We now wish to apply Lemma \ref{Le: WeylInequality}. The problem is that if  $\underline{a}/r\in L(d\underline{c})$, we do not necessarily have $(a_1, r)=1$. However, recall from 
\eqref{Eq: GenLine} that  $\underline{a}/r\in L(d\underline{c})$ if and only if $d\underline{a}\cdot\underline{c}=kr$ for some $k\in \O$ with $(k,d)=1$. In particular, if $c_2=0$, then $c_1=1$ and so $\underline{a}/r\in L(d(1,0))$ if and only if $a_1/r=k/d$ with $(k,d)=1$. It is now easily checked that the constraints coming from \eqref{Cond: Thetalarge.rsmall} together with Lemma \ref{Le: WeylInequality} yield
\[
S(\underline{a}/r+\underline{\theta}) \ll \widehat{P}^{n+\varepsilon}(\widehat{P}^{-2}+|d\theta_1|+\widehat{P}^{-3}|d\theta_1|^{-1})^{n/8}.
\]
Since $c_1\neq 0$, the condition $\max\{\widehat{R}_i|dc_i^\bot|\}\geq |r|$ implies $|d|\gg |r|\widehat{P}^{-1/3}$. We deduce that the contribution from 
 $|r|=\widehat{Y}$ with $\widehat{Y}\leq \widehat{P}^{1-\eta}$ and $|\theta_1|=\widehat{\Theta}_1$ to $E_2(P)$ via \eqref{Eq: N(P)Step1} is  
 \begin{align*}
 &\ll \sum_{|r|=\widehat{Y}}\sum_{\substack{|d|\leq \widehat{Y}\\|d|\gg \widehat{Y}\widehat{P}^{-1/3} \\ d\mid r}} \sideset{}{'}\sum_{\substack{|\underline{a}|<|r| \\ \underline{a}/r\in L(d(1,0))}}\int_{|\theta_1|=\widehat{\Theta}_1}\int_{|\theta_2|\ll \widehat{P}^{-1/3}\widehat{Y}^{-1}}S(\underline{a}/r+\underline{\theta})\dd\underline{\theta}\\     &\ll\widehat{P}^{n-1/3+\varepsilon}\widehat{\Theta}_1\sum_{|r|=\widehat{Y}}\sum_{\substack{|d|\leq \widehat{Y}\\|d|\gg \widehat{Y}\widehat{P}^{-1/3} \\ d\mid r}} |d|\left(\widehat{P}^{-2}+|d|\widehat{\Theta}_1+\widehat{P}^{-3}|d|^{-1}\widehat{\Theta}_1^{-1}\right)^{n/8} \\
     & \ll \widehat{P}^{3n/4-5/3+\varepsilon}\widehat{Y}+\widehat{P}^{n-1/3+\varepsilon}\widehat{Y}^{2+n/8}\widehat{\Theta}_1^{1+n/8}+\widehat{P}^{2n/3-2/3+\varepsilon}\widehat{Y}(\widehat{\Theta}_1\widehat{Y})^{1-n/8} \\
     &\ll \widehat{P}^{3n/4-2/3-\eta+\varepsilon}+\widehat{P}^{5n/6-2/3-\eta +\varepsilon}+\widehat{P}^{2n/3-2/3-\delta(1-n/8) +\varepsilon }\widehat{Y}
 \end{align*}
where we used again Lemma \ref{Le: L(dc)RatPoints} to bound the number of $\underline{a}$'s, that $\widehat{\Theta}_1\ll \widehat{Y}^{-1}\widehat{P}^{-4/3}$ and \eqref{Cond: Thetalarge.rsmall} to estimate $\widehat{Y}$ and $\widehat{\Theta}_1$. We have $3n/4-2/3<n-5$ for $n\geq 18$, so that he first term is sufficiently small. Moreover, $5n/6-2/3-\eta<n-5$ as soon as $n\geq 26$, which is also satisfactory. Lastly, the third term above is
\[
\widehat{P}^{2n/3 -2/3 -\delta(1-n/8)+\varepsilon}\widehat{Y}\ll \widehat{P}^{2n/3+1/3-\eta -\delta(1-n/8)+\varepsilon}=\widehat{P}^{n-5-\eta+\varepsilon},
\]
where of course $\delta$ was chosen in such a way as to simplify the exponent above. Therefore, this contribution is also satisfactory. Since there are $O(\widehat{P}^\varepsilon)$ choices for $Y$ and $\Theta_1$, we have thus shown that 
\begin{equation}\label{Eq: UpperBound.E2(P)}
E_2(P)\ll \widehat{P}^{n-5-\kappa}
\end{equation}
for some $\kappa>0$ if $n\geq 26$. \\
The goal for the remainder of this work is to establish the following result. 
\begin{prop}\label{Prop: TheProp}
    If $n\geq 26$ and $\cha(K)>3$, then
    \[
    N(P)=c \widehat{P}^{n-5}+O\left(\widehat{P}^{n-5-\delta'}\right)
    \]
       for some $\delta'>0$, where $c>0$ if for every prime $\varpi$ there exists $\bm{x}_\varpi \in \O^n_\varpi$ such that \break $F_1(\bm{x}_\varpi)=F_2(\bm{x}_\varpi)=0$ and $|\bm{b}-\bm{x}_\varpi|_\varpi < |N|_\varpi$ and where the implied constant depends on $F_1$, $F_2$,  $\bm{b}$ and $N$.
\end{prop}
Once we have established Proposition \ref{Prop: TheProp}, there is no difficulty in deducing the weak approximation property for $X$. The exact details do not merit repetition here and can for example be found in Section 7.1 of \cite{CubicHypersurfacesBV} in the case of cubic hypersurfaces. Therefore, Theorems \ref{Th: TheTheorem} and \ref{Th: WA} are a consequence of Proposition \ref{Prop: TheProp}. In light of \eqref{Eq: Decomp.N(P)1} and \eqref{Eq: UpperBound.E2(P)} it will be enough to show that 
\[
M(P)=c\widehat{P}^{n-5}+O\left(\widehat{P}^{n-5-\kappa'}\right)\quad\text{and}\quad E_1(P)\ll \widehat{P}^{n-5-\kappa''},
\]
for some $\kappa', \kappa''>0$, where $c$ satisfies the properties claimed in Proposition \ref{Prop: TheProp}. This goal will ultimately be achieved in Section \ref{Sec: Circle} and requires a thorough analysis of the exponential sums and oscillatory integrals that appear in \eqref{Eq: N(P)Step2Delta}. We carry out this investigation in the subsequent three sections.
\section{Exponential integrals}\label{Sec: ExpIntegrals}
To get control over $I_s(\underline{\theta},\bm{v})$, we consider for $\bm{w}\in K_\infty^n$, $\underline{\gamma}\in K_\infty^2$ and $G_1,G_2 \in K_\infty[x_1,\dots, x_n]$ the following oscillatory integral
\[
J_{\underline{G}}(\underline{\gamma}, \bm{w})\coloneqq \int_{\TT^n}\psi\left( \underline{\gamma}\cdot \underline{G}(\bm{x})+\bm{w}\cdot\bm{x}\right)\dd \bm{x},
\]
where we henceforth adopt the notation $\underline{\gamma}\cdot \underline{G}(\bm{x})= \gamma_1G_1(\bm{x})+\gamma_2G_2(\bm{x})$.  The main ingredients to deal with the exponential integrals appearing in our work are \cite[Lemma 2.1--2.2]{vishe2019rational}, which we recall here for our convenience.
\begin{lemma}\label{Le: IntVanishes}
We have $J_{\underline{G}}(\underline{\gamma},\bm{w})=0$ if $|\bm{w}|>\max\{1, |\gamma_1|H_{G_1}, |\gamma_2|H_{G_2}\}$.
\end{lemma}
\begin{lemma}\label{Le: IntEstimate}
Let $\Omega=\{\bm{x}\in \TT^n\colon |\gamma_1\nabla G_1(\bm{x})+\gamma_2\nabla G_2(\bm{x})+\bm{w}|\leq H_{\underline{G}}\max \{1, |\underline{\gamma}|^{1/2}\}\}$. Then 
\[
J_{\underline{G}}(\underline{\gamma},\bm{w})=\int_{\Omega}\psi\left( \underline{\gamma}\cdot \underline{G}(\bm{x})+\bm{w}\cdot\bm{x}\right)\dd \bm{x}.
\]
\end{lemma}
Moreover, our analysis of the singular integral appearing in the main term of the asymptotic formula for $N(P)$ makes use of Lemma 5.5 in \cite{Browning2021}.
\begin{lemma}\label{Le: OrthoInt}
    Let $x\in \TT$ and $Z\geq 1$. Then 
    \[
    \int_{|\theta|<\widehat{Z}^{-1}}\psi(x\theta)\dd\theta =\begin{cases} \widehat{Z}^{-1} &\text{if }|x|<\widehat{Z}\\ 0 &\text{else.}\end{cases}
    \]
\end{lemma}
 To begin our treatment of $I_s(\underline{\theta},\bm{v})$, note that by the definition of $w$ in \eqref{Eq: DefWeight} we have
\begin{align}
\begin{split}\label{Eq: RelateIntegraltoJ}
I_s(\underline{\theta},\bm{v})& = \widehat{L}^{-n}\psi\left(\frac{t^p\bm{v}\cdot \bm{x}_0}{s}\right)\int_{\TT^n}\psi\left(t^{3P}\theta_1 G_1(\bm{y})+t^{2P}\theta_2G_2(\bm{y})+\frac{t^p\bm{v}\cdot\bm{y}}{t^Ls}\right)\dd\underline{\theta}\\
&=\widehat{L}^{-n}\psi\left(\frac{t^p\bm{v}\cdot \bm{x}_0}{s}\right) J_{\underline{G}}(t^{3P}\theta_1,t^{2P}\theta_2, (t^P\bm{v}t^{-L}/s)),
\end{split}
\end{align}
where $G_i(\bm{y})=F_i(\bm{x}_0+t^{-L}\bm{y})$ for $i=1,2$ and we applied the change of variables \break $\bm{y}=t^L(\bm{x}-\bm{x}_0)$. It is clear that $G_i$ is a polynomial with coefficients in $K_\infty$ and $H_{\underline{G}}\leq H_{\underline{F}}$. Therefore, it follows from Lemma \ref{Le: IntVanishes} that 
\begin{equation}\label{Eq: TruncateInt}
I_s(\underline{\theta},\bm{v})=0 \quad\text{if}\quad |\bm{v}|>\widehat{L}H_{\underline{F}}|s|\frac{\max\{1, \widehat{P}^3|\theta_1|, \widehat{P}^2|\theta_2|\}}{\widehat{P}}.\end{equation}
Before we can  derive upper bounds for $I_s(\underline{\theta}, v)$ from Lemma \ref{Le: IntEstimate}, we need a preliminary step. 

\begin{lemma}\label{Le: LowerBoundDiagEntries}
Let $C\subset K^2_\infty$ be compact and bounded away from $\bm{0}$. If we define $A(\underline{\gamma}, \bm{x})$ to be the maximum of the absolute values of the $(n-3)\times (n-3)$-minors of $\gamma_1H(\bm{x})+\gamma_2M$, then 
\[ A(\underline{\gamma}, \bm{x})\gg_{C, w, \underline{F}} 1\]
for $\underline{\gamma}\in C$ and $\bm{x}\in\supp(w)$.
\end{lemma}
\begin{proof}
Suppose by contradiction that the statement of the lemma is false. Then there exists a sequence $(\underline{\gamma}_k, \bm{x}_k)\in C\times \supp (w)$ such that $A(\underline{\gamma}_k, \bm{x}_k)\leq 1/k$ for all $k\geq 1$. Since $C\times \supp(w)$ is compact, we can pass to a convergent subsequence with limit $(\underline{\gamma}', \bm{x}')\in C\times \supp(w)$. However, since the map $(\underline{\gamma}, \bm{x})\mapsto A(\underline{\gamma}, \bm{x})$ is continuous, this implies that every $(n-3)\times (n-3)$ minor of  $\gamma'_1H(\bm{x}')+\gamma'_2M$ vanishes. Therefore, $\text{rk}(\gamma'_1H(\bm{x}')+\gamma'_2 M)\leq n-4$, which is a contradiction since any $\bm{x}\in \supp(w)$ satisfies the third condition in $\eqref{Eq: PropertiesNicePoint}$. 
\end{proof}
When $\underline{\Theta}=(\Theta_1,\Theta_2)\in \ZZ^2$, then we shall henceforth adopt the convention that $|\underline{\gamma}|=\underline{\widehat{\Theta}}$ means $|\theta_1|=\widehat{\Theta}_1$ and $|\theta_2|=\widehat{\Theta}_2$. We finally have all the ingredients at hand to provide an upper bound for an average of $I_s(\underline{\theta},\bm{v})$ over $\underline{\theta}$.
\begin{prop}\label{Prop: IntEstimate}
Let $\underline{\Theta}\in\ZZ^2$, $\bm{v}\in\O^n$, $s\in \O$ be monic and put $\widehat{Z}=\max \{1,\widehat{P}^3\widehat{\Theta}_1,\widehat{P}^2\widehat{\Theta}_2\}$. Then if $|\theta_2|\gg_{F_1,F_2,w} \widehat{P}|\theta_1|$, we have
\[
\int_{|\underline{\theta}|=\underline{\widehat{\Theta}}}I_s(\underline{\theta},\bm{v})\dd \underline{\gamma}\ll_{F_1, F_2, w}\widehat{\Theta}_1\widehat{\Theta}_2\widehat{Z}^{-(n-1)/2}
\]
while if $|\theta_2|\ll_{F_1,F_2,w} \widehat{P}|\theta_1|$, then
\[
\int_{|\underline{\theta}|=\underline{\widehat{\Theta}}}I_s(\underline{\theta},\bm{v})\dd \underline{\gamma}\ll_{F_1,F_2, w}\widehat{\Theta}_1\widehat{\Theta}_2\widehat{Z}^{-(n-2)/2}.
\]
\end{prop}
\begin{proof}
For the ease of notation, let us write $\bm{w}=t^P \bm{v}(st^L)^{-1}$ and $\gamma_i=t^{(4-i)P}\theta_i$ for $i=1,2$. If $\widehat{Z}=1$, then we use the trivial estimate $I_s(\underline{\theta},\bm{v})\leq \widehat{L}^{-n}$ that is an immediate consequence of \eqref{Eq: RelateIntegraltoJ}. We shall therefore  assume $\widehat{Z}>1$ from now on. It then follows from \eqref{Eq: RelateIntegraltoJ} and Lemma~\ref{Le: IntEstimate} after an obvious change of variables that 
\begin{align*}
    |I_{s}(\underline{\theta},\bm{v})|&\leq \widehat{L}^{-n}\text{meas}\{\bm{x}\in \TT^n \colon |t^{3P}\theta_1 \nabla G_1 (\bm{x})+\theta_2t^{2P}\nabla G_2(\bm{x})+\bm{w}|< H_{\underline{G}}\widehat{Z}^{1/2}\}\\
    &\leq \text{meas}\{\bm{x}\in\TT^n\colon |\bm{x}-\bm{x}_0|< \widehat{L}^{-1}, |\theta_1t^{3P}\nabla F_1(\bm{x})+\theta_2 t^{2P}\nabla F_2(\bm{x})+\bm{w}|<H_{\underline{F}}\widehat{Z}^{1/2}\}.
\end{align*}
Now let us denote the last set whose measure we want to estimate by $\Omega$ and suppose \break $\bm{x}, \bm{x}+\bm{x}'\in \Omega$.  By definition of $\Omega$, we must then have 
\begin{equation}\label{Eq: EquationOmega}
    |\gamma_1(\nabla F_1(\bm{x}+\bm{x}')-\nabla F_1(\bm{x}))+\gamma_2 \nabla F_2(\bm{x}')|<H_{\underline{F}}\widehat{Z}^{1/2}.
\end{equation}
We now distinguish between the relative sizes of $\gamma_1$ and $\gamma_2$. Firstly, suppose $|\gamma_2|\gg |\gamma_1|$, so that $\widehat{Z}=|\gamma_2|$. Since $\rk(M)\geq n-1$, there exists indices $1\leq i, j \leq n$ such that the submatrix $M'$ obtained from $M$ by deleting the $i$th row and $j$th column has rank $n-1$. Let us now fix $a\in \TT^n$ and consider the set $\Omega_a$ of $\bm{x}\in \Omega$ whose $j$th entry is $a$. Assume $\Omega_a$ is non-empty and $\bm{x}', \bm{x}'+\bm{x}$ are both in $\Omega_a$. We shall now write $\bm{x}'_{\hat{j}}$ for the vector obtained from $\bm{x}'$ by deleting the $j$th entry and similarly for $\bm{x}$ and $\bm{x}+\bm{x}'$. Note that the $j$th entry of $\bm{x}$ must be $0$. In addition, $H'$ denotes the submatrix of $H$ after deleting the $i$th row and $j$th column.  It then follows from \eqref{Eq: EquationOmega} that
\begin{equation}\label{Eq: IntEstimategamma2}
|(\gamma_1H'(\bm{x}+\bm{x}')+\gamma_2 M')\bm{x}_{\hat{j}}|\leq |(\gamma_1 H(\bm{x}+\bm{x}')+\gamma_2M)\bm{x}| \ll\widehat{Z}^{1/2}.
\end{equation}
Since $\rk M'=n-1$, we have $M'\bm{x}_{\hat{j}}\gg |\bm{x}_{\hat{j}}|$. In particular, the trivial estimate \break $H'(\bm{x}+\bm{x}')\bm{x}_{\hat{j}}\ll |\bm{x}_{\hat{j}}|$ together with the assumption $|\gamma_1|\ll |\gamma_2|$ implies that \eqref{Eq: IntEstimategamma2} can only hold if 
\[
|\gamma_2M'\bm{x}_{\hat{j}}|\ll \widehat{Z}^{1/2}.
\]
We can now multiply the left hand side by $M'^{-1}$, whose entries have absolute value $O(1)$, to deduce that $|\bm{x}_{\hat{j}}|\ll \widehat{Z}^{-1/2}$ and thus 
\[
\int_{|\underline{\gamma}|=\widehat{\underline{\Theta}}}I_s(\underline{\theta}, \bm{v})\dd \underline{\gamma} \ll \int_{|\underline{\gamma}|=\widehat{\underline{\Theta}}} \int_{\TT^n}\text{meas}(\Omega_a)\dd a \dd \underline{\gamma} \ll \widehat{\Theta}_1\widehat{\Theta}_2\widehat{Z}^{-(n-1)/2},
\]
which is satisfactory.

We now treat the more complicated case when $|\gamma_1|\gg |\gamma_2|$, so that $\widehat{Z}=|\gamma_1|$.\break For $\bar{i}=\{i_1, i_2, i_3\}$, $\bar{j}=\{j_1,j_2, j_3\}\subset \{1,\dots, n\}$ and a matrix $B\in \text{Mat}_{n\times n}(K_\infty)$, we write $B_{\bar{i}, \bar{j}}$ for the matrix obtained from $B$ by deleting the $i_1$th, $i_2$th and $i_3$th rows as well as the $j_1$th, $j_2$th and $j_3$th columns. It follows from Lemma \ref{Le: LowerBoundDiagEntries} that 
\[
A\coloneqq \max_{\substack{\bar{i},\bar{j}\subset \{1,\dots, n\} \\ |\bar{i}|=|\bar{j}|=3}} |\text{det} ((\gamma_1 H(\bm{x})+\gamma_2M)_{\bar{i},\bar{j}}| \gg |\gamma_1|
\]
for $\bm{x}\in \supp (w)$. Next we divide $\Omega$ into at most $(n(n-1)(n-2)/6)^2$ subsets according to the indices at which the maximum above occurs, that is for $\bar{i}, \bar{j}\subset \{1,\dots, n\}$ we set 
\[
\Omega_{\bar{i},\bar{j}}\coloneqq \{\bm{x}\in \Omega \colon A= |\text{det} ((\gamma_1 H(\bm{x})+\gamma_2M)_{\bar{i},\bar{j}}|\}.
\]
Moreover, to estimate the measure of $\Omega_{\bar{i}, \bar{j}}$ we shall again fix the $j_1$th, $j_2$th and $j_3$th entries of $\bm{x}$ and denote by $\bm{x}_{\bar{j}}$ the vector obtained from $\bm{x}$ by deleting the $j_1$th, $j_2$th and $j_3$th entries, so that 
\[
\text{meas}(\Omega_{\bar{i},\bar{j}})\leq \int_{\TT^3}\text{meas}\{\bm{x}\in \Omega_{\bar{i},\bar{j}}\colon x_{j_k}=a_k \text{ for }k=1,2,3\}\dd \underline{a}.
\]
If $\bm{x}', \bm{x}+\bm{x}'$ are both in $\Omega_{\bar{i},\bar{j}}$ and $x'_{j_k}, x_{j_k}+x'_{j_k}=a_k$ for $k=1,2,3$, then \eqref{Eq: EquationOmega} implies that 
\[
|(H(\bm{x}+\bm{x}')+\gamma_2 \gamma_1^{-1}M)_{\bar{i},\bar{j}} \bm{x}_{\bar{j}}| \leq |(H(\bm{x}+\bm{x}')+\gamma_2\gamma_1^{-1}M)\bm{x}|\ll \widehat{Z}^{-1/2}.
\]
Since $\bm{x}+\bm{x}'\in \Omega_{\bar{i},\bar{j}}$, the entries of the inverse of $(H(\bm{x}+\bm{x}')+\gamma_2 \gamma_1^{-1}M)_{\bar{i}, \bar{j}}$ have absolute value $O(1)$. In particular, after multiplying the last equation above with it from the left we get that $|\bm{x}_{\bar{i},\bar{j}}|\ll \widehat{Z}^{-1/2}$. From what we have shown so far, it thus follows that 
\begin{align}
\begin{split}\label{Eq: penultimateEqationIntEstimate}
\int_{|\underline{\gamma}|=\widehat{\underline{\Theta}}}I_s(\underline{\theta}, \bm{v})\dd \underline{\gamma} &\ll \sum_{\substack{\bar{i},\bar{j}\subset \{1,\dots, n\} \\ |\bar{i}|=|\bar{j}|=3}}\int_{|\underline{\gamma}|=\widehat{\underline{\Theta}}} \text{meas}(\Omega_{\bar{i},\bar{j}})\dd \underline{\gamma}\\
 &\ll \widehat{Z}^{-(n-3)/2}\sum_{\substack{\bar{i},\bar{j}\subset \{1,\dots, n\} \\ |\bar{i}|=|\bar{j}|=3}}\int_{\TT^3}\sup_{\substack{ \bm{x}\in \supp(w) \\ \bm{x}_{j_k}=a_k, k=1,2,3}}\text{meas}(\Omega(\bm{x}))\dd \underline{a},
 \end{split}
\end{align}
where 
\[
\Omega(\bm{x})\coloneqq \{|\underline{\gamma}|=\widehat{\underline{\Theta}}\colon |\gamma_1\nabla F_1(\bm{x})+\gamma_2\nabla F_2(\bm{x})+\bm{w}|<\widehat{Z}^{1/2}\}.
\]
Now since $\bm{0}\not\in\supp(w)$ and $F_1$ is non-singular, we have $|\nabla F_1(\bm{x})|\gg 1$ for $\bm{x}\in \supp(w)$. In particular, the condition $|\gamma_1\nabla F_1(\bm{x})+\gamma_2\nabla F_2(\bm{x})+\bm{w}|<\widehat{Z}^{1/2}$ can only hold if $\gamma_1 = a\gamma_2 +t$ for some $a\in \TT$ and $|t|\ll \widehat{Z}^{1/2}$. Since $\gamma_1=\theta_1t^{3P}$, we therefore have 
\[
\text{meas}(\Omega(\bm{x}))\ll \widehat{\Theta}_2\widehat{Z}^{1/2}\widehat{P}^{-3}\ll\widehat{\Theta}_1\widehat{\Theta}_2\widehat{Z}^{-1/2}.
\]
The conclusion of the lemma now follows upon plugging this into \eqref{Eq: penultimateEqationIntEstimate}.
\end{proof}
\begin{remark}
When $|\gamma_1|>C|\gamma_2|$ for a sufficiently large constant $C>0$, then the method that was used when $|\gamma_1|\ll |\gamma_2|$ would have handed us the estimate $\text{meas}(\Omega)\ll \widehat{Z}^{-(n-2)/2}$, since $\rk(H(\bm{x}))\geq n-2$ for every $\bm{x}\in \supp(w)$. Moreover, we could have again saved an additional factor of $\widehat{Z}^{-1/2}$ by averaging over $\underline{\gamma}$ in the same way as we did in the second case of the proof of the lemma. In total this would have yielded
\[
\int_{|\underline{\gamma}|=\widehat{\underline{\Theta}}}I_s(\underline{\theta}, \bm{v})\dd \underline{\gamma} \ll \widehat{\Theta}_1\widehat{\Theta}_2 \widehat{Z}^{-(n-1)/2},
\]
so that we have to use the worse upper bound from Lemma \ref{Le: IntEstimate} only when $|\gamma_1|\asymp |\gamma_2|$. However, this improvement is not necessary for our work.
\end{remark}
\section{Exponential sums: Pointwise estimates}\label{Se: Exp1}
The aim of this section is to collect estimates for the complete exponential sums $S_{d\underline{c}, r, \bm{b}, N}(\bm{v})$ defined in \eqref{Defi: SdcrbN}. These sums enjoy a twisted multiplicativity property, which essentially reduces the task of estimating them to the case of prime power moduli. For $r, R\in \O$, we adopt the notation 
\[
r\mid R^\infty
\]
to mean that every prime divisor of $r$ also divides $R$. 
\begin{lemma}\label{Le: MultiS(v)}
Suppose $d\mid r$ and $r=r_1r_2$ with $(r_1,r_2)=1$. If we write $N=N_1N_2N_3$, where $N_i\mid r_i^\infty$ for $i=1,2$ and $(r,N_3)=1$, and let $s_i=r_iN_i/(r_i,N_i)$ for $i=1,2$, then there exist $\bm{b}' \in (\O/N_3\O)^n$ and $t_i\in(\O/s_i\O)^\times$ for $i=1,2$ such that 
\[
S_{d\underline{c},r,\bm{b}, N}(\bm{v})=S_{d_1\underline{c}, r_1,\bm{b}, N_1}(t_1\bm{v})S_{d_2\underline{c},r_2,\bm{b}, N_2}(t_2\bm{v})\psi\left(\frac{-\bm{v}\cdot \bm{b}'}{N_3}\right).
\]
where $d=d_1d_2$ with $d_i\mid r_i$ for $i=1,2$.
\end{lemma}
\begin{proof}
By construction, $s_1,s_2$ and $N_3$ are pairwise coprime so that $r_N=s_1s_2N_3$. In particular, if $\bm{y}_i$ runs through a complete sets of residues modulo $s_i$ for $i=1,2$ and $\bm{y}_3$ modulo $N_3$, then 
\[
\bm{x}=s_2N_3\bm{y}_1+s_1N_3\bm{y}_2+s_1s_2\bm{y}_3
\]
constitutes a complete set of residues modulo $r$. Next, for $\underline{a}/r\in L(d\underline{c})$, we write \break $\underline{a}=r_2\underline{a}_1+r_1\underline{a}_2$, where $|\underline{a}_i|<|r_i|$ and $(\underline{a}_i,r_i)=1$. It is then clear that 
\[
\psi\left(\frac{\underline{a}\cdot \underline{F}(\bm{x})}{r}\right)= \psi\left(\frac{\underline{a}_1\cdot \underline{F}(s_2N_3\bm{y_1})}{r_1}\right)\psi\left(\frac{\underline{a}_2\cdot\underline{F}(s_1N_3\bm{y}_2)}{r_2}\right)
\]
and
\[
\psi\left(\frac{-\bm{v}\cdot\bm{x}}{r_N}\right)=\psi\left(\frac{-\bm{v}\cdot \bm{y}_1}{s_1}\right)\psi\left(\frac{-\bm{v}\cdot\bm{y}_2}{s_2}\right)\psi\left(\frac{-\bm{v}\cdot \bm{y}_3}{N_3}\right).
\]
Moreover, it is demonstrated in the proof of Lemma 5.2 in \cite{vishe2019rational} that $\underline{a}/r\in L(d\underline{c})$ if and only if $\underline{a}_i/r_i\in L(d_i\underline{c})$. The result now follows after the change of variables $\bm{x}_1=s_2N_3\bm{y}_1$ and $\bm{x}_2=s_1N_3\bm{y}_2$ and taking $t_1\equiv (s_2N_3)^{-1} \: (s_1)$, $t_2\equiv(s_1N_3)^{-1} \: (s_2)$ and $\bm{b}'\equiv(s_1s_2)^{-1}\bm{b} \:(s_3)$.
\end{proof}
In some cases we will obtain estimates for the sums $S_{d\underline{c},r,\bm{b},N}(\bm{v})$ by considering their relatives 
\begin{equation}\label{Eq: DefT(a,r,v)}
T(\underline{a},r,\bm{v})\coloneqq \sum_{|\bm{x}|<|r|}\psi\left(\frac{a_1G_1(\bm{x})+a_2G_2(\bm{x})-\bm{v}\cdot \bm{x}}{r}\right)
\end{equation}
for appropriate polynomials $G_1,G_2\in\O[x_1,\dots,x_n]$. These sums satisfy the following twisted multiplicativity property.
\begin{lemma}\label{Le: MultiT(v)}
Let $r=r_1r_2$ with $(r_1,r_2)=1$. Then \[
T(\underline{a},r,\bm{v})=T(\underline{a}_{r_2}, r_1, \bm{v})T(\underline{a}_{r_1},r_2,\bm{v}),
\]
where $\underline{a}_s\coloneqq (s^2a_1, sa_2)$ for $s\in\O$.
\end{lemma}
\begin{proof}
As $\bm{x}_i$ runs over a full set of residues mod $r_i$, $\bm{x}=r_2\bm{x}_1+r_1\bm{x}_2$ runs over a full set of residues mod $r$. Moreover, using Taylor's formula it is easy to see that \[
\psi\left(\frac{a_iF_i(\bm{x})}{r}\right)=\psi\left(\frac{a_ir_2^{4-i}F_i(\bm{x}_1)}{r_1}\right)\psi\left(\frac{a_ir_1^{4-i}F_i(\bm{x}_2)}{r_2}\right)
\]
for $i=1,2$ and 
\[
\psi\left(\frac{-\bm{v}\cdot \bm{x}}{r}\right)=\psi\left(\frac{-\bm{v}\cdot\bm{x}_1}{r_1}\right)\psi\left(\frac{-\bm{v}\cdot\bm{x}_2}{r_2}\right),
\]
from which the statement of the lemma follows.
\end{proof}
For our investigation we shall also need a good understanding of the distribution of rational points $\underline{a}/r$ on an individual line $L(d\underline{c})$ when $r$ is fixed. By Lemma \ref{Le: MultiS(v)} it suffices to consider the case $r=\varpi^k$ and $d=\varpi^m$ with $m\leq k$. The following lemma summarises the content of equations (6.9)--(6.11) of \cite{vishe2019rational}.
\begin{lemma}\label{Le: L(dc)RatPoints}
If $1\leq m<k$, then modulo $\varpi^k$ we have the following equality of sets
\begin{align*}
\{\underline{a}\colon \underline{a}/\varpi^k \in L(\varpi^m\underline{c})\}=&\{a\underline{c}^\bot+\varpi^{k-m}\underline{d}\colon |a|<|\varpi|^{k-m}, (a,\varpi)=1, |\underline{d}|<|\varpi|^m\}\setminus \\
&\{a\underline{c}^\bot +\varpi^{k-m+1}\underline{d}\colon |a|<|\varpi|^{k-m+1},|\underline{d}|<|\varpi|^{m-1}, (a,\varpi)=1\}
\end{align*}
and for $k=m$ we have 
\begin{align*}
    \{\underline{a}\colon \underline{a}/\varpi^k\in L(\varpi^k\underline{c})\}= & \{\underline{d}\colon (\underline{d},\varpi)=1, |\underline{d}|<|\varpi|^k\}\setminus \\
    &\{a\underline{c}^\bot +\varpi\underline{d}\colon (a,\varpi)=1, |a|<|\varpi|, |\underline{d}|<|\varpi|^{k-1}\}. 
\end{align*}
Moreover, when $m=0$, then 
\[
\{\underline{a}\colon \underline{a}/\varpi^k \in L(\underline{c})\} = \{a\underline{c}^\bot \colon (a,\varpi)=1, |a|<|\varpi|^k\}.
\]
In particular, we have $\#\{\underline{a}\colon \underline{a}/r\in L(d\underline{c})\}\leq |d||r|$.
\end{lemma}
\subsection{Square-free moduli} 
We will now deal with $S_{d\underline{c},r,\bm{b},N}(\bm{v})$ when $r$ is square-free. A key player in our estimates is the dual form $F_1^*\in \O[x_1,\dots, x_n]$, whose zero locus parameterises hyperplanes that have singular intersection with the projective hypersurface defined by $F$. It is well known that $F_1^*$ is absolutely irreducible and of degree $3\times 2^{n-2}$. 
We begin our treatment by assuming that $d=1$. In this case Lemma \ref{Le: L(dc)RatPoints} tells us that $S_{\underline{c},\varpi, \bm{0}, 1}(\bm{v})$ equals the familiar exponential sum
\[
S_\varpi(\bm{v})\coloneqq \sideset{}{'}\sum_{a \:(\varpi)}\sum_{\bm{x} \:(\varpi) } \psi\left(\frac{aF_{\underline{c}}(\bm{x})-\bm{v}\cdot \bm{
x}}{\varpi}\right),
\]
where $F_{\underline{c}}(\bm{x})=-c_2F_1(\bm{x})+c_1F_2(\bm{x})$. Let $\FF_\varpi=\O/\varpi \O$ be the residue field of $\varpi$. Our main ingredient is the following special case of a result due to Katz \cite[Theorem 4]{katz1999estimates}.
\begin{theorem}\label{Th: katz}
    Let $X\subset \PP^n_{\FF_\varpi}$ be a complete intersection of dimension $r$ defined by forms of degrees $d_1, \dots, d_{n-r}$  and let $L, H\in H^0(\PP^n, O_{\PP^n}(1))$. If  the following conditions are met:
    \begin{enumerate}[(i)]
    \item $X\cap L\cap H$ has dimension $r-2$
    \item the singular locus of $X\cap L$ has dimension $\varepsilon$
    \item the singular locus of $X\cap L \cap H$ has dimension $\delta\geq \varepsilon$,
\end{enumerate}
then there exists a constant $C>0$ depending only on $n, d_1, \dots, d_{n-r}$ such that for $f=H/L$ it holds that 
\[
\sum_{\bm{x}\in X[1/L]}\psi\left(\frac{f(\bm{x})}{\varpi}\right) \leq C |\varpi|^{(r+1+\delta)/2},
\]
where $X[1/L]$ is the affine variety defined as the complement of the hyperplane cut out by $L$ in $X$.
\end{theorem}
\begin{remark}
    Katz states Theorem 4 for arbitrary closed subvarieties of projective space that are geometrically integral or equidimensional and Cohen-Macualy. However, our assumption that $X$ is a complete intersection implies that $X$ is Cohen-Macauly and equidimensional, thereby allowing us to state the simplified version above.
\end{remark}
Suppose that $\varpi\nmid \bm{v}c_2$. Using orthogonality of characters we see that
\[
S_{\varpi}(\bm{v})=|\varpi|\sum_{\substack{\bm{x} \in\FF_\varpi^n \\ F_{\underline{c}}(\bm{x})=0}}\psi\left(\frac{-\bm{v}\cdot \bm{x}}{\varpi}\right).
\]
Let $F(x_0,\bm{x})\in\O[x_0,x_1,\dots, x_n]$ be the homogenization of $F_{\underline{c}}$, that is 
\[F(x_0,\bm{x})=-c_2F_1(\bm{x})+x_0c_1F_2(\bm{x}),\]
and define $X=V(F)\subset \PP_{\FF_\varpi}^n$ to be the projective variety cut out by the reduction of $F$ modulo $\varpi$. Note that the point $(1,0,\dots, 0)$ will always be a singularity of $X$. 
Moreover, we also set $L(x_0,\bm{x})=x_0$ and $H(\bm{x})=-\bm{v}\cdot \bm{x}$. In our situation we thus have $X\cap L= V(F_1)$ and $X\cap L\cap H= V(F_1, \bm{v}\cdot \bm{x})$. In particular, $\delta=\varepsilon = -1$ provided $\varpi \nmid \Delta_{F_1}F_1^*(\bm{v})$, where $\Delta_{F_1}$ is the discriminant of $F_1$. Indeed, the condition $\varpi\nmid \Delta_{F_1}$ guarantees that the reduction of $F_1$ modulo $\varpi$ is non-singular and $\varpi\nmid F_1^*(\bm{v})$ implies that $V(F_1,\bm{v}\cdot\bm{x})\subset \PP^{n-1}_{\FF_\varpi}$  is non-singular. 
We can thus apply Theorem \ref{Th: katz} with $\delta = -1$ and $r=n-1$ to deduce that
\[
|\varpi|^{-1}|S_{\varpi}(\bm{v})|\leq C |\varpi|^{(n-1)/2},
\]
where $C$ is a constant that only depends on the degrees of $F_1$ and $F_2$ and $n$. Absorbing the primes $\varpi \mid \Delta_{F_1}$ into the constant and invoking Lemma \ref{Le: MultiS(v)}, we have thus established the following result.
\begin{lemma}\label{Le: ExpSumSwgeneric}
Suppose that $r\in \O$ is square-free with $(r, F_1^*(\bm{v})c_2)=1$. There exists a constant $C>0$ depending only on $\Delta_{F_1}, \deg F_1, \deg F_2$ and $n$ such that 
\[
S_{\underline{c}, r, \bm{0},1}(\bm{v}) \leq C^{\omega(r)}|r|^{(n+1)/2},
\]
 where $\omega(r)$ denotes the number of prime divisors of $r$
\end{lemma}
Let us now turn to the case $d\neq 1$. We shall use the following estimate of Deligne \cite[Th\'eor\`eme 8.4]{DeligneWI}, which states that for a polynomial $F\in\FF_\varpi[x_1,\dots, x_n]$ of degree $d$ with $\cha(\FF_\varpi)\nmid d$ such that the highest degree part cuts out a smooth projective hypersurface in $\PP^{n-1}_{\FF_\varpi}$, one has
\begin{equation}\label{Eq: DelignePoly}
\left|\sum_{\bm{x}\in\FF_\varpi^n}\psi\left(\frac{F(\bm{x})}{\varpi}\right)\right|\leq (d-1)^n |\varpi|^{n/2}.
\end{equation}
Recalling the definition of the sum $T$ in \eqref{Eq: DefT(a,r,v)} with $G_i=F_i$ for $i=1,2$, the estimate \eqref{Eq: DelignePoly} implies 
\[
T(\underline{a}, \varpi,\bm{v})\leq 2^n|\varpi|^{n/2}
\]
whenever $\varpi \nmid a_1\Delta_{F_1}$. On the other hand, if $\varpi\mid a_1$, then $\varpi\nmid a_2$ and then 
\begin{equation}\label{Eq: QuadSumEstimate}
T(\underline{a},\varpi, \bm{v})\leq |\varpi|^{(n+1)/2},
\end{equation}
provided $F_2$ is a quadratic form of rank at least $n-1$ modulo $\varpi$, as for example follows from \cite[Lemma 3.5]{vishe2019rational}. Now let us assume $r\in\O$ is square-free and write $r=r_1r_2$, with $(r_1,r_2)=1$ and $r_2\mid N^\infty$. If $d=d_1d_2$ with $d_1\mid r_1$ and $d_2\mid r_2$, then by Lemma \ref{Le: MultiS(v)} we have 
\[
S_{d\underline{c},r,\bm{b},N}(\bm{v})=S_{d_1\underline{c}, r_1, \bm{0}, 1}(t_1\bm{v})S_{d_2\underline{c},r_2,\bm{b},N}(t_2\bm{v})
\]
for some $t_i\in (\O/r_i\O)^\times$. After absorbing the primes $\varpi\mid \Delta_{F_1}$ or for which the reduction of $F_2$ has rank strictly less than $n-1$ into the constant, it now follows from the estimates we just recorded and Lemma \ref{Le: L(dc)RatPoints} that
\begin{align*}
    S_{d_1\underline{c},r_1, \bm{0}, 1}(t_1\bm{v})&= \sum_{\underline{a}/r \in L(d_1\underline{c})}T(\underline{a}, r_1, t_1\bm{v})\leq C^{\omega(r_1)} |d_1||r_1|^{(n+3)/2},
\end{align*}
for some constant $C>0$.

We can now estimate $S_{d_2\underline{c}, r_2, \bm{b}, M}(t_2\bm{v})$ trivially to arrive at the following result.
\begin{lemma}\label{Le: ExpSumSquarefreedc}
Suppose $r$ is square-free and $d\mid r$. Then there exists a constant $C>0$ depending only on $\deg F_1,  F_2, \Delta_{F_1}$ and $N$ such that
\[
S_{d\underline{c},r,\bm{b}, N}(\bm{v})\leq C^{\omega(r)}|d||r|^{(n+3)/2}.
\]
Moreover, if $(c_2,r)=1$, then 
\[
S_{\underline{c},r, \bm{0},1}(\bm{v})\leq C^{\omega(r)}|r|^{n/2+1}.
\]
\end{lemma}

\subsection{Square-full moduli} To satisfactorily deal with square-full moduli, we begin with the case $r=\varpi^2$. Our main ingredient is the following result due to Heath-Brown \cite{HeathBrownExpp2}. It should be noted that Heath-Brown's proves his result solely over the integers. However, it is a routine exercise and the required adaptions are minor to check that his argument holds over $\FF_q[t]$ as well. Let $F\in\O[x_1,\dots, x_n]$ be a polynomial of degree $d$ with $\cha(\FF_q)>d$ and suppose the reduction of the top degree part of $F$ defines a smooth projective hypersurface modulo $\varpi$. Then it holds that
\begin{equation}
    \sum_{\bm{x} \: (\varpi^2)}\psi\left(\frac{F(\bm{x})}{\varpi^2}\right)\leq (d-1)^n|\varpi|^n.
\end{equation}
After absorbing the contribution from the primes dividing $\Delta_{F_1}$ into the constant and employing Lemma \ref{Le: L(dc)RatPoints}, we arrive at the following estimate.
\begin{lemma}\label{Le: ExpSumSquareGeneric}
If $\varpi \nmid a_1$, then 
\[
T(\underline{a}, \varpi^2, \bm{v})\ll_{\Delta_{F_1}} |\varpi|^n. 
\]
In particular, we also have 
\[
S_{\underline{c},\varpi^2, \bm{0}, 1}(\bm{v}) \ll_{\Delta_{F_1}} |\varpi|^{2+n}
\]
provided $\varpi\nmid c_2$. 
\end{lemma}
Since $F_2$ is a quadratic form, there exists a matrix $G\in \GL_n(K)$ with entries in $\O$ such that after the change of variables $\bm{y}=G\bm{x}$ one has
\[
F_2(\bm{y})=\sum_{i=1}^nb_iy_i^2\quad\text{with }b_2\dots b_n\neq 0.
\]
If $\varpi \mid \det(G)$, we may still locally diagonalise $F_2$ with a matrix $G_\varpi\in \GL_n(K_\varpi)$ that has coefficients in $\O_\varpi$, so that after the change of variables $\bm{y}=G_\varpi \bm{x}$ we have 
\[
F_2(\bm{y})=\sum_{i=1}^n b_{\varpi, i}y_i^2 \quad\text{with }b_{\varpi, 2}\cdots b_{\varpi,n}\neq 0.
\]
Let us define 
\[
\Delta_{F_2}=\begin{cases} b_1\cdots b_n\prod_{\varpi \mid \det(G)}\varpi^{\nu_\varpi(b_{\varpi,1}\cdots b_{\varpi,n})} &\text{if }\rk(M)=n,\\ b_2\cdots b_n\prod_{\varpi \mid \det(G)}\varpi^{\nu_\varpi(b_{\varpi,2}\cdots b_{\varpi,n})} & \text{if }\rk(M)=n-1.\end{cases}
\]
and 
\begin{equation}\label{Eq: Defi.nuvarpi}
v_\varpi = \nu_{\varpi}(\Delta_{F_2}).
\end{equation}
We then have the following result \cite[Lemma 6.2]{browning2015rational}. 
\begin{lemma}\label{Le: ExpSumBada1}
Let $k\geq 2$ and suppose $\varpi^{1+v_\varpi}\mid a_1$. Then for any $\bm{v}\in \O^n$ we have
\[
T(\underline{a},\varpi^k, \bm{v})\ll_{F_2} |\varpi|^{(n+1)/2}.
\]
\end{lemma}
Lemma 6.2 in \cite{browning2015rational} is again only proved for the analogous sum over the integers. Moreover, they assume that $F_2$ is diagonal from the beginning, a difference we take care of by the diagonalisation process above. The proof goes through verbatim in our setting, and so we shall not repeat it here. 
\begin{corollary}\label{Cor: ExpSumsquarepartd}
Let $r\in\O$ be such that $\varpi\mid r$ implies $\varpi^2\parallel r$. Then there exists a constant $C>0$ depending on $F_1$, $F_2$ and $N$ such that
\[
S_{d\underline{c}, r, \bm{b},N}(\bm{v})\leq C^{\omega(r)} |d||r|^{(n+3)/2}.
\]
\end{corollary}
\begin{proof}
Write $r=r_1r_2$ with coprime $r_1,r_2\in \O$ such that $(r_1,N)=1$ and $r_2\mid N^\infty$. As in the proof of Lemma \ref{Le: ExpSumSquarefreedc} it suffices to obtain and upper bound for the sum $S_{d_1\underline{c}, r_1, \bm{0}, 1}(t_1\bm{v})$, where $t_1\in (\O/r_1\O)^\times$, and estimate the sum corresponding to $r_2$ trivially. By definition we then have 
\[
S_{d_1\underline{c}, r_1, \bm{0}, 1}(t_1\bm{v})=\sum_{\underline{a}/r\in L(d_1\underline{c})}T(\underline{a}, r_1, t_1\bm{v}).
\]
Using the mulitplicativity property recorded in Lemma \ref{Le: MultiT(v)}, we can now invoke Lemma \ref{Le: ExpSumSquareGeneric} in conjunction with Lemma \ref{Le: ExpSumBada1} to obtain $T(\underline{a},r_1,t_1\bm{v})\leq C^{\omega(r_1)}|r_1|^{(n+1)/2}$. Lemma \ref{Le: L(dc)RatPoints} provides us with an upper bound for the number of $\underline{a}$'s such that $\underline{a}/r_1\in L(d_1\underline{c})$ that completes the proof.
\end{proof}
\subsection{The case $c_2=0$.}
We now consider separately the case $c_2=0$. This can only occur if $c_1=1$ and so to ease notation, we write $\underline{c}_0\coloneqq (1,0)$. By Lemma \ref{Le: MultiS(v)} we can reduce to the case when $r=\varpi^k$ and $d=\varpi^m$ with $m\leq k$ and we again begin our treatment assuming that $r=\varpi$. When $d=1$, Lemma \ref{Le: ExpSumSquarefreedc} already provides sufficiently good upper bounds. However, when $d=\varpi$, we have to do better and establishing an estimate that is superior to Lemma~\ref{Le: ExpSumSquarefreedc} is our first goal.
For $k\geq 1$, let us define 
\[
\rho_1(\varpi^k)=\#\{\bm{x}\: (\varpi)\colon F_1(\bm{x})\equiv F_2(\bm{x})\equiv 0 \:(\varpi^k)\}\quad\text{and}\quad \rho_2(\varpi^k)=\#\{\bm{x}\: (\varpi)\colon  F_2(\bm{x})\equiv 0 \:(\varpi^k)\}.
\]
By Lemma \ref{Le: L(dc)RatPoints} we have 
\begin{align*}
    S_{\varpi\underline{c}_0, \varpi, \bm{0},1}(\bm{v})&=\sideset{}{'}\sum_{a_1 \:(\varpi)}\sum_{a_2\: (\varpi)}\sum_{\bm{x}\:(\varpi)}\psi\left(\frac{a_1F_1(\bm{x})+a_2F_2(\bm{x})-\bm{v}\cdot\bm{x}}{\varpi}\right)\\
    &= |\varpi|\Biggl(|\varpi|\sum_{\substack{\bm{x}\:(\varpi) \\ F_1(\bm{x})\equiv F_2(\bm{x})\equiv 0 \:(\varpi)}}\psi\left(\frac{-\bm{v}\cdot \bm{x}}{\varpi}\right)-\sum_{\substack{\bm{x}\:(\varpi) \\ F_2(\bm{x})\equiv 0 \:(\varpi)}}\psi\left(\frac{-\bm{v}\cdot\bm{x}}{\varpi}\right)\Biggr).
\end{align*}
If $\varpi \mid \bm{v}$, then the expression above simplifies to 
\[
S_{\varpi\underline{c}_0, \varpi, \bm{0},1}(\bm{v}) = |\varpi|(|\varpi|\rho_1(\varpi)-\rho_2(\varpi)).
\]
Since the reduction of $X$ modulo $\varpi$ is non-singular for $|\varpi|$ sufficiently large, we have 
\[
\rho_1(\varpi)= |\varpi|^{n-2}+O\left(|\varpi|^{(n-1)/2}\right),
\]
as follows for example from Equation (3.12) in \cite{CubicHypersurfacesBV}. 
 Moreover, because $F_2$ has rank at least $n-1$, it holds that 
\[
\rho_2(\varpi)=|\varpi|^{n-1} + O\left(|\varpi|^{(n+1)/2}\right).
\]
So that in total we have 
\begin{equation}\label{Eq: Estimate.Svv.cbad.v0}
    S_{\varpi\underline{c}_0, \varpi, \bm{0},1}(\bm{v}) \ll |\varpi|^{(n+3)/2}
\end{equation}
whenever $\varpi\mid \bm{v}$, where the implied constant only depends on $F_1$ and $F_2$. Let us now deal with the opposite case $\varpi\nmid \bm{v}$. We want to apply Theorem \ref{Th: katz} to our situation. For this, we define $X'_\varpi=V(F_1,F_2)\subset \PP^{n}=\text{Proj}(\FF_\varpi[x_0,x_1,\dots, x_n])$. In addition, we set $L(x_0, \bm{x})=x_0$ and $H(x_0,\bm{x})=-\bm{v}\cdot \bm{x}$. Provided $|\varpi|$ is sufficiently large, $X'_\varpi$ is a complete intersection of dimension $n-2$ with the only singularity at $(1:0:\cdots :0)\in \PP^n$. Moreover, we have $X'_\varpi \cap L = V(F_1,F_2)\subset \PP^{n-1}$, which is non-singular. It follows from a result of Zak and Fulton--Lazarsfeld \cite[Remark 7.5]{fulton1981connectivity} that $X'_\varpi \cap H\cap L= V(F_1,F_1)\cap H\subset \PP^{n-1}$ has at worst isolated singularities, so that in the notation of Theorem \ref{Th: katz} we have $\varepsilon =-1$ and $\delta \leq 0$. In particular, 
\begin{align*}
\sum_{\substack{\bm{x}\:(\varpi) \\ F_1(\bm{x})\equiv F_2(\bm{x})\equiv 0 \:(\varpi)}}\psi\left(\frac{-\bm{v}\cdot\bm{x}}{\varpi}\right) \ll |\varpi|^{(n-1)/2}.
\end{align*}
Combining this with \eqref{Eq: QuadSumEstimate}, we infer
\begin{align*}
    S_{\varpi\underline{c}_0, \varpi, \bm{0},1}(\varpi) & = |\varpi|^2\sum_{\substack{\bm{x} \:(\varpi)\\ F_1(\bm{x})\equiv F_2(\bm{x})\equiv 0 \:(\varpi)}}\psi\left(\frac{-\bm{v}\cdot{x}}{\varpi}\right)-\sum_{a_2 \:(\varpi)}\sum_{\bm{x}\:(\varpi)}\psi\left(\frac{a_2F_2(\bm{x})-\bm{v}\cdot\bm{x}}{\varpi}\right)\\
    &\ll |\varpi|^{(n+3)/2}+|\varpi|^{(n+3)/2}.
\end{align*}
Estimating the contribution from the primes $\varpi\mid N$ trivially, and using Lemma \ref{Le: MultiS(v)}, it thus follows from \eqref{Eq: Estimate.Svv.cbad.v0} that
    \begin{equation}\label{Eq: estimate.Svv.cbad}
    |S_{d\underline{c}_0, d, \bm{b}, N}(\bm{v})|\leq C^{\omega(d)}|d|^{(n+3)/2}
    \end{equation}
    for some constant $C>0$ that only depends on $F_1, F_2$ and $N$. \\

We also require strong upper bounds for the sums 
    \[
S_1 = S_{\varpi\underline{c}_0, \varpi^2, \bm{0},1}(\bm{v})\quad\text{and}\quad S_2=S_{\varpi^2\underline{c}_0, \varpi^2, \bm{0},1}(\bm{v}).
    \]
Let us begin with the former. In this case Lemma \ref{Le: L(dc)RatPoints} implies
\begin{align*}
    S_1 & = \sideset{}{'}\sum_{|a|<|\varpi|}\sum_{|\underline{d}|<|\varpi|}\sum_{|\bm{x}|<|\varpi|^2}\psi\left(\frac{(a+\varpi d_2)F_2(\bm{x})+\varpi d_1F_1(\bm{x})-\bm{v}\cdot \bm{x}}{\varpi^2}\right)\\
    &=|\varpi| \sideset{}{'}\sum_{|a|<|\varpi|^2}\sum_{\substack{|\bm{x}|<|\varpi|^2 \\ F_1(\bm{x})\equiv 0 \: (\varpi)}}\psi\left(\frac{aF_2(\bm{x})-\bm{v}\cdot\bm{x}}{\varpi^2}\right)\\
    &= |\varpi|^{3}\Biggl(\sum_{\substack{|\bm{x}|<|\varpi|^2 \\ F_1(\bm{x})\equiv 0 \:(\varpi) \\ F_2(\bm{x})\equiv 0 \:(\varpi^2)}}\psi\left(\frac{-\bm{v}\cdot \bm{x}}{\varpi^2}\right)-|\varpi|^{-1}\sum_{\substack{|\bm{x}|<|\varpi|^2 \\ F_1(\bm{x})\equiv 0 \:(\varpi) \\ F_2(\bm{x})\equiv 0 \:(\varpi)}}\psi\left(\frac{-\bm{v}\cdot \bm{x}}{\varpi^2}\right)\Biggr)\\
    &=|\varpi|^{3}\left(\Sigma_1 -|\varpi|^{-1}\Sigma_2\right)
\end{align*}
say. The conditions $F_1(\bm{x})\equiv 0 \:(\varpi)$ and $F_2(\bm{x})\equiv 0 \:(\varpi^2)$ are invariant under scaling $\bm{x}$ by any $b$ with $(b,\varpi)=1$ and so we deduce that that 
\begin{align*}
    \Sigma_1 &= \frac{1}{|\varpi|^2(1-|\varpi|^{-1})}\sum_{\substack{|\bm{x}|<|\varpi|^2 \\ F_1(\bm{x})\equiv 0 \: (\varpi) \\ F_2(\bm{x})\equiv 0 \:(\varpi^2)}} \sideset{}{'}\sum_{|b|<|\varpi|^2}\psi\left(\frac{b\bm{v}\cdot\bm{x}}{\varpi^2}\right)\\
    & = \frac{1}{1-|\varpi|^{-1}}\sum_{\substack{\bm{x}\:(\varpi) \\ F_1(\bm{x})\equiv F_2(\bm{x}\equiv 0 \: (\varpi) \\ \bm{v}\cdot \bm{x}\equiv 0 \:(\varpi)}}(\rho_1(\bm{x})-|\varpi|^{-1}\rho_2(\bm{x})),
\end{align*}
where 
\[
\rho_1(\bm{x})\coloneqq \#\{\bm{y}\:(\varpi^2) \colon \bm{y}\equiv \bm{x} \:(\varpi), F_2(\bm{x})\equiv \bm{v}\cdot\bm{x}\equiv 0 \:(\varpi^2)\}
\]
and 
\[
\rho_2(\bm{x})\coloneqq \#\{\bm{y}\:(\varpi^2) \colon \bm{y}\equiv \bm{x} \:(\varpi), F_2(\bm{x})\equiv 0 \:(\varpi^2)\}.
\]
Running the exact argument again yields 
\[
\Sigma_2 = \frac{1}{1-|\varpi|^{-1}}\sum_{\substack{\bm{x} \: (\varpi) \\ F_1(\bm{x})\equiv F_2(\bm{x})\equiv 0 \: (\varpi) \\ \bm{v}\cdot \bm{x}\equiv 0 \:(\varpi)}}(\rho_1'(\bm{x})-|\varpi|^{n-1}),
\]
where 
\[
\rho_1'(\bm{x})\coloneqq \#\{\bm{y} \:(\varpi^2) \colon \bm{y}\equiv \bm{x} \:(\varpi), \bm{v}\cdot\bm{x}\equiv 0 \:(\varpi^2)\}.
\]
Suppose that $\bm{y}=\bm{x}+\varpi\bm{z}$. Then $\bm{y}$ is counted by $\rho_1(\bm{x})$ if and only if ${\varpi \mid (\bm{x}^tM\bm{y}+\varpi^{-1}F_2(\bm{x}))}$ and $\varpi \mid (\bm{v}\cdot\bm{z}+\varpi^{-1}\bm{v}\cdot \bm{x})$. Whereas $\bm{y}$ is counted by $\rho_2(\bm{x})$ if and only if ${\varpi\mid (\bm{x}^tM\bm{y}+\varpi^{-1}F_2(\bm{x}))}$. In particular, we see that $\rho_1(\bm{x})-|\varpi|^{-1}\rho_2(\bm{x})=0$ 
unless
\begin{equation}\label{Eq: RankConditionExpsum}
\rk\begin{pmatrix}
    \bm{v} \\ M\bm{x}
\end{pmatrix} = \rk (M\bm{x}) \mod \varpi.
\end{equation}
Note that for $\bm{x}\not\equiv 0\:(\varpi)$ and $|\varpi|$ sufficiently large this can only happen if $\bm{v}$ and $M\bm{x}$ are proportional, since then the non-singularity of $X$ implies that $M\bm{x}\not\equiv 0\:(\varpi)$. 
In particular, if \eqref{Eq: RankConditionExpsum} holds and $\bm{x}\not\equiv 0\:(\varpi)$, then $\rho_1(\bm{x})-|\varpi|^{-1}\rho_2(\bm{x})=|\varpi|^{n-1}-|\varpi|^{n-2}$, while if $\bm{v}\equiv 0 \:(\varpi)$, then $\rho_1(\bm{0})-|\varpi|^{-1}\rho_2(\bm{0})=|\varpi|^n-|\varpi|^{n-1}$. Moreover, we have $\rho_1'(\bm{x})=|\varpi|^{n-1}$ unless $\varpi \mid \bm{v}$, in which case $\rho_1(\bm{x})=|\varpi|^n$. In total, we thus have 

\[
S_1 = |\varpi|^{3}\left(\frac{|\varpi|^{n-1}}{(1-|\varpi|^{-1})^2}\mathcal{N}_1 -\frac{|\varpi|^{n-1}}{(1-|\varpi|^{-1})^2}\mathcal{N}_2\right)+O\left(|\varpi|^{n+3}\right).
\]
The error term takes care of the contribution from $\bm{x}\equiv \bm{0} \:(\varpi)$ in $\Sigma_1$ and $\Sigma_2$, and where we have defined
\[
\mathcal{N}_1\coloneqq \#\{ \in \FF_\varpi^n\setminus\{\bm{0}\}\colon \bm{v}\cdot \bm{x} = F_1(\bm{x}) = F_2(\bm{x})= 0 , \eqref{Eq: RankConditionExpsum} \text{ holds}\} 
\]
and 
\[
\mathcal{N}_2\coloneqq \#\{ \bm{x} \in \FF_\varpi^n\setminus\{\bm{0}\} \colon F_1(\bm{x})= F_2(\bm{x})= 0, \bm{v}\equiv \bm{0} \:(\varpi)\}. 
\]
Let us first deal with the case $\varpi \mid \bm{v}$. It follows that $\mathcal{N}_1 =  \mathcal{N}_2$ and thus
\[
S_1 =O(|\varpi|^{n+3}).
\]
Suppose next that $\varpi \nmid \bm{v}$. In this case $\mathcal{N}_2=0$ and \eqref{Eq: RankConditionExpsum} holds if and only if $M\bm{x}$ and $\bm{v}$ are proportional. Since $M$ has rank at least $n-1$, this can happen for at most $O(|\varpi|^2)$ choices of $\bm{x}$, so that 
\begin{align*}
    S_1  \ll |\varpi|^{n+2}\mathcal{N}_1 +|\varpi|^{n+3}\ll |\varpi|^{n+4}.
\end{align*}
In total we have therefore established that
\begin{equation}\label{Eq: UpperBoundS_1}
    S_1 \ll |\varpi|^{n+4}.
\end{equation}
Let us now turn to the sum $S_2$. It follows from Lemma \ref{Le: L(dc)RatPoints} that 
\[
S_2\leq \sideset{}{'}\sum_{a_1 \:(\varpi^2)}\sum_{a_2 \:(\varpi^2)}\left|\sum_{\bm{x}\:(\varpi^2)}\psi\left(\frac{a_1F_1(\bm{x})+a_2F_2(\bm{x})-\bm{v}\cdot\bm{x}}{\varpi^2}\right)\right|.
\]
Since $(a_1,\varpi)=1$, we can apply Lemma \ref{Le: ExpSumSquareGeneric} to deduce that the sum over $\bm{x}$ is $O(|\varpi|^{n})$ and hence 
\begin{equation}\label{Eq: EstimateS_2}
S_2\ll |\varpi|^{n+4},
\end{equation}
which completes our treatment of $S_1$ and $S_2$.

Finally, when $\varpi\nmid \Delta_{F_2}$ and $0\leq m < k$, we can invoke Lemma \ref{Le: ExpSumBada1} to deduce that
\begin{equation}\label{Eq: Estimate.c_0bad.squarefullgen}
    S_{\varpi^m\underline{c}_0, \varpi^k, \bm{0}, 1}(\bm{v}) \ll |\varpi|^{m+k(n+3)/2}.
\end{equation}
Using Lemma \ref{Le: MultiS(v)} and estimating the contribution from $N$ trivially, we see that the following result summarises the content of \eqref{Eq: estimate.Svv.cbad} and  \eqref{Eq: UpperBoundS_1}--\eqref{Eq: Estimate.c_0bad.squarefullgen}.
\begin{prop}\label{Prop: S(c,v)estimate.cbad}
    Let $d_1, d_2, d_3\in \O$ be square-free. Then there exists a constant $C>0$ depending on $F_1, F_2$ and $N$ such that 
    \[
    S_{d_1d_2d_3^2\underline{c}_0, d_1d_2^2d_3^2, \bm{b},N}(\bm{v})\leq C^{\omega(d_1d_2d_3)}|d_1|^{(n+3)/2}|d_2d_3|^{n+4}.
    \]
    In addition, let $d, r \in \O$ be both monic such that $d\mid r$ and $(r,\Delta_{F_2})=1$. If $\nu_{\varpi}(d)<\nu_{\varpi}(r)$ for all $\varpi \mid d$, then 
    \[
    S_{d\underline{c}_0, r, \bm{0},1}(\bm{v})\leq C^{\omega(r)}|d||r|^{(n+3)/2}.
    \]
\end{prop}
\section{Exponential sums: Averages}\label{Se: Exp2}
We also need to deal with certain averages over exponential sums. Our first ingredient is the following result, which we apply to our situation in the corollary directly afterwards. 
\begin{lemma}\label{Le: ExpSumAverageI}
Let $\bm{v}_0\in K_\infty^n$ and $V\geq 1$. If $r$ is cube-full and $\underline{a}\in\O^2$ is such that $|(a_1, r)|\ll 1$,  then 
\[
\sum_{\substack{\bm{v}\in \O^n\\ |\bm{v}-\bm{v}_0|<\widehat{V}}}|T(\underline{a},r, \bm{v})|\ll_{F_1} |r|^{n/2+\varepsilon}\left(\widehat{V}^n+|r|^{n/3}\right).
\]
\end{lemma}
The Lemma we just stated follows from equation (6.9) in \cite{CubicHypersurfacesBV}. There only the case when $(a_1,r)=1$ is considered. However, as explained in the paragraph after Lemma 6.4 in \cite{browning2015rational}, the argument leading to the estimate continues to hold when $|(a_1,r)|\ll 1$ after employing some minor modifications. 
\begin{corollary}\label{Cor: ExpSumTheAverage}
Let $\bm{v}_0\in K_\infty^n$ and $V\geq 1$. Suppose $r$ is cube-full such that $d\mid r$ and define 
\[
P_1(r) \coloneqq \{\varpi^k\parallel r \colon  (\varpi, \Delta_{F_2}c_2)=1, \varpi^k\nmid d\} \quad\text{ and } \quad P_2(r)= \{\varpi^k \parallel r \colon (\varpi, \Delta_{F_2})=1, \varpi^k\mid d\}.
\]
If we write $r=r_1r_2$, where
\begin{equation}\label{Eq: DefinitionBadCubefull}
r_1=\begin{cases}\prod\limits_{\varpi^k \in P_1(r)}\varpi^k &\text{if }c_2\neq 0\\ \prod\limits_{\varpi^k\in P_2(r)}\varpi^k &\text{else}\end{cases}
\end{equation}
and $r_2=r/r_1$, then
\[
\sum_{\substack{\bm{v}\in\O^n \\ |\bm{v}-\bm{v}_0|<\widehat{V}}}|S_{d\underline{c},r,\bm{b},N}(\bm{v})|\ll_{F_1,F_2, N} |d||r|^{n/2+1+\varepsilon}|r_2|^{1/2}\left(\widehat{V}^n+|r|^{n/3}\right).
\]
\end{corollary}
Note that since $d\mid r$, the condition $\varpi^k\nmid d$ in the definition of $P_1(r)$ means that every prime factor of $d$ that divides $r_1$ in fact properly divides $r_1$ when $c_2\neq 0$.
\begin{proof}
Denote the sum to be estimated by $S$. After making the change of variables \break $\bm{x}=\bm{y}N+\bm{b}$, we obtain the identity
\[
S_{d\underline{c}, r, \bm{b}, N}(\bm{v})=\psi\left(\frac{-\bm{b}\cdot \bm{v}}{r_N}\right)\sum_{\underline{a}/r\in L(d\underline{c})}T(\underline{a}, r/(r,N), \bm{v})
\]
with underlying polynomials $G_i(\bm{y})= (r,N)^{-1}F_i(N\bm{y}+\bm{b})$ for $i=1,2$ in the definition \eqref{Eq: DefT(a,r,v)}. Since $N\mid F_i(\bm{b})$, it follows that $G_i$ has coefficients in $\O$. Moreover, the cubic part of $G_1$ is given by the non-singular polynomial $g_0(\bm{y})=(r,N)^{-1}N^3F_1(\bm{y})$. We now factor $r/(r,N)$ into its cube-free part $t$ and its cube-full part $s$. Since $r$ is cube-full, we must have $|t|\leq |N|$. Using Lemma \ref{Le: MultiT(v)} and estimating the contribution from the sum corresponding to $t$ trivially, we see that
\begin{equation}\label{Eq: RelateStoT}
|S_{d\underline{c},r,\bm{b},N}(\bm{v})|\leq |N|^n\sum_{\underline{a}/r\in L(d\underline{c})}|T(\underline{a}_t, s, \bm{v})|.
\end{equation}
Next we write $s=s_1s_2$, where 
\[
s_1=\prod_{\substack{\varpi^k \parallel s\\ \varpi^{1+v_\varpi}\nmid a_1}}\varpi^k \quad\text{and}\quad s_2=\prod_{\substack{\varpi^k \parallel s\\ \varpi^{1+v_\varpi}\mid a_1}}\varpi^k,
\]
with $v_\varpi$ defined in \eqref{Eq: Defi.nuvarpi}. It then follows from Lemma \ref{Le: ExpSumBada1} that $T(\underline{a}_{ts_1},s_2,\bm{v})\ll |s_2|^{(n+1)/2+\varepsilon}$. Therefore, after applying Lemma \ref{Le: MultiT(v)} and Lemma \ref{Le: ExpSumAverageI} together with the identity \eqref{Eq: RelateStoT}, we obtain
\begin{align*}
    S&\ll |r|^{\varepsilon}\sum_{\underline{a}/r\in L(d\underline{c})}|s_1|^{n/2}|s_2|^{(n+1)/2}\left(\widehat{V}^n+|s_1|^{n/3}\right)\\
& \ll |r|^{n/2+\varepsilon}\sum_{\underline{a}/r\in L(d\underline{c})}|s_2|^{1/2}\left(\widehat{V}^n+|r|^{n/3}\right).
\end{align*}
Let us write $r=r_1r_2$ as in the statement of the lemma and $d=d_1d_2$ with $d_i\mid r_i^\infty$ for $i=1,2$. The explicit description of $L(d\underline{c})$ in Lemma \ref{Le: L(dc)RatPoints} implies that if $\underline{b}/r_1 \in L(d_1\underline{c})$, then $(b_1,r_1)=1$. It is shown in the proof of \cite[Lemma 5.2]{vishe2019rational} that if $|\underline{a}|<|r|$ and $\underline{a}=r_2 \underline{b}+r_1\underline{b}'$ with $|\underline{b}|<|r_1|, |\underline{b}'|<|r_2|$, then $\underline{a}/r \in L(d\underline{c})$ if and only if $\underline{b}/r_1\in L(d_1\underline{c})$ and  $\underline{b}'/r_2\in L(d_2\underline{c})$. In particular, we must have $|s_2|\leq |r_2|$. Thus it follows from Lemma \ref{Le: L(dc)RatPoints} that 
\[
S\ll |d||r|^{n/2+1+\varepsilon}|r_2|^{1/2}\left(\widehat{V}^n+|r|^{n/3}\right)
\]
as desired.
\end{proof}
Our next result is concerned about averages of $S_{d\underline{c}, s, \bm{b},N}(\bm{v})$ over a sparse set of $\bm{v}\in \O^n$. Let $V\geq 1$, $C_1\geq C_2 \geq 1$ and $\bm{v}_0\in K_\infty^n$. For $d,s\in \O$ with $d\mid s$ and $s$ cube-full, we proceed to consider the average
\begin{equation}\label{Defi: S(V,C)}
S(V,C_1,C_2)\coloneqq \sum_{\substack{\underline{c}\in\O^2_{\text{prim}} \\ |c_i|\leq \widehat{C}_i}}\sum_{\substack{|\bm{
v}-\bm{v}_0|<\widehat{V}\\ F_1^*(\bm{v})=0}}|S_{d\underline{c},s, \bm{b},N}(\bm{v})|,
\end{equation}
where $F_1^*$ is the dual form of $F_1$ that we already met in Section 6. Note that upon replacing $\bm{v}_0$ with the nearest integer vector, we can assume without loss of generality that $\bm{v}_0\in\O^n$. 
Our basic strategy is to relate $S_{d\underline{c}, s, \bm{b},N}(\bm{v})$ to a point-counting problem and gain savings when summing this problem over $\bm{v}$ and $\underline{c}$ first. For this let us write $s=r'\tilde{s}$ into coprime $r', \tilde{s}\in \O$ with 
\[
r'\coloneqq \prod_{\nu_\varpi(s)\geq \nu_\varpi(N)+3}\varpi^{\nu_\varpi(s)}.
\]
Note that $r'$ is cube-full and $\varpi \mid \tilde{s}$ implies $\nu_\varpi(N)\geq 1$ since $s$ is cube-full. In particular, we have $|\tilde{s}|\leq |N|^3$ and thus by Lemma \ref{Le: MultiS(v)} that
\begin{equation}\label{Eq: UpperBound Sdc(v).step1}
|S_{d\underline{c},s, \bm{b},N}(\bm{v})|\leq |N|^{3(n+2)}|S_{d'\underline{c},r',\bm{b},N'}(t\bm{v})|
\end{equation}
for some $N'\mid r'$, $d'\mid r'$ and $t\in (\O/(r'N'/(r',N'))\O)^\times$. Next we write $r'=r(r',N')$ and make the change of variables $\bm{x}=\bm{y}N'+\bm{b}$, so that 
\begin{equation}\label{Eq: RelateS.to.T}
S_{d'\underline{c}, r', \bm{b}, N'}(t\bm{v})=\psi\left(\frac{-\bm{b}\cdot \bm{v}}{r_{N'}}\right)\sum_{\underline{a}/r'\in L(d'\underline{c})}T(\underline{a},r, t\bm{v}). 
\end{equation}
Let us now further write $r=e^2f$, where $f\mid e$ and 
\begin{equation}\label{Eq: Defe^2f}
f\coloneqq \prod_{2\nmid \nu_\varpi(r)}\varpi. 
\end{equation}
Our first step is to deduce to a congruence condition for $\bm{v}$ from the sum $T(\underline{a}, r ,t\bm{v})$. This will be achieved in the next lemma.
\begin{lemma}\label{Le: T(a,c)pointcount}
    Let $r\in \O$ be cube-full, $\underline{a}\in \O$, $\bm{v}\in \O^n$ and $r=e^2f$ with $f\mid e$ and $f$ given by \eqref{Eq: Defe^2f}. Then 
    \[
    T(\underline{a}, r, \bm{v})= |e|^n\sum_{\substack{|\bm{y}|<|ef|\\ \nabla (\underline{a}\cdot \underline{F})(\bm{y})\equiv \bm{v} \:(e)}}\psi\left(\frac{\underline{a}\cdot\underline{F}(\bm{y})-\bm{v}\cdot \bm{y}}{r}\right).
    \]
\end{lemma}
\begin{proof}
    Let us write $\bm{x}=\bm{y}+ef \bm{z}$ with $|\bm{y}|<|ef|$ and $|\bm{z}|<|e|$. Then Taylor's formula implies
    \begin{align*}
        T(\underline{a}, r, \bm{v})& = \sum_{|y|<|ef|}\sum_{|\bm{z}|<|e|}\psi\left(\frac{\underline{a}\cdot \underline{F}(\bm{y}+ef\bm{z})-\bm{v}(\bm{y}+ef \bm{z})}{r}\right)\\
        & = \sum_{|\bm{y}|<|ef|}\psi\left(\frac{\underline{a}\cdot \underline{F}(\bm{y})-\bm{v}\cdot \bm{y}}{r}\right)\sum_{|\bm{z}|<|e|}\psi\left(\frac{\bm{z}\cdot (\nabla (\underline{a}\cdot \underline{F})(\bm{y})-\bm{v})}{e}\right)\\
        &= |e|^n\sum_{\substack{|\bm{y}|<|ef|\\ \nabla (\underline{a}\cdot \underline{F})(\bm{y})\equiv \bm{v} \:(e)}}\psi\left(\frac{\underline{a}\cdot\underline{F}(\bm{y})-\bm{v}\cdot \bm{y}}{r}\right).
    \end{align*}
\end{proof}
Next we want to establish extra congruence conditions for $F_1(\bm{y})$ and $F_2(\bm{y})$ by considering the sum over $\underline{a}/r \in L(d\underline{c})$. This step underpins the first substantial deviation from the treatment of the averages of exponential sums in \cite{CubicHypersurfacesBV} and results in a significant complication of the argument. The reason for this extra difficulty is that in the setting of one polynomial the underlying exponential sum is a Ramanujan sum, whose behaviour is well understood, while in our case the orthogonality relations we obtain stem from the more involved structure of rational points on the lines $L(d\underline{c})$.\\

Before we begin our treatment, we make the following convention to ease notation. Whenever we have a sum of the form $\sum_{|\underline{a}_i|<|g_i|}'$, we understand $'$ to mean that $(\underline{a}_i,r_i')=1$. It then follows from  \eqref{Eq: UpperBound Sdc(v).step1} and \eqref{Eq: RelateS.to.T} combined with Lemma \ref{Le: T(a,c)pointcount} that 
\begin{equation}\label{Eq: S(V, C_1, C_2)Step1'}
    S(V,C_1,C_2) \leq |N|^{3(n+2)}|e|^n\sum_{\substack{\underline{c}\in\O^2_{\text{prim}} \\ |c_i|\leq \widehat{C}_i}}\sum_{\substack{|\bm{v}-\bm{v}_0|<\widehat{V} \\ F_1^*(\bm{v})=0 }}\sum_{|\bm{y}|<|ef|}\psi\left(\frac{-t\bm{v}\cdot \bm{y}}{r}\right)\sum_{\substack{\underline{a}/r'\in L(d'\underline{c})\\\nabla (\underline{a}\cdot \underline{F})(\bm{y})\equiv t\bm{v} \:(e)}}\psi\left(\frac{\underline{a}\cdot \underline{F}(\bm{y})}{r}\right).
\end{equation}
Our goal is now to investigate the sum 
\[
\Gamma(\bm{v}, \bm{y})\coloneqq \sum_{\substack{\underline{a}/r'\in L(d\underline{c}) \\ \nabla (\underline{a}\cdot \underline{F})(\bm{y})\equiv \bm{v} \:(e)}}\psi\left(\frac{\underline{a}\cdot\underline{F}(\bm{y})}{r}\right).
\]
for $r, r'\in \O$ with $r\mid r'$. To do so, let us write $(r',N)=kk'$, where $(k,k')=(r,k')=1$. Then we factor $r'=r'_1r'_2r'_3$ with pairwise coprime $r_i$'s, $(r'_1,d)=1$ and
    \begin{equation}\label{Eq: Def.r'}
    r'_2\coloneqq \prod_{\substack{\nu_\varpi(ek')\geq \nu_\varpi(r')-\nu_\varpi(d)+1\\ \nu_\varpi(d)>0}}\varpi^{\nu_\varpi(r')}.    
    \end{equation}
    Accordingly we shall also write $d=d_2d_3$, $e=e_1e_2e_3$, $f=f_1f_2f_3$, $k=k_1k_2k_3$, $k'=k'_1k'_2k'_3$ and $r=r_1r_2r_3$   with $d_i, e_i, f_i, k_i, k_i', r_i\mid r'_i$, so that $r_i=e_i^2f_i$ for $i=1,2,3$. Moreover, we let $d_3'$ be the maximal divisor of $d_3$ that divides $r_3$. In particular, we have $|d_3'|\asymp |d_3|$. The definition of $L(d\underline{c})$ implies that $\underline{a}/r'\in L(d\underline{c})$ if and only if $\varpi^{\nu_{\varpi}(r')-\nu_\varpi(d)}\parallel \underline{a}\cdot \underline {c}$ when $\nu_\varpi(d)\geq 1$ and $\varpi^{\nu_\varpi(r')}\mid \underline{a}\cdot\underline{c}$ when $\nu_\varpi(d)=0$ for all $\varpi \mid r'$. 

    Since $r_i\mid r_i'$, it is therefore clear that the sum we are investigating is multiplicative and accordingly we shall denote the sum corresponding to $r'_i$ by $S_i$ for $i=1,2,3$, so that $\Gamma(\bm{v},\bm{y})=S_1S_2S_3$. 
    
    We now treat each sum individually and start with $S_1$. When $|\underline{a}|<|r_1'|$, then by Lemma~\ref{Le: StructureRatLines} we have $\underline{a}/r'_1\in L(\underline{c})$ if and only if $\underline{a}\equiv a \underline{c}^\bot\: (r'_1)$ with $a\in (\O/r_1'\O)^\times$. Since $r_1'=r_1(r_1',N)$, we can write $a=a_1+e_1k_1'a_2$ with $|a_1|<|e_1k_1'|$ and $|a_2|<|r_1'||e_1k_1'|^{-1}=|e_1f_1k_1|$. From the definition of $e_1$ and $k_1'$ it is clear that $(a,r'_1)=1$ if and only if $(a_1,r'_1)=1$. Therefore, after splitting $a_2$ into residue classes modulo $e_1f_1$ and using the fact that $(k_1', e_1f_1)=1$, we have 
    \begin{align}
    \begin{split}\label{Eq: ProofOrthoL(dc)1}
        S_1 &= |k_1|\sideset{}{'}\sum_{\substack{|a_1|<|e_1k_1'|\\  a_1\nabla F_{\underline{c}}(\bm{y})\equiv \bm{v}\:(e_1)}} \psi\left(\frac{a_1 F_{\underline{c}}(\bm{y})}{r_1}\right)\sum_{a_2 \:(e_1f_1)}\psi\left(\frac{a_2k_1'F_{\underline{c}}(\bm{y})}{e_1f_1}\right)\\
        &= |k_1e_1f_1|\sideset{}{'}\sum_{\substack{|a_1|<|e_1k_1'| \\a_1\nabla F_{\underline{c}} (\bm{y})\equiv \bm{v}\:(e_1) \\ F_{\underline{c}}(\bm{y})\equiv 0 \:(e_1f_1)}} \psi\left(\frac{a_1 F_{\underline{c}}(\bm{y})}{r_1}\right).
        \end{split}
    \end{align}    
Next we deal with the sum $S_2$. As before we make the change of variables  ${\underline{a}=\underline{a}_1+k'_2e_2\underline{a}_2}$ with $|\underline{a}_1|<|e_2k'_2|$ and $|\underline{a}_2|<|e_2f_2k_2|$, so that $(\underline{a},r_2')=1$ if and only if $(\underline{a}_1, r_2')=1$. Moreover, it follows from the definition of $r_2$ that ${\nu_\varpi({r_2'})-\nu_\varpi(d_2)}=\nu_\varpi(\underline{a}\cdot \underline{c})$ if and only if ${\nu_\varpi({r_2'})-\nu_\varpi(d_2)}=\nu_\varpi(\underline{a}_1\cdot \underline{c})$, so that $\underline{a}/r_2'\in L(d_2\underline c)$ if and only if $\underline{a}_1/r_2'\in L(d_2\underline{c})$. We can again divide $\underline{a}_2$ into residue classes modulo $e_2f_2$ to obtain
\begin{align}
\begin{split}\label{Eq: ProofOrthoL(dc)2}
    S_2 & = |k_2'|^2\sum_{\substack{|\underline{a}_1|<|e_2k_2'|\\ \underline{a}_1/r_2'\in L(d_2\underline{c}) \\ \nabla(\underline{a}_1\cdot \underline{F})(\bm{y})\equiv \bm{v} \:(e_2)}}\psi\left(\frac{\underline{a}_1\cdot \underline{F}(\bm{y})}{r_2}\right)\sum_{\underline{a}_2 \:(e_2f_2)}\psi\left(\frac{k_2'\underline{a}_2\cdot \underline{F}(\bm{y})}{e_2f_2}\right)\\
    &=|e_2f_2k_2'|^2 \sum_{\substack{|\underline{a}_1|<|e_2k_2'|\\ \underline{a}_1/r_2'\in L(d_2\underline{c}) \\ \nabla(\underline{a}_1\cdot \underline{F})(\bm{y})\equiv \bm{v} \:(e_2)\\ F_1(\bm{y})\equiv F_2(\bm{y})\equiv 0 \:(e_2f_2)}}\psi\left(\frac{\underline{a}_1\cdot \underline{F}(\bm{y})}{r_2}\right),
    \end{split}
\end{align}
where we used that $(k_2',e_2f_2)=1$.\\

Finally, we begin our treatment of the sum $S_3$, which is slightly more involved. First, we introduce character sums to detect the condition $\nu_\varpi(r_3')-\nu_\varpi(d_3)=\nu_{\varpi}(\underline{a}\cdot\underline{c})$:
\begin{align*}
    S_3=&|r_3'^{-1}d_3|\sideset{}{'}\sum_{\substack{\underline{a}\: (r_3') \\ \nabla(\underline{a}\cdot\underline{F})(\bm{y})\equiv \bm{v} \:(e_3) }}\psi\left(\frac{\underline{a}\cdot \underline{F}(\bm{y})}{r_3}\right) \\&\times \prod_{\substack{\varpi^k\parallel r_3' \\\varpi^m\parallel d_3}}\left(\sum_{b_0 \:(\varpi^{k-m})}\psi\left(\frac{b_0\underline{a}\cdot\underline{c}}{\varpi^{k-m}}\right)-|\varpi|^{-1}\sum_{b_1 \:(\varpi^{k-m+1})}\psi\left(\frac{b_1 \underline{a}\cdot\underline{c}}{\varpi^{k-m+1}}\right)\right).
\end{align*}

Then we make the change of variables $\underline{a}=\underline{a}_1+e_3k'_3\underline{a}_2$, with $|\underline{a}_1|<|e_3k'_3|$ and $|\underline{a}_2|<|k_3e_3f_3|$. It follows from the definition of $e_3$ and $k_3$ that $(\underline{a},r_3')=1$ if and only if $(\underline{a}_1, r_3')=1$. Moreover, $r_3'$ was defined in such a way that $\nu_\varpi\:(e_3k_3')\leq k-m$ for all $\varpi\mid r_3'$. Slightly abusing notation, note that we have
\begin{align*}    \sum_{|\underline{a}_2|<|e_3f_3k_3|}\psi\left(\frac{k_3'\underline{a}_2 \cdot \underline{F}(\bm{y})}{e_3f_3}\right)\prod_{\substack{\varpi^k\parallel r_3' \\\varpi^m\parallel d_3}}\left(\sum_{b_0 \:(\varpi^{k-m})}\psi\left(\frac{b_0\underline{a}\cdot\underline{c}}{\varpi^{k-m}}\right)-|\varpi|^{-1}\sum_{b_1 \:(\varpi^{k-m+1})}\psi\left(\frac{b_1 \underline{a}\cdot\underline{c}}{\varpi^{k-m+1}}\right)\right)\\
    = \prod_{\substack{\varpi^k\parallel r_3'\\ \varpi^m\parallel d_3}}\left(\sum_{b_0 \:(\varpi^{k-m})}\psi\left(\frac{b_0\underline{a}_1\cdot \underline{c}}{\varpi^{k-m}}\right)S_0(\varpi)-|\varpi|^{-1}\sum_{b_1 \:(\varpi^{k-m+1})}\psi\left(\frac{b_1\underline{a}_1\cdot \underline{c}}{\varpi^{k-m+1}}\right)S_1(\varpi)\right),
\end{align*}
where 
\begin{align*}
S_i(\varpi) &\coloneqq \sum_{\underline{a}_2 \:(\varpi^{k-l})}\psi\left(\frac{k_3'\underline{a}_2\cdot (\underline{F}(\bm{y})+\varpi^{m-i}b_i\underline{c})}{\varpi^{k'-l}}\right)\\
&= |\varpi|^{2(k-l)}\delta_{-k_3' F_j(\bm{y})\equiv \varpi^{m-i}b_ic_j \:(\varpi^{k'-l})}
\end{align*}
for $i=0,1$ and where we temporarily wrote $l=\nu_\varpi\:(e_3k_3')$ and $k'=\nu_\varpi(r_3)$. Observe that for $\varpi\mid k_3'$ the sums $S_i(\varpi)$ are independent of $b_i$. For $\varpi \mid r_3$, it follows upon making the change of variables $b_i=b_i'+\varpi^{k'-m-l+i}b_i''$ with $|b_i'|<|\varpi|^{k'-m-l+i}$ and $|b_i''|<|\varpi|^{k-k'+l}$ that
\begin{align*}
    \sum_{b_i\: (\varpi^{k-m+i})}\psi\left(\frac{b_i\underline{a}_1\cdot \underline{c}}{\varpi^{k-m+i}}\right)S_i(\varpi)  = |\varpi|^{2(k-l)}&\sum_{\substack{b_i\:(\varpi^{k-m+i})\\ -k_3'F_j(\bm{y})\equiv \varpi^{m-i}b_ic_j \:(\varpi^{k'-l})}}\psi\left(\frac{b_i \underline{a}_1\cdot \underline{c}}{\varpi^{k-m+i}}\right)\\
    = |\varpi|^{2(k-l)}&\sum_{\substack{|b_i'|<|\varpi|^{k'-l-(m-i)}\\ -k_3'F_j(\bm{y})\equiv \varpi^{m-i}b_ic_j \:(\varpi^{k'-l})}}\psi\left(\frac{b_i'\underline{a}_1\cdot \underline{c}}{\varpi^{k-m+i}}\right)\\\times&\sum_{b_i''\:(\varpi^{k-k'+l})}\psi\left(\frac{b_i'' \underline{a}_1\cdot \underline{c}}{\varpi^{k-k'+l}}\right)\\
    = \frac{|\varpi|^{3k}}{|\varpi|^{l+k'}}&\delta_{\varpi^{k-k'+l}\mid \underline{a}_1\cdot \underline{c}}\sum_{\substack{|b_i'|<|\varpi|^{k'-l-(m-i)}\\ -k_3'F_j(\bm{y})\equiv \varpi^{m-i}b'_ic_j \:(\varpi^{k'-l})}}\psi\left(\frac{b_i'\underline{a}_1\cdot \underline{c}}{\varpi^{k-m+i}}\right)
\end{align*}
for $i=0,1$. Note that since $(c_1,c_2)=1$, we have $(c_i,\varpi)=1$ for $i=1$ or $i=2$. In particular, there is a unique $b_i'$ with $|b_i'|<|\varpi|^{k-l-m+i}$ and $F_j(\bm{y})\equiv \varpi^{m-i}b_i'k_3'c_j \:(\varpi^{k-l})$ for $j=1,2$. In addition, the latter equation implies $F_{\underline{c}}(\bm{y})\equiv 0 \:(\varpi^{k'-l})$ and $F_j(\bm{y})\equiv 0 \:(\varpi^m)$ for $j=1,2$. Using that $k'-l=\nu_\varpi(e_3f_3)$ for $\varpi\mid r_3$ and $k-k'=\nu_\varpi(k_3)$, we arrive at the identity
\begin{equation}\label{Eq: ProofOrthoL(dc)3} S_3=|d_3e_3f_3k_3^2||k_3'|^{-1}\delta_{F_1(\bm{y})\equiv F_2(\bm{y})\equiv 0 \:(d_3')}\delta_{F_{\underline{c}}(\bm{y})\equiv 0 \:(e_3f_3)}\sideset{}{'}\sum_{\substack{|\underline{a}_1|<|e_3k_3'|\\ \nabla (\underline{a}_1\cdot \underline{F})(\bm{y})\equiv \bm{v} \:(e_3) \\ \varpi^{\nu_\varpi(k_3e_3)}\parallel \underline{a}_1\cdot \underline{c}}}\psi\left(\frac{\underline{a}_1\cdot\underline{c}}{r_3}\right)\Pi(\underline{a}_1),
\end{equation}
where $d_3'$ is the maximal divisor of $d_3$ dividing $r_3$ and
\begin{align*}
\Pi(\underline{a}_1) &\coloneqq \prod_{\substack{\varpi^k\parallel r_3' \\ \varpi^m\parallel d_3 \\ \varpi \mid r_3}}\Biggl( \sum_{\substack{|b_0'|<|\varpi|^{k'-l-m})\\ F_i(\bm{y})\equiv \varpi^{m}b'_1c_i \:(\varpi^{k'-l})}}\psi\left(\frac{b_0'\underline{a}_1\cdot \underline{c}}{\varpi^{k-m}}\right)-|\varpi|^{-1}\sum_{\substack{|b_1'|<|\varpi|^{k'-l-(m-1)})\\ F_i(\bm{y})\equiv \varpi^{m-1}b'_2c_i \:(\varpi^{k'-l})}}\psi\left(\frac{b_1'\underline{a}_1\cdot \underline{c}}{\varpi^{k-m}}\right)\Biggr)\\
 &\times \prod_{\substack{\varpi^k\parallel r_3' \\ \varpi^m\parallel d_3 \\ \varpi \nmid r_3}}\Biggl(\sum_{b_0 \:(\varpi^{k-m})}\psi\left(\frac{b_0\underline{a}_1\cdot \underline{c}}{\varpi^{k-m}}\right)S_0(\varpi)-|\varpi|^{-1}\sum_{b_1 \:(\varpi^{k-m+1})}\psi\left(\frac{b_1\underline{a}_1\cdot \underline{c}}{\varpi^{k-m+1}}\right)S_1(\varpi)\Biggr).
\end{align*}
Moreover, since there is a unique $b_i'$ with $|b_i'|<|\varpi|^{k-l-m-i}$ and $F_i(\bm{y})\equiv \varpi^{m-i}b_i'k_3'c_i \:(\varpi^{k-l})$ it is easy to see that $\Pi(\underline{a}_1)\leq |k_3'|$. After relabelling the variables, \eqref{Eq: ProofOrthoL(dc)1}, \eqref{Eq: ProofOrthoL(dc)2} and \eqref{Eq: ProofOrthoL(dc)3} show that we have established the following result. 

\begin{lemma}\label{Le: L(dc)orthogonality}
Let $r'\in \O$ be cube-full, $d\in \O$ with $d \mid r$, $\bm{v}\in \O^n$. Define $r=r'/(r',M)$ and $r=e^2f$ with $f\mid e$ and $f$ given by \eqref{Eq: Defe^2f}. Then with the notation introduced in \eqref{Eq: Def.r'} for $\bm{y}, \bm{v}\in \O^n$, we have
\begin{align*}
&\sum_{\substack{\underline{a}/r'\in L(d\underline{c}) \\ \nabla (\underline{a}\cdot \underline{F})(\bm{y})\equiv \bm{v} \:(e)}}\psi\left(\frac{\underline{a}\cdot\underline{F}(\bm{y})}{r}\right) = |k_1e_1f_1||e_2f_2k_2'|^2|d_3e_3f_3k_3^2||k_3'|^{-1}\delta_{F_1(\bm{y})\equiv F_2(\bm{y})\equiv 0 \:(e_2f_2d_3')}\\ &\times\delta_{F_{\underline{c}}(\bm{y})\equiv 0 \:(e_1f_1e_3f_3)}\sideset{}{^{(1)}}\sum\psi\left(\frac{a_1F_{\underline{c}}(\bm{y})}{r_1}\right)\sideset{}{^{(2)}}\sum\psi\left(\frac{\underline{a}_2\cdot \underline{F}(\bm{y})}{r_2}\right)\sideset{}{^{(3)}}\sum\psi\left(\frac{\underline{a}_3\cdot \underline{c}}{r_3}\right)\Pi(\underline{a}_3)
\end{align*}

where $(1)$ indicates that we are summing over $|a_1|<|e_1k_1'|$ subject to $(a_1,r'_1)=1$, $a_1\nabla F_{\underline{c}} (\bm{y})\equiv \bm{v}\:(e_1)$; $(2)$ that $|\underline{a}_2|<|e_2k_2'|$ with $(\underline{a}_2, r_2')=1$, $\underline{a}_2/r_2'\in L(d_2\underline{c})$ and \break ${\nabla(\underline{a}_2\cdot \underline{F})(\bm{y})\equiv \bm{v} \:(e_2)}$; and  $(3)$ that $|\underline{a}_3|<|e_3k_3'|$ with $(\underline{a}_3, r_3')=1$, $\nabla (\underline{a}_3\cdot \underline{F})(\bm{y})\equiv \bm{v} \:(e_3)$ and $ \varpi^{\nu_\varpi(k_3e_3)}\parallel \underline{a}_3\cdot \underline{c}$. In addition, $|\Pi(\underline{a}_3)|\leq |k_3'|$.
\end{lemma}
Recall that $N, k_i, k'_i$ are all $O(1)$. In particular, once we combine Lemmas \ref{Le: T(a,c)pointcount} and \ref{Le: L(dc)orthogonality} 
 with the Chinese remainder theorem, we obtain from \eqref{Eq: S(V, C_1, C_2)Step1'} that
\begin{align}
\begin{split}\label{Eq: UppberBoundSdc(v).step2}
|S_{d\underline{c},s,\bm{b},N}(\bm{v})|\ll |e|^{n+1}|fe_2f_2d_3| \sideset{}{'}\sum_{|a_1|<|e_1k_1'|}\sideset{}{'}\sum_{\substack{|\underline{a}_2|<|e_2k'_2|\\ \underline{a}_2/r_2'\in L(d_2\underline{c})}} \sideset{}{'}\sum_{\substack{|\underline{a}_3|<|e_3k_3'| \\ \varpi^{\nu_\varpi(k_3e_3)}\parallel \underline{a}_3\cdot \underline{c}}}\left|\sideset{}{^{(4)}}\sum_{\phantom{Ji}|\bm{y}|<|ef|}\psi\left(\frac{\underline{a}\cdot \underline{F}(\bm{y})-t\bm{v}\cdot \bm{y}}{r}\right)\right|,
\end{split}
\end{align}
where $\underline{a}=a_1\underline{c}^\bot r_2r_3+\underline{a}_2r_1r_3+\underline{a}_3r_1r_2$ and $(4)$ denotes the conditions $\nabla(\underline{a}\cdot \underline{F})(\bm{y})\equiv t\bm{v} \:(e)$, $F_1(\bm{y})\equiv F_2(\bm{y})\equiv 0 \:(e_2f_2d_3')$ and $F_{\underline{c}}(\bm{y})\equiv 0 \:(e_1f_1e_3f_3)$. 

Recall that $r=e^2f$. Next we write $\bm{y}=\bm{y}_1+e\bm{y}_2$ with $|\bm{y}_1|<|e|$ and $|\bm{y}_2|<|f|$. Note that the definition of $r_3$ implies that $\nu_\varpi(d_3')\leq \nu_\varpi(e_3f_3)$. We shall therefore write 
\begin{equation}\label{Defi: e3'}
d_3'=e_3'f_3', \quad\text{where}\quad f'_3=\prod_{\nu_\varpi(d_3')=\nu_{\varpi}(e_3)+1}\varpi,
\end{equation}
so that $e_3'\mid e_3$ and $f_3'\mid f_3$. Hence $F_i(\bm{y})\equiv 0 \:(e_2f_2d_3')$ if and only if $F_i(\bm{y}_1)= e_2e_3'm_i$ say and $f_2f'_3 \mid (m_i+e/(e_2e_3')\bm{y}_2\cdot\nabla F_i(\bm{y}_1))$ for $i=1,2$. Similarly, if $F_{\underline{c}}(\bm{y}_1)=e_1e_3n$, then it must hold that $f_1f_3\mid (n+e_2 \bm{y}_2\cdot \nabla F_{\underline{c}}(\bm{y}_1))$. In addition, if $\nabla (\underline{a}\cdot \underline{F})(\bm{y}_1)=t\bm{v}+e\bm{k}$, then upon writing $\underline{a}=(a_1,a_2)$ we have
\[
\underline{a}\cdot \underline{F}(\bm{y}) - t\bm{v}\cdot \bm{y} \equiv \underline{a}\cdot \underline{F}(\bm{y}_1) - t\bm{v}\cdot \bm{y}_1 + e^2\left(a_1\bm{y}_1\cdot \nabla F_1(\bm{y}_2)+a_2F_2(\bm{y}_2) +\bm{y}_2\cdot \bm{k}\right) \: (r).
\]
It thus follows that
\begin{align}
\begin{split}\label{Eq: UpperBoundSdc(v).step3}
    \Biggl|\sideset{}{^{(4)}}\sum_{\phantom{Ji}|\bm{y}|<|ef|}\psi &\left(\frac{\underline{a}\cdot \underline{F}(\bm{y})-t\bm{v}\cdot \bm{y}}{r}\right)\Biggr| \\ &\leq \sum_{\substack{|\bm{y}_1|<|e| \\ e_2e_3' \mid F_i(\bm{y}_1), \: i=1,2 \\ e_1e_3\mid F_{\underline{c}}(\bm{y}_1) \\ \nabla(\underline{a}\cdot \underline{F})(\bm{y}_1) \equiv t\bm{v} \:(e)}}\max_{\bm{k}}\left|\sideset{}{^{(5)}}\sum_{\substack{\bm{y}_2 \:(f) }}\psi\left(\frac{a_1\bm{y}_1\cdot \nabla F_1(\bm{y}_2)+a_2F_2(\bm{y}_2) +\bm{y}_2\cdot \bm{k}}{f}\right)\right|
\end{split}
\end{align}
where $(5)$ denotes the conditions 
\[f_2f'_3\mid (m_i+e/(e_2e_3')\bm{y}_2\cdot\nabla F_i(\bm{y}_1))\quad\text{and}\quad f_1f_3\mid (n+e_2 \bm{y}_2\cdot \nabla F_{\underline{c}}(\bm{y}_1)).\]
By abuse of notation we denote the sum over $\bm{y}_2$ by $\Sigma^{(5)}$. We can then use orthogonality of characters to detect the congruence conditions in $(5)$. After employing the triangle inequality, a standard squaring and differencing argument delivers
\begin{align}
\begin{split}\label{Eq: UppberBoundSdc(v).step4}
   \sideset{}{^{(5)}}\sum & \leq |(f_2f'_3)^2f_1f_3|^{-1} \sum_{b_0 \:(f_1f_3)}\sum_{\underline{b}_2 \:(f_2f'_3)}\left|\sum_{\bm{y}_2 \:(f)}\psi\left(\frac{a_1\bm{y}_1\cdot \nabla F_1(\bm{y}_2) +a_2 F_2(\bm{y}_2) +\bm{y}_2\cdot \bm{k}'}{f}\right)\right|\\
    & \leq |f|^{n/2}N_f(\underline{a}, \bm{y}_1)^{1/2},
    \end{split}
\end{align}
where $\bm{k}'$ is a term that depends at most on $m_1, m_2, \bm{y}_1$ and the $e_i$'s, and 
\begin{equation*}   N_f(\underline{a},\bm{y})\coloneqq \#\{\bm{z} \:(e) \colon (a_1H(\bm{y})+a_2M)\bm{z}\equiv 0 \:(f)\}.
\end{equation*}
We now pause for a moment and collect what we have achieved so far. Inserting \eqref{Eq: UpperBoundSdc(v).step3} and \eqref{Eq: UppberBoundSdc(v).step4} into \eqref{Eq: UppberBoundSdc(v).step2}, we get
\begin{align}
\begin{split}\label{Eq: UpperBoundSdc(v).final}
    |S_{d\underline{c}, s, \bm{b},N}(\bm{v})|&\ll |e|^{n+1}|f|^{n/2+1}|e_2f_2d_3| \sideset{}{'}\sum_{|a_1|<|e_1k_1'|} \\ 
    &\times\sideset{}{'}\sum_{\substack{|\underline{a}_2|<|e_2k'_2|\\ \underline{a}_2/r_2'\in L(d_2\underline{c})}}\sideset{}{'}\sum_{\substack{|\underline{a}_3|<|e_3k_3'| \\ \varpi^{\nu_\varpi(k_3e_3)}\parallel \underline{a}_3\cdot \underline{c}}}\sum_{\substack{|\bm{y}_1|<|e| \\ e_2e_3' \mid F_i(\bm{y}_1), i=1,2 \\ e_1e_3\mid F_{\underline{c}}(\bm{y}) \\ \nabla(\underline{a}\cdot \underline{F})(\bm{y}_1) \equiv t\bm{v} \:(e)}}N_f(\underline{a},\bm{y})^{1/2}.
\end{split}
\end{align}
The innermost sum is clearly multiplicative, and so our next step is to focus on the sums 
\begin{align*}
S_{1}(e_1,f_1) &\coloneqq \sum_{\substack{\bm{y}\:(e_1) \\ F_{\underline{c}}(\bm{y}) \equiv 0 \:(e_1) \\ F_1(\bm{y})F_2(\bm{y})\equiv 0 \:(f_1)}}N_{f_1}(\underline{c}^\bot,\bm{y})^{1/2},\\
S'_{1}(e_1,f_1) &\coloneqq \sum_{\substack{\bm{y}\:(e_1) \\ F_{\underline{c}}(\bm{y}) \equiv 0 \:(e_1) \\ F_1(\bm{y})F_2(\bm{y})\not\equiv 0 \:(f_1)}}N_{f_1}(\underline{c}^\bot,\bm{y})^{1/2},\\
S_2(e_2,f_2) &\coloneqq \sum_{\substack{ \bm{y} \:(e_2) \\ F_1(\bm{y})\equiv F_2(\bm{y})\equiv 0 \:(e_2)}}N_{f_2}(\underline{a},\bm{y})^{1/2}\\
S_3(e_3,f_3) &\coloneqq \sum_{\substack{\bm{y} \:(e_3) \\ F_1(\bm{y})\equiv F_2(\bm{y})\equiv 0 \:(e_3') \\ F_{\underline{c}}(\bm{y})\equiv 0 \:(e_3)}} N_{f_3}(\underline{a},\bm{y})^{1/2}.
\end{align*}
We will establish sufficiently strong estimates for $S_i(e_i,f_i)$ when $i=1,2,3$, while we will obtain an additional saving by averaging $S'_1(e_1,f_1)$ over $\underline{c}$. Before we can provide upper bounds for them, we prove the following intermediate step.
\begin{lemma}\label{Le: N_i(r)}
    Let $r', r\in \O$ with $r' \mid r$ and $n\geq 13$.  Then we have
    \begin{align*}
        N_1(r)&\coloneqq \#\{ \bm{x} \: (r) \colon F_{\underline{c}}(\bm{x})\equiv 0 \: (r)\} \ll |r|^{n-1+\varepsilon},\\
        N_2(r) &\coloneqq \#\{ \bm{x}\: (r) \colon F_1(\bm{x})\equiv F_2(\bm{x})\equiv 0 \:(r)\}\ll |r|^{n-2+\varepsilon} \text{ and}\\
        N_3(r)&\coloneqq  \#\{\bm{x} \: (r) \colon F_1(\bm{x})\equiv F_2(\bm{x})\equiv 0 \:(r'), F_{\underline{c}}(\bm{x})\equiv 0 \:(r)\} \ll |r|^{n-1+\varepsilon}|r'|^{-1}.
    \end{align*}
\end{lemma}
\begin{proof}
All the quantities are mutliplicative by the Chinese remainder theorem, and so we may assume that $r=\varpi^k$ and $r'=\varpi^m$ with $m\leq k$ during the proof. Let us begin with the treatment of $N_3(\varpi^k)$ by detecting the congruence condition with character sums:
\begin{align*}
    |\varpi|^{2m+k}N_3(\varpi^k) & = \sum_{\underline{a} \: (\varpi^m)}\sum_{b\: (\varpi^k)}\sum_{\bm{x}\:(\varpi^k)}\psi\left(\frac{(\varpi^{k-m}\underline{a}+b\underline{c}^\bot)\cdot \underline{F}(\bm{x})}{\varpi^k}\right)
\end{align*}
Suppose now that $0\leq l \leq k-1$ is such that $\varpi^l\parallel \varpi^{k-m}\underline{a}+b\underline{c}^\bot$. Then we claim that the sum over $\bm{x}$ above is 
\[
|\varpi|^{ln}\sum_{\bm{x}\:(\varpi^{k-l})}\psi\left(\frac{\varpi^{-l}(\varpi^{k-m}\underline{a}+b\underline{c}^\bot)\cdot\underline{F}(\bm{x})}{\varpi^{k-l}}\right)\ll |\varpi|^{ln+5(k-l)n/6}.
\]
Indeed, if $\varpi^{1+v_\varpi}\mid \varpi^{-l}(\varpi^{k-m}a_1-bc_2)$, then the sum is $O(|\varpi|^{ln+(k-l)(n+1)/2})$ by Lemma~\ref{Le: ExpSumBada1}, while if $\varpi^{1+v_\varpi}\nmid \varpi^{-l}(\varpi^{k-m}a_1-bc_2)$, then we can apply Lemma \ref{Le: ExpSumAverageI} with $\bm{v_0}=\bm{0}$ and $\widehat{V}=1$ to obtain the claimed estimate. 

For $0\leq l \leq k$ fixed, let us now determine the number of triples $(a_1,a_2, b)$ such that $\varpi^{k-m}\underline{a}+b\underline{c}^\bot\equiv 0 \:(\varpi^l)$. If $l\leq k-m$, then this holds if and only if $\varpi^l\parallel b$ since $\underline{c}$ is primitive, so that the number of available $(a_1, a_2, b)$ is $O(|\varpi|^{2m+k-l})$. On the other hand, if $l>k-m$, then again because $\underline{c}$ is primitive, we can without loss of generality assume that  $(c_1, \varpi)=1$. This implies 
\[
b\equiv\varpi^{k-m}a_2c_1^{-1}\:(\varpi^l),
\]
which determines $b$ uniquely modulo $\varpi^{l}$. We thus also have
\[
\varpi^{k-m}a_1\equiv -c_2b \equiv \varpi^{k-m}a_2c_2c_1^{-1} \:(\varpi^l),
\]
which determines $a_1$ uniquely modulo $\varpi^{l-k+m}$ provided $a_2$ is given. In total we get that the number of such $(a_1,a_2, b)$ is at most $|\varpi|^{m+k-l+m-(l-k+m)}=|\varpi|^{2(k-l)+m}$. We conclude that 
\begin{align*}
|\varpi|^{2m+k}N_3(\varpi^k) \ll |\varpi|^{m+nk}+|\varpi|^{m+nk}\sum_{l=0}^{k-1}|\varpi|^{(n/6-2)(l-k)}\ll |\varpi|^{m+nk}+|\varpi|^{m+nk}
\end{align*}
for $n/6-2>0$. 

The proof of the statement for the quantities $N_1(r)$ and $N_2(r)$ runs along the same lines and is in fact less involved. The argument again relies on the estimates provided by Lemmas~\ref{Le: ExpSumBada1} and \ref{Le: ExpSumAverageI} and we do not provide details here.
\end{proof}
Before we continue with our study of the sums $S_i(e_i,f_i)$, we make some preliminary observations. First of all, if $(a_1, f)\ll 1$, then it follows from Equation (6.12) in \cite{CubicHypersurfacesBV} that 
\begin{equation}\label{Eq: HessSumf}
    \sum_{\bm{y} \:(f)}N_f(\underline{a}, \bm{y})\ll |f|^{n+\varepsilon}.
\end{equation}
Moreover, if $f \mid a_1$, then $\rk (M)\geq n-1$ readily implies $N_f(\underline{a},\bm{y})\ll |f|$.
\begin{lemma}\label{Le: Estimate.Si(e,f)}
    Let $e_i, f_i\in \O$ with $f_i\mid e_i$ and $f_i$ square-free for $i=1,2,3$. Then for $n\geq 13$, we have
    \begin{align*}
        S_{1}(e_1,f_1) &\ll |e_1|^{n-1+\varepsilon}|f_1|^{1/2},\\
        S_2(e_2,f_2) &\ll |e_2|^{n-2+\varepsilon}|f_2|\text{ and}\\
        S_3(e_3,f_3) &\ll |e_3|^{n-1+\varepsilon}|e_3'|^{-1}|f_3|.
    \end{align*}
\end{lemma}
\begin{proof}
All of the sums under consideration are multiplicative, and so we only have to prove the corresponding estimates when $e_i=\varpi^k$ and $f_i=1,\varpi$. Moreover, we shall write \break $m=\nu_{\varpi}(e'_3)$, so that $k\geq m \geq 1$.

Let $X_\varpi$ be the reduction of $V(F_1,F_2)$ modulo $\varpi$. When $i=2,3$, we begin with the case when $f_i=1$ or $a_1\equiv 0 \: (\varpi)$, while when $i=1$ we assume $f_i=1$ or $c_2\equiv 0 \:(\varpi)$. Since ${N_\varpi((0, a_2), \bm{y})\ll |\varpi|}$, in our situation we thus see that 
\[
S_i(\varpi^k,f_i) \ll |f_i|^{1/2}N_i(\varpi^k).
\]
Lemma \ref{Le: N_i(r)} provides estimates for $N_i(\varpi^k)$ that are satisfactory for the statement of the lemma.

Moreover, when $c_1\equiv 0 \:(\varpi)$, then it follows from the second display after Equation (6.15) in \cite{CubicHypersurfacesBV} that $S_1(\varpi^k,f_1)\ll \varpi^{k(n-1)}|f_1|^{1/2}$, which is also sufficient.

We may therefore assume that $f_i=\varpi$ from now on. When $i=1$, we are left with the case $(c_2, \varpi)=(c_1,\varpi)=1$, while for $i=2,3$ we have to deal with the case when $(a_1, \varpi)=1$. Let us first assume that $X_\varpi$, the reduction of $V(F_1,F_2)$ modulo $\varpi$, is singular. Since this can only happen for at most finitely $\varpi$'s, it must hold that $N_\varpi(\underline{a}, \bm{y})\ll 1$. In particular, we obtain 
\[
S_i(\varpi^k, f_i) \ll N_i(\varpi)
\]
in this case, which is again satisfactory by Lemma \ref{Le: N_i(r)}.

So let us now assume that $X_\varpi$ is non-singular. We first provide an upper bound for the contribution from $\bm{y} \not\equiv \bm{0}\:(\varpi^k)$ to $S_i(\varpi^k, \varpi)$. If $F_{\underline{c}}(\bm{y})\equiv 0 \:(\varpi^l)$ and $F_1(\bm{y})F_2(\bm{y})\equiv 0 \: (\varpi)$, then $F_{\underline{c}}(\bm{y}+\varpi^l\bm{z})\equiv 0 \: (\varpi^{l+1})$ if and only if $\varpi^{-l}F_{\underline{c}}(\bm{y})\equiv - \bm{z}\cdot\nabla F_{{\underline{c}}}(\bm{y})\: (\varpi)$. Since $c_1c_2\not\equiv 0 \:(\varpi)$, the condition $F_1(\bm{y})F_2(\bm{y})\equiv 0 \:(\varpi)$ forces that $\nabla F_{\underline{c}}(\bm{y})\not\equiv \bm{0} \:(\varpi)$ as otherwise $\bm{y}$ would be a singular point of $X_\varpi$. In particular for $\bm{y}\not\equiv \bm{0} \:(\varpi)$, inductively we obtain
\[
M_1(\bm{y})\coloneqq \#\{\bm{z}\: (\varpi^k)\colon \bm{z}\equiv \bm{y} \: (\varpi), F_{\underline{c}}(\bm{z})\equiv 0 \:(\varpi^k)\}\ll |\varpi|^{(k-1)(n-1)}.
\]
Similarly, if $F_1(\bm{y})\equiv F_2(\bm{y})\equiv 0 \:(\varpi^l)$, then $F_1(\bm{y}+\varpi^l\bm{z})\equiv F_1(\bm{y}+\varpi^l\bm{z}) \equiv 0 \:(\varpi^{l+1})$ holds if and only if $\varpi^{-l}F_i(\bm{y})\equiv - \bm{z}\cdot \nabla F_i(\bm{y}) \:(\varpi)$. As $X_\varpi$ is non-singular and $\bm{y}\not\equiv  \bm{0} \:(\varpi)$, we must have $\rk(\nabla F_1(\bm{y}), \nabla F_2(\bm{y}))=2$. Therefore, it follows by induction that
\[
M_2(\bm{y})\coloneqq \#\{\bm{z} \:(\varpi^k) \colon \bm{z}\equiv \bm{y} \:(\varpi), F_1(\bm{z})\equiv F_2(\bm{z})\equiv 0 \:(\varpi^k)\} \ll |\varpi|^{(k-1)(n-2)}.
\]
Finally, if $F_1(\bm{y})\equiv F_2(\bm{y})\equiv 0 \:(\varpi^m)$, then $F_{\underline{c}}(\bm{y}+\varpi^m \bm{z})\equiv 0 \:(\varpi^{m+1})$ if and only if \break $\varpi^{-m}F_{\underline{c}}(\bm{y})\equiv - \bm{z}\cdot \nabla F_{\underline{c}}(\bm{y}) \:(\varpi)$. As $F_1(\bm{y})\equiv F_2(\bm{y}) \equiv 0 \:(\varpi)$, we must have $\nabla F_{\bm{c}}(\bm{y})\not\equiv \bm{0} \:(\varpi)$, as otherwise $\bm{y}$ would be a singular point of $X_\varpi$. Combing this with the arguments that were used to estimate $M_1(\bm{y})$ and $M_2(\bm{y})$, we obtain 
\begin{align*}
M_3(\bm{y}) &\coloneqq \{\bm{z} \:(\varpi^k)\colon \bm{z}\equiv \bm{y}, F_1(\bm{z})\equiv F_2(\bm{z})\equiv 0 \:(\varpi^m), F_{\bm{c}}(\bm{z})\equiv 0 \:(\varpi^k)\}\\&\ll |\varpi|^{(m-1)(n-2) +(k-m)(n-1)}.
\end{align*}
It now follows from an application of the Cauchy-Schwarz inequality  that the contribution from $\bm{y}\not\equiv 0 \:(\varpi)$ to $S_i(\varpi^k,\varpi)$ is at most 
\begin{align*}
   \max_{\bm{y}\not\equiv 0 \:(\varpi)}M_i(\bm{y})N_i(\varpi)^{1/2} \huge(\sum_{\bm{y}\:(\varpi)}N_\varpi(\underline{a},\bm{y})\large)^{1/2}\ll |\varpi|^{n/2}\max_{\bm{y}\not\equiv 0 \:(\varpi)}M_i(\bm{y})N_i(\varpi)^{1/2}
\end{align*}
by \eqref{Eq: HessSumf}. Combining the estimates we just provided for $M_i(\bm{y})$ with Lemma \ref{Le: N_i(r)} to bound $N_i(\varpi)$, we see that this contribution is sufficient for the conclusion of the lemma to hold. 

We are thus left with estimating the contribution from $\bm{y}\equiv \bm{0} \:(\varpi)$ to $S_i(\varpi^k,\varpi)$. In this case we have $N_\varpi(\underline{a}, \bm{y})\ll |\varpi|$, so that 
\[
S_i(\varpi^k,\varpi)\ll |\varpi|^{1/2}N_i(\varpi^k),
\]
which is again satisfactory by Lemma \ref{Le: N_i(r)}.
\end{proof}
We now return to the main task of this section: estimating the quantity $S(V,C_1,C_2)$ that was defined in \eqref{Defi: S(V,C)}. Equation \eqref{Eq: UpperBoundSdc(v).final} gives
\begin{align}
\begin{split}\label{Eq: FirstUpperBoundS(V,W)}
S(V,C_1,C_2) \ll|e|^{n+1}|f|^{n/2+1}|e_2f_2d_3|&\sum_{\substack{\underline{c}\in \O^2_{\text{prim}}\\ |c_i|\leq \widehat{C}_i}}\sideset{}{'}\sum_{|a_1|<|e_1k_1'|}\sideset{}{'}\sum_{\substack{|\underline{a}_1|<|e_2k'|\\ \underline{a}_1/r_2'\in L(d_2\underline{c})}}\sideset{}{'}\sum_{\substack{|\underline{a}_1'|<|e_3k_3'| \\ \varpi^{\nu_\varpi(k_3e_3)}\parallel \underline{a}_1\cdot \underline{c}}}\\
&\times\sum_{\substack{\bm{y} \:(e) }}N_f(\underline{a},\bm{y})^{1/2}\sum_{\substack{|\bm{v}-\bm{v}_0|\leq \widehat{V}\\ F_1^*(\bm{v})=0\\ \nabla (\underline{a}\cdot \underline{F})(\bm{y})\equiv t\bm{v} \: (e)}}1.
\end{split}
\end{align}
Let us now write $\bm{v}=\bm{v}_0+\bm{v}_1+\bm{v}_2e$, where $|\bm{v}_1|<|e|$, $t\bm{v}_1\equiv \nabla (\underline{a}\cdot \underline{F})(\bm{y}) \:(e)$ and $|\bm{v}_2|<\widehat{V}|e|^{-1}$. Observe that $(e,t)=1$ implies that $\bm{v}_1$ is unique. Note that $G(\bm{v}_2)=F_1^*(\bm{v}_0+\bm{v}_1+\bm{v}_2e)$ is of degree $3\times 2^{n-2}$ and its leading degree part is absolutely irreducible, so that we may invoke Lemma \ref{Th: DimensionGrowth} to deduce that 
\begin{equation}\label{Eq: Estimate.sum.over.v}
    \sum_{\substack{|\bm{v}-\bm{v}_0|\leq \widehat{V}\\ F_1^*(\bm{v})=0\\ \nabla (\underline{a}\cdot \underline{F})(\bm{y})\equiv t\bm{v} \: (e)}}1 \ll 1+\left(\frac{\widehat{V}}{|e|}\right)^{n-2}.
\end{equation}
Using the Chinese remainder together with Lemma \ref{Le: Estimate.Si(e,f)}, we obtain 
\begin{align}
    \begin{split}
        \sum_{\bm{y}\: (e)}N_f(\underline{a}, \bm{y})^{1/2} & = S_2(e_2,f_2)S_3(e_3,f_3)\left(S_1(e_1,f_1)+S_1'(e_1,f_1)\right)\\
        &\ll |e_2|^{n-2+\varepsilon}|f_2||e_3|^{n-1+\varepsilon}|e_3'|^{-1}|f_3|\left(|e_1|^{n-1+\varepsilon}|f_1|^{1/2}+ S_1'(e_1,f_1)\right).
    \end{split}
\end{align}
Moreover, the conjunction of the conditions $F_1(\bm{y})F_2(\bm{y})\not\equiv 0 \:(f_1)$ and $F_{\underline{c}}(\bm{y})\equiv 0 \:(f_1)$ can only hold if $(c_1,f_1)=(c_2,f_1)=1$. Therefore, for $C_1\geq C_2\geq 1$ we have 
\begin{align}
    \begin{split}        \sum_{\substack{\underline{c}\in \O^2_{\text{prim}} \\ |c_i|\leq \widehat{C}_i}}S'_1(e_1,f_1)& \leq \max_{\substack{\underline{a}\: (f_1) \\ (a_1,f_1)=(a_2,f_1)=1}}\sum_{\substack{\bm{y}\: (e_1) \\ F_1(\bm{y})F_2(\bm{y})\not\equiv 0 (f_1)}}N_{f_1}(\underline{a}, \bm{y})^{1/2}\sum_{\substack{ \underline{c}\in \O^2_{\text{prim}}\\ |c_i|\leq \widehat{C}_i \\ F_{\underline{c}}(\bm{y})\equiv 0 \:(e_1) }}1\\
        &\leq \widehat{C}_2\left(1+\frac{\widehat{C}_1}{|e_1|}\right)\max_{\substack{\underline{a}\: (f_1) \\ (a_1,f_1)=(a_2,f_1)=1}}\left(\frac{|e_1|}{|f_1|}\right)^n\sum_{\substack{\bm{y}\: (f_1)}}N_{f_1}(\underline{a}, \bm{y})^{1/2}\\
        &\ll |e_1|^n\widehat{C}_2\left(1+\frac{\widehat{C}_1}{|e_1|}\right),
    \end{split}
\end{align}
where we used \eqref{Eq: HessSumf} together with the Cauchy-Schwarz inequality to arrive at the last estimate.

Next observe that $\underline{a}_2/r_2'\in L(d_2\underline{c})$ implies that $\varpi^{\nu_\varpi(r_2'd_2^{-1})}\mid \underline{a}_2\cdot \underline{c}$, so that that 
\begin{equation}\label{Eq: available.a_2}
\#\{|\underline{a}_2|<|e_2k'|\colon \underline{a}_2/r_2' \in L(d_2\underline{c})\}\leq \left(\frac{|e_2k'|}{|r_2'd_2^{-1}|}\right)^{2}|r_2'd_2^{-1}| = |e_2k'|^2|r_2'|^{-1}|d_2|.
\end{equation}
A similar argument delivers 
\begin{equation}\label{Eq: available.a_3}
\#\{|\underline{a}_3|<|e_3k'_3|\colon \varpi^{\nu_\varpi(e_3)}\mid \underline{a}_3\cdot \underline{c}\} \leq |k_3'|^2|e_3|.
\end{equation}
Recall that $e=e_1e_2e_3$, $f=f_1f_2f_3$ and $d=d_2d_3$. Then since $k, k'$ are $O(1)$ and $|s|\asymp |r|$, we can combine \eqref{Eq: Estimate.sum.over.v} -- \eqref{Eq: available.a_3} with \eqref{Eq: FirstUpperBoundS(V,W)} to obtain 
\begin{align*}
    S(V, C_1, C_2)  \ll &|e|^{n+1}|f|^{n/2+1}|e_1e_3|^2|e_2||f_1|^{1/2}|f_2|^2|f_3||d_2d_3||r_2'|^{-1}|e_3'|^{-1}\left(|e|+\widehat{V}\right)^{n-2}\\
    &\times \widehat{C}_2\left(\widehat{C}_1+|e_1||f_1|^{-1/2}+\widehat{C}_1|f_1|^{-1/2}\right)\\
\ll& |d||s|^{(n+3)/2}\frac{|f_2^3f_3|^{1/2}}{|r_2'e_3'|}\widehat{C}_2\left(|e|+\widehat{V}\right)^{n-2}\left(\widehat{C}_1+|e_1||f_1|^{-1/2}\right),
\end{align*}
where we used that $|e^2f|\asymp |s|$ and $d_2d_3=d$. Since $r_2$ is cube-full and $|r_2'|\asymp |r_2|$, we have $f_2^3 \mid r_2$ and thus $|f_2|^{3/2}|r_2'|^{-1}\ll 1$. Moreover, since $\nu_\varpi(d_3')\geq 1$ for all $\varpi \mid e_3$, the definition of $e_3'$ in \eqref{Defi: e3'}
implies that $f_3\mid e'_3$ and hence $|f_3|^{1/2}|e_3'|\ll 1$. We have thus established the following result.
\begin{lemma}\label{Le: ExpSumAverage.dualFormVanish}
Let $s\in \O$ be cube-full and $d\mid s$. Then with the notation of \eqref{Eq: Def.r'}, it holds that
\begin{align*}   
S(V,C_1,C_2)\ll_{F_1,F_2,N}|d||s|^{(n+3)/2}\widehat{C}_2\left(|s|^{n/2-1}+\widehat{V}^{n-2}\right)\left(\widehat{C}_1+|e_1|f_1|^{-1/2}\right).
\end{align*}

\end{lemma}

\section{Return to the circle method}\label{Sec: Circle}
In this section we combine the estimates for the various exponential sums and integrals that we have produced so far to finish our treatment of $N(P$).  To ease of notation, for $d\in\O$  and $\underline{c}\in \O^2$  we abbreviate the properties $d$ monic, $\underline{c}$ primitive, $|dc_1|\leq \min\{\widehat{T}|r|^{1/2}, \lvert r \rvert\},$ $|dc_2|<\widehat{T}^{-1}|r|^{1/2}$ and $\max\{\widehat{R}_i|dc_i^\bot|\}\geq |r|$ by $P(d,\underline{c})$. In addition, throughout this section we shall assume $n\geq 26$. We will not make the dependence of the implied constants explicit anymore, but allow them to depend at most on $F_1, F_2, w$ and $N$ as well as on $\varepsilon$ if it appears in the inequality. Recall the decomposition of $N(P)$ in \eqref{Eq: Decomp.N(P)1}.

\subsection{The main term} We begin to carry out the analysis of $M(P)$. Since we chose \break $\widehat{R}_2\asymp \widehat{P}^{1/3}$ in \eqref{Eq: ParameterSize}, it follows from Corollary \ref{Cor: FaryMajor} that 
\begin{equation}\label{Eq: Decomp.M(P)}
    M(P)=\widehat{P}^n\sum_{\substack{|r|\ll \widehat{P}^{1/3}\\r\text{ monic}}}|r_N|^{-n}S_rK_r +\widehat{P}^n\sum_{\substack{ \widehat{P}^{1/3}\ll |r|\leq \widehat{R}_1\widehat{R}_2\\r \text{ monic}}}|r_N|^{-1}\sum_{\substack{d\underline{c} \colon  P(d,\underline{c})\\d\mid r}}S_{d\underline{c}, r, \bm{b},N}(\bm{0})K_r,
\end{equation}
where 
\[
S_r=\sideset{}{'}\sum_{\underline{a}\:(r)}\sum_{\substack{|\bm{x}|<|r_N|\\ \bm{x}\equiv \bm{b}\:(N)}}\psi\left(\frac{\underline{a}\cdot\underline{F}(\bm{x})}{r}\right)
\]
and
\[
K_r=\int_{|\theta_1|<\widehat{R}_1^{-1}|r|^{-1}}\int_{|\theta_2|<\widehat{R}_2^{-1}|r|^{-1}}I_{r_N}(\underline{\theta},\bm{0})\dd\underline{\theta}.
\]
It is a consequence of Proposition \ref{Prop: IntEstimate} that 
\begin{equation}\label{MajorInt}
    K_r \ll \widehat{P}^{-5+\varepsilon}
\end{equation}
and from Corollary \ref{Cor: ExpSumTheAverage} with $\bm{v}_0=\bm{0}$ and $\widehat{V}=1$ 
we deduce 
\begin{equation}\label{MajorExp}
S_{d\underline{c},r, \bm{b},N}(\bm{v})\ll |d||r|^{5n/6+3/2+\varepsilon}.
\end{equation}
Note that if $c_2\neq 0$, then since $\widehat{T}\asymp \widehat{P}^{1/2}$, the condition $|dc_2|<\widehat{T}^{-1}|r|^{1/2}$  can only hold if $|r|\gg \widehat{P}$. In particular, it is now easy to see that 
\begin{equation}\label{sum.dc}
\sum_{\substack{d\underline{c}\colon P(d,\underline{c}) \\ d \mid r}}|d| \ll |r|^{1+\varepsilon}.
\end{equation}
It follows from \eqref{MajorInt}--\eqref{sum.dc} that the rightmost term in \eqref{Eq: Decomp.M(P)} is of order 
\begin{align*}
    \widehat{P}^{n-5+\varepsilon}\sum_{\substack{ \widehat{P}^{1/3}\ll|r|\leq \widehat{R}_1\widehat{R}_2\\r \text{ monic}}}\frac{|r|^{5n/6+3/2}}{|r_N|^{n}}\sum_{\substack{d\underline{c}\: P(d,\underline{c})\\d\mid r}}|d| & \ll \widehat{P}^{n-5+\varepsilon}\sum_{\substack{ \widehat{P}^{1/3}\ll|r|\leq \widehat{R}_1\widehat{R}_2\\r \text{ monic}}}|r|^{5/2-n/6}\\
    &\ll \widehat{P}^{n-5-(7/2-n/6)/3+\varepsilon},
\end{align*}
where we used that that $7/2-n/6<0$ for $n>21$. Consequently, the contribution from this term is negligible and it remains to investigate the first term on the right hand side of \eqref{Eq: Decomp.M(P)}. \\

The first step we take is to analyse the integral $K_r$. Let $C>0$ be a fixed positive integer, whose exact value will be determined in due course. We then split up the integral $K_r$ into 
\begin{equation}\label{Eq: Decomp.K_r}
K_r=\int_{|\theta_1|<\widehat{C}^{-1}\widehat{P}^{-3}}\int_{|\theta_2|<\widehat{C}^{-1}\widehat{P}^{-2}}I_{r_N}(\underline{\theta}, \bm{0})\dd\underline{\theta}+\int_{\Xi}I_{r_N}(\underline{\theta},\bm{0})\dd\underline{\theta},
\end{equation}
where $\Xi$ is defined by
\[
\Xi\coloneqq \{\underline{\theta}\in\TT^2\colon |\theta_i|<\widehat{R}_i^{-1}|r|^{-1} \text{ for }i=1,2\text{ and }\widehat{C}^{-1}\widehat{P}^{-3}\leq |\theta_1|\text{ or }\widehat{C}^{-1}\widehat{P}^{-2}\leq |\theta_2|\}.
\]
Note that $\Xi$ is non-empty only if $|r|<\max_{i=1,2}\{\widehat{C}\widehat{R}_i^{-1}\widehat{P}^{4-i}\}$. Since $\widehat{R}_i^{-1}\widehat{P}^{4-i}\asymp \widehat{P}^{5/3}$ by \eqref{Eq: ParameterSize}, this will certainly hold for $|r|\ll \widehat{P}^{1/3}$ and $P$ sufficiently large. We will show that the second integral vanishes and produce a lower bound for the first one. Beginning with the former task, we have by \eqref{Eq: RelateIntegraltoJ} that the second integral is equal to
\[
\widehat{L}^{-n}\int_{\Xi}\int_{\TT^n}\psi\left(t^{3P}\theta_1G_1(\bm{x})+t^{2P}\theta_2G_2(\bm{x})\right)\dd\bm{x}\dd\underline{\theta},
\]
where $G_i(\bm{x})=F_i(\bm{x}_0+t^{-L}\bm{x})$ for $i=1,2$. 
Let $\Gamma = (\{1\}\times \TT) \cup (\TT\times \{1\})$ and define 
\[
\lambda =\min_{\underline{\gamma}\in \Gamma}|\gamma_1\nabla F_1(\bm{x}_0)+\gamma_2 \nabla F_2(\bm{x}_0)|.
\]
Observe that $0<\lambda<1$ since $\Gamma$ is compact and $\bm{x}_0$ a non-singular point of the variety defined by $F_1$ and $F_2$ with $|\bm{x}_0|<H_{\underline{F}}^{-1}$. To simplify notation, we write $\gamma_1=t^{3P}\theta_1$ and $\gamma_2=t^{2P}\theta_2$. Let now $\gamma\in \TT$ be such that $|\gamma|=|\underline{\gamma}|$. Then by the ultrametric property we have 
\[
|\gamma_1\nabla G_1(\bm{x})+\gamma_2\nabla G_2(\bm{x})|=\widehat{L}^{-1}|\gamma||\gamma_1\nabla F_1(\bm{x}_0)/\gamma +\gamma_2 \nabla F_2(\bm{x}_0)/\gamma |\geq \widehat{L}^{-1}|\gamma|\lambda
\]
provided $L$ is sufficiently large. Moreover, all higher partial derivatives of $\gamma_1G_1(\bm{x})+\gamma_2G_2(\bm{x})$ are of order $O(\widehat{L}^{-2}|\gamma|)$. By the second derivative test \cite[Lemma 2.5]{CubicHypersurfacesBV} we thus have $K_r = 0$ if $\widehat{L}^{-1}|\gamma|\lambda \geq 1.$ Since $|\gamma| \geq \widehat{C}^{-1}$, this can be ensured if we make the choice $\widehat{C}=\lambda \widehat{L}^{-1}$, which we henceforth assume. 

We proceed to investigate the first integral in \eqref{Eq: Decomp.K_r}. After making the change of variables $\gamma_i=t^{(4-i)P}\theta_i$, by \eqref{Eq: RelateIntegraltoJ} we have
\begin{align*}
K_r = \widehat{L}^{-n}\widehat{P}^{-5}\int_{|\underline{\gamma}|<\widehat{C}^{-1}}\int_{\TT^n}\psi\left(\gamma_1G_1(\bm{x})+\gamma_2G_2(\bm{x})\right)\dd\bm{x}\dd\underline{\gamma}.
\end{align*}
It is now clear that $K_r$ is in fact independent of $r$ and to emphasise this, we define
\[
\sigma_\infty\coloneqq \widehat{L}^{-n}\int_{|\underline{\gamma}|<\widehat{C}^{-1}}\int_{\TT^n}\psi\left(\gamma_1G_1(\bm{x})+\gamma_2G_2(\bm{x})\right)\dd\bm{x}\dd\underline{\gamma}. 
\]
The integral $\sigma_\infty$ is the singular integral associated to our counting problem and our next step is to show that $\sigma_\infty>0$. To do so, we exchange the order of integration and apply Lemma \ref{Le: OrthoInt} to deduce that 
\[
\sigma_\infty =\widehat{L}^{-n}\widehat{C}^{-2}\text{meas}\left(\{\bm{x}\in\TT^n\colon |G_i(\bm{x})|<\widehat{C}\text{ for }i=1,2\}\right).
\]
Using Taylor expansion and the fact that $F_i(\bm{x}_0)=0$ for $i=1,2$, it follows that
\[
G_1(\bm{x})=t^{-L}\bm{x}\cdot \nabla F_1(\bm{x}_0) +\frac{1}{2}t^{-2L}\bm{x}^tH(\bm{x}_0)\bm{x}+t^{-3L}F_1(\bm{x})
\]
and \[
G_2(\bm{x})=t^{-L}\bm{x}\cdot \nabla F_2(\bm{x}_0)+t^{-2L}F_2(\bm{x}).
\]
Provided $L$ is sufficiently large, we must then have 
\[
|G_1(\bm{x})|=\widehat{L}^{-1}|\bm{x}\cdot\nabla F_1(\bm{x}_0)|\quad\text{and}\quad |G_2(\bm{x})|=\widehat{L}^{-1}|\bm{x}\cdot \nabla F_2(\bm{x}_0)|. 
\]
However, if we recall that $\widehat{C}=\widehat{L}^{-1}\lambda$, then for $|\bm{x}|< \lambda$ it is clear that 
\[
\widehat{L}^{-1}|\bm{x}\cdot \nabla F_i(\bm{x}_0)|<\widehat{L}^{-1}\lambda = \widehat{C},
\]
so that $\sigma_\infty > \widehat{L}^{-n}\widehat{C}^{-2}\lambda^n$. We summarise our investigation of the integral $K_r$ in the following result.
\begin{lemma}\label{Le: SingIntegral}
    Let $|r|\ll \widehat{P}^{1/3}$. Then 
    \[
    K_r=\sigma_\infty \widehat{P}^{-5},
    \]
    where $\sigma_\infty >0$ depends only on the weight function $w$. 
\end{lemma}
It follows from Lemma \ref{Le: SingIntegral} and the upper bound provided after \eqref{Eq: Decomp.M(P)} that
\[
M(P)=\sigma_\infty \mathfrak{S}_{\bm{b},N}(1/3)\widehat{P}^{n-5}+O\left(\widehat{P}^{n-5-\delta'}\right),
\]
for some $\delta'>0$, where for $\Delta>0$ we have defined
\[
\mathfrak{S}_{\bm{b},N}(\Delta)\coloneqq \sum_{\substack{|r|\ll \widehat{P}^{\Delta} \\ r\text{ monic}}}|r_N|^{-N}S_r
\]
to be the truncated singular series associated to our counting problem. Let 
\[
\mathfrak{S}_{\bm{b},N}\coloneqq \sum_{r\text{ monic}}|r_N|^{-n}S_r
\]
be the completed singular series. It follows from Lemma \ref{Le: ExpSumAverageI} and Lemma \ref{Le: ExpSumBada1} that 
\[
S_r = \sideset{}{'}\sum_{\underline{a} \: (r)}T(\underline{a}, r, \bm{0}) \ll |r|^{2+5n/6+\varepsilon}
\]
so that $\mathfrak{S}_{\bm{b},N}$ converges absolutes for $n> 18$ and 
\[
\left|\mathfrak{S}_{\bm{b},N}(\Delta)-\mathfrak{S}_{\bm{b},N}\right| \ll \sum_{|r|>\widehat{P}^\Delta}|r|^{2-n/6+\varepsilon} \ll \widehat{P}^{\Delta(3-n/6+\varepsilon)}.
\]
It is a routine exercise to show that $\mathfrak{S}_{\bm{b},N}>0$ provided there exists $\bm{x}_\varpi \in \O^n_\varpi$ such that $F_1(\bm{x}_\varpi)=F_2(\bm{x}_\varpi)=0$ and $|\bm{b}-\bm{x}_\varpi|_\varpi < |N|_\varpi $ for all $\varpi$. In particular, we have established the following result. 
\begin{prop}
    For $n>18$ we have 
    \[
    M(P)=\sigma_\infty \mathfrak{S}_{\bm{b},N}\widehat{P}^{n-5}+O\left(\widehat{P}^{n-5-\delta''}\right)
    \]
    for some $\delta''>0$ with $\sigma_\infty>0$. Moreover, $\mathfrak{S}_{\bm{b},N}>0$ if for every $\varpi$ there exists $\bm{x}_\varpi \in \O^n_\varpi$ such that $F_1(\bm{x}_\varpi)=F_2(\bm{x}_\varpi)=0$ and $|\bm{b}-\bm{x}_\varpi|_\varpi < |N|_\varpi $.
\end{prop}

To prove Proposition \ref{Prop: TheProp} and hence also Theorems \ref{Th: TheTheorem} and \ref{Th: WA}, it remains to give a satisfactory upper bound for the error term $E_1(P)$ defined in \eqref{Eq: Decomp.N(P)1}. This will occupy the remainder of our work and makes use of the various estimates we have provided for the oscillatory integrals and exponentials sums under consideration. 
\subsection{Preparation for the error terms} We continue our investigation of the error term $E_1(P)$. Before doing so, we take some preliminary steps. Firstly, we shall fix the absolute value of $r$ to be $\widehat{Y}$ and of $\theta_i$ to be $\widehat{\Theta}_i$ for $i=1,2$ respectively, where
\begin{equation}
   1\leq Y\leq R_1+R_2, \quad\quad -9P+R_2 \leq \Theta_1 < -Y- R_1\quad\text{and}\quad -9P+R_1 \leq \Theta_2 < -Y-R_2.
\end{equation}
and neither of the conditions \eqref{Cond: thetasmall} nor \eqref{Cond: Thetalarge.rsmall} recorded after \eqref{Eq: Decomp.N(P)1} hold. Observe that by our choice of $R_1$ and $R_2$ in \eqref{Eq: ParameterSize} the number of admissible triples $(Y,\Theta_1, \Theta_2)$ is $O(\widehat{P}^\varepsilon)$.  Secondly, we treat separately the contribution from $c_2\neq 0$ and $c_2 =0$ and denote the contribution of such $r$'s, $\underline{\theta}$'s and $\underline{c}$'s to $E_1(P)$ by $E_{1,a}(Y, \Theta_1,\Theta_2)$ and $E_{1,b}(Y,\Theta_1, \Theta_2)$ respectively, so that 
\begin{equation}\label{Eq: DefE1a}
E_{1,a}(Y, \Theta_1, \Theta_2) = \widehat{P}^n\sum_{\substack{|r|= \widehat{Y}\\ r\text{ monic}}}|r_N|^{-n}\sum_{\substack{d \underline{c}: P(d\underline{c})\\ c_2\neq 0 \\ d\mid r}}\int_{|\theta_i|=\widehat{\Theta}_i}\sum_{\substack{\bm{v}\in \O^n\setminus\{\bm{0}\}}} S_{d\underline{c},r,\bm{b},N}(\bm{v})I_{r_{N}}(\underline{\theta},\bm{v})\dd\underline{\theta}
\end{equation}
and 
\begin{equation}\label{Eq: DefE1b}
E_{1,b}(Y,\Theta_1,\Theta_2) = \widehat{P}^n\sum_{\substack{|r|= \widehat{Y}\\ r\text{ monic}}}|r_N|^{-n}\sum_{\substack{|d|\leq \widehat{Y}^{1/2}P^{1/2} \\ d \mid r}}\int_{|\theta_i|=\widehat{\Theta}_i}\sum_{\substack{\bm{v}\in \O^n\setminus\{\bm{0}\}}} S_{d\underline{c}_0,r,\bm{b},N}(\bm{v})I_{r_{N}}(\underline{\theta},\bm{v})\dd\underline{\theta}
\end{equation}
where $\underline{c}_0=(1,0)$ since by our convention $\underline{c}$ with $c_2=0$ can only be primitive when $c_1=1$. If we can show that $E_{1,i}(Y, \Theta_1,\Theta_2)\ll \widehat{P}^{n-5-\kappa}$ for $i=a,b$ and some $\kappa>0$, then  since the number of admissible triples $(Y,\Theta_1,\Theta_2)$ is $O(\widehat{P}^\varepsilon)$, the same estimate will hold with a new choice of $\kappa$ for $E_{1}(P)$. Moreover, if we let $\widehat{Z}= \max\{1, \widehat{P}^3\widehat{\Theta}_1, \widehat{P}^2\widehat{\Theta}_2\}$, then  \eqref{Eq: TruncateInt} implies that the summation range of $\bm{v}$ in the definition of $E_{1, i}(Y,\Theta_1,\Theta_2)$ is empty unless
\begin{equation}\label{Eq: TruncateV}
    |\bm{v}|\ll\widehat{V},\quad\text{where} \quad \widehat{V}= \frac{\widehat{Y}\widehat{Z}}{\widehat{P}}.
\end{equation}
In particular, since $\bm{v}\neq \bm{0}$, it must also hold that 
\begin{equation}\label{Eq: LowerBoundYZ}
    \widehat{Y}\widehat{Z}\gg \widehat{P}.
\end{equation}
Finally, we use the convention that $\varepsilon>0$ is an arbitrarily small real number whose exact value may change from one appearance to the next.
\subsubsection{Treatment of $E_{1,a}(Y, \Theta_1,\Theta_2)$} Note that we must have
\begin{equation}\label{Eq: E1a.lowerBoundY}
 \widehat{Y}\gg \widehat{P}, 
\end{equation}
because as $c_2\neq 0$ and $\widehat{T}\asymp \widehat{P}^{1/2}$, the inequality $|dc_2|<\widehat{T}^{-1}\widehat{Y}^{1/2}$ can only hold if $\widehat{Y}\gg \widehat{P}$. Since we assume that $c_2\neq 0$, Lemma \ref{Le: ExpSumSwgeneric} gives strong upper bounds for $S_{\underline{c}, r, \bm{b},N}(\bm{v})$ provided $r$ is square-free and $(r,F_1^*(\bm{v}))=1$. Accordingly, we shall further split up $E_{1,a}(Y, \Theta_1, \Theta_2)$ into the contribution from those $\bm{v}$ with $F_1^*(\bm{v})\neq 0$ and $F_1^*(\bm{v})=0$ and denote it by $E'_1$ and $E'_2$ respectively. In the treatment of $E_2'$ we compensate the worse exponential sum estimates compared to $E_1'$ by exploiting the sparsity of vectors $\bm{v}$ such that $F_1^*(\bm{v})=0$. \\

Let us begin by dealing with the term $E_1'$. Applying \eqref{Eq: RelateIntegraltoJ} and Lemma \ref{Le: IntVanishes} to the integral $I_{r_N}(\underline{\theta}, \bm{v})$ in \eqref{Eq: DefE1a}, it follows that 
\begin{equation}\label{Eq: E_1a(Y)}
    E_1' \leq |N|^n\widehat{L}^{-n} \frac{\widehat{P}^n}{\widehat{Y}^n}\sum_{\substack{d\underline{c}\colon P(d\underline{c})\\ c_2\neq 0}}\int_{|\underline{\theta}|=\underline{\widehat{\Theta}}} \int_{\TT^n}\sum_{\substack{|\bm{v}-\bm{v}_0|<\widehat{V}\widehat{Z}^{-1/2}\\ F_1^*(\bm{v})\neq 0}}\sum_{\substack{|r|=\widehat{Y}\\d\mid r}}|S_{d\underline{c}, r, \bm{b}, M}(\bm{v})|\dd \bm{x}\dd\underline{\theta},
\end{equation}
where $\bm{v}_0=-r_Nt^{L}(t^{3P} \nabla F_1(\bm{x}_0+t^{-L}\bm{x})+t^{2P}\nabla F_2(\bm{x}_0+t^{-L}\bm{x}))$.
Our next goal is to estimate the sum
\[
S\coloneqq \sum_{\substack{|\bm{v}-\bm{v}_0|<\widehat{V}\widehat{Z}^{-1/2}\\ F^*_1(\bm{v})\neq 0}}\sum_{\substack{|r|=\widehat{Y}\\ d\mid r}}|S_{d\underline{c},r,\bm{b},N}(\bm{v})|.
\]
For this we write $r=b_1b_1'b_2b_2'r_3$ into pairwise coprime $b_1,b_1', b_2, b_2', r_3\in \O$, where $b_1b_1'$ is the square-free part of $r$ satisfying $(b_1,dNF_1^*(\bm{v})c_2)=1$; $b_2b_2'$ is such that $\nu_\varpi(b_2b_2')=2$ for all $\varpi \mid b_2b_2'$ and $(b_2, dNc_2)=1$ and $r_3$ is the cube-full part of $r$. Accordingly we shall also write $d=d_1d_2d_3$ with $d_1\mid b_1'$, $d_2\mid b_2'$ and $d_3\mid r_3$. We can then use the multiplicativity of $S_{d\underline{c}, r, \bm{b}, N}(\bm{v})$ recorded in Lemma \ref{Le: MultiS(v)} to deduce for appropriate $t_1, t_1', t_2, t_2', t_3, N_1, N_2, N_3\in \O$ that 
\begin{align}
\begin{split}\label{Eq: E1a.ExpSums}
    |S_{d\underline{c}, r, \bm{b}, N}(\bm{v})| & = |S_{b_1}(t_1\bm{v})S_{d_1\underline{c}, b_1', \bm{b}, N_1}(t'_1\bm{v})S_{b_2}(t_2\bm{v})S_{d_2\underline{c}, b'_2, \bm{b}, N_2}(t'_2\bm{v})S_{d_3\underline{c}, r_3, \bm{b}, N_3}(t_3\bm{v})|\\
    &\ll |r|^{\varepsilon}|b_1|^{(n+1)/2}|b_2|^{n/2+1}|d_1d_2||b_1'b'_2|^{(n+3)/2}|S_{d_3\underline{c}, r_3, \bm{b}, N_3}(t_3\bm{v})|,
    \end{split}
\end{align}
where we used Lemmas \ref{Le: ExpSumSwgeneric}, \ref{Le: ExpSumSquarefreedc} and \ref{Le: ExpSumSquareGeneric} to estimate the sums corresponding to $b_1$, $b_1'$ and $b_2$ respectively and Corollary \ref{Cor: ExpSumsquarepartd} for the sum corresponding to $b_2'$. Moreover, by Corollary \ref{Cor: ExpSumTheAverage} we have 
\begin{equation}\label{Eq: E1a.Averagev}
    \sum_{\substack{|\bm{v}-\bm{v}_0|<\widehat{V}\widehat{Z}^{-1/2}\\ F^*_1(\bm{v})\neq 0}}|S_{d_3\underline{c}, r_3, \bm{b}, N_3}(t_3\bm{v})|\ll |d_3| |r_3|^{n/2+1+\varepsilon}|r_3''|^{1/2}\left(\widehat{V}^n\widehat{Z}^{-n/2}+|r_3|^{n/3}\right),
\end{equation}
where $r_3=r_3'r_3''$ with $(r_3',d)=1$ are defined in \eqref{Eq: DefinitionBadCubefull}. Consequently, plugging \eqref{Eq: E1a.ExpSums} and \eqref{Eq: E1a.Averagev} into the definition of $S$ yields
\begin{align*}
S &\ll \widehat{Y}^{(n+1)/2+\varepsilon} |d|\left(\widehat{V}^n\widehat{Z}^{-n/2}+\widehat{Y}^{n/3}\right)\sum_{\substack{|r|=\widehat{Y} \\ d\mid r}} |b_2r_3'|^{1/2}|b_1'b_2'r_3''|\\
& = \widehat{Y}^{(n+3)/2+\varepsilon} |d|\left(\widehat{V}^n\widehat{Z}^{-n/2}+\widehat{Y}^{n/3}\right)\sum_{\substack{|r|=\widehat{Y} \\ d\mid r}} |b_1|^{-1}|b_2r'_3|^{-1/2},
\end{align*}
since $|b_1b_1'b_2b_2'r_3'r_3''|=\widehat{Y}$. By definition, we must have $b_1'b_2'r_3''\mid (dNc_2F_1^*(\bm{v}))^\infty$. The number of available $b_1', b_2', r_3''$ with $|b_1'b_2'r_3''|\leq \widehat{Y}$ is hence $O(\widehat{Y}^\varepsilon)$. Moreover, since $b_2r_3'$ is square-full, the number of $b_2$ and $r_3'$'s of fixed absolute value $\widehat{B}$ is $O(\widehat{B}^{1/2}$). After summing over $q$-adic intervals it thus follows that
\[
S \ll \widehat{Y}^{(n+3)/2+\varepsilon}|d|\left(\widehat{V}^n\widehat{Z}^{-n/2}+\widehat{Y}^{n/3}\right)
\]
and hence 
\begin{align*} 
E_1' &\ll \widehat{P}^n\widehat{Y}^{3/2-n/2+\varepsilon}\widehat{\Theta}_1\widehat{\Theta}_2\sum_{d\underline{c}\colon P(d\underline{c})}|d|\left(\widehat{V}^n\widehat{Z}^{-n/2}+\widehat{Y}^{n/3}\right)\\
&\ll \widehat{P}^{n-5}\widehat{Y}^{5/2-n/2+\varepsilon}\widehat{Z}^2\left(\widehat{V}^n\widehat{Z}^{-n/2}+\widehat{Y}^{n/3}\right),
\end{align*}
where we used that $\widehat{\Theta}_1\widehat{\Theta}_2\ll \widehat{P}^{-5}\widehat{Z}^2$ and $\sum_{d\underline{c}\colon P(d\underline{c})}|d|\ll \widehat{Y}^{1+\varepsilon}$. If the first term in the brackets dominates, we get 
\begin{align*}
    E_1' \ll \widehat{P}^{-5}\widehat{Y}^{5/2+n/2+\varepsilon}\widehat{Z}^{2+n/2}
    \ll \widehat{P}^{5n/6-5 +10/3}\widehat{Y}^{1/2+\varepsilon}
    \ll \widehat{P}^{5n/6 - 5/3+5/6+\varepsilon},
\end{align*}
because $\widehat{Z}\ll \widehat{P}^{5/3}\widehat{Y}^{-1}$ and $\widehat{Y}\ll \widehat{P}^{5/3}$. Thus, the contribution from $E_1'$ in this case is $O(\widehat{P}^{n-5-\kappa})$ for some $\kappa>0$ as soon as $n>25$. If the second term dominates, then 
\begin{align*}
    E_1'\ll \widehat{P}^{n-5}\widehat{Y}^{5/2-n/6+\varepsilon}\widehat{Z}^2\ll \widehat{P}^{n-5+10/3}\widehat{Y}^{1/2-n/6+\varepsilon}\ll \widehat{P}^{n-5 +10/3 +1/2 -n/6+\varepsilon},
\end{align*}
where we used that $\widehat{Y}\gg \widehat{P}$ by \eqref{Eq: E1a.lowerBoundY}. This is satisfactory as soon as $n>23$, which completes our treatment of $E_1'$.\\

Next we consider the contribution from $E_2'$. This time we apply Lemma \ref{Le: IntEstimate} to the integral $I_{r_N}(\underline{\theta},\bm{v})$ and obtain
\[
E_2'\ll \frac{\widehat{P}^n}{\widehat{Y}^n}\widehat{Z}^{1-n/2}\widehat{\Theta}_1\widehat{\Theta}_2\sum_{|d|\leq \widehat{T}^{-1/2}|\widehat{Y}|^{1/2}}\sum_{\substack{|c_i|<\widehat{C}_i \\ c_2\neq 0}}\sum_{\substack{0<|\bm{v}|\leq \widehat{V}\\ F_1^*(\bm{v})=0}}\sum_{\substack{|r|=\widehat{Y}\\ d\mid r}}|S_{d\underline{c}, r, \bm{b}, N}(\bm{v})|,
\]
where $\widehat{C}_1=\widehat{T}\widehat{Y}^{1/2}|d|^{-1}$ and $\widehat{C}_2=\widehat{T}^{-1}\widehat{Y}^{1/2}|d|^{-1}$. We proceed to consider the sum 
\[
S' \coloneqq \sum_{\substack{|c_i|<\widehat{C}_i \\ c_2\neq 0}}\sum_{\substack{0<|\bm{v}|\leq \widehat{V}\\ F_1^*(\bm{v})=0}}\sum_{\substack{|r|=\widehat{Y}\\ d\mid r}}|S_{d\underline{c}, r, \bm{b}, N}(\bm{v})|.
\]
For this we first factor $r=b_1b_2r_3$ into pairwise coprime $b_1, b_2, r_3\in \O$ and $d=d_2d_3$ with $d_2\mid b_2$ and $d_3\mid r_3$, where $b_1$ is square-free, $b_2$ is cube-free, $(b_1,dNc_2)=1$ and $r_3$ is the cube-full part of $r$.  Parallel to our argument for $E_1'$ we use Lemma \ref{Le: MultiS(v)} to factor $S_{d\underline{c}, r, \bm{b}, N}(\bm{v})$ and invoke Lemmas~\ref{Le: ExpSumSquarefreedc} and \ref{Le: ExpSumSquareGeneric} to bound the sum corresponding to $b_1$ as well as Lemma \ref{Le: ExpSumSquarefreedc} and Corollary \ref{Cor: ExpSumsquarepartd} to bound the sum corresponding to $b_2$ to obtain
\[
S_{d\underline{c},r, \bm{b}, N}(\bm{v})\ll |r|^{\varepsilon}|d_2||b_1|^{n/2+1}|b_2|^{(n+3)/2}|S_{d_3\underline{c}, r_3, \bm{b}, N'}(t\bm{v})|
\]
for appropriate $t, N' \in \O$. We wish to apply Lemma \ref{Le: ExpSumAverage.dualFormVanish} to estimate the average 
\[
A\coloneqq \sum_{|c_i|< \widehat{C}_i}\sum_{\substack{0<|\bm{v}|\leq \widehat{V}\\ F_1^*(\bm{v})=0}}|S_{d_3\underline{c}, r_3, \bm{b}, N'}(t\bm{v})|.
\]
For this note that with the notation of Lemma \ref{Le: ExpSumAverage.dualFormVanish} for $P$ sufficiently large we have \break${|e_1|\leq \widehat{Y}^{1/2} \ll\widehat{Y}^{1/2}\widehat{T}|d|^{-1}=\widehat{C}_1}$, since $\widehat{T}\asymp \widehat{P}^{1/2}$ and $|d|\leq \widehat{Y}^{1/2}\widehat{T}^{-1} \ll \widehat{P}^{1/3}$. In particular, Lemma \ref{Le: ExpSumAverage.dualFormVanish} hands us 
\begin{align}
\begin{split}\label{Eq: E1b.A}
A &\ll |d_3||r_3|^{n/2+3/2+\varepsilon}\widehat{C}_1\widehat{C}_2\left(|r_3|^{n/2-1}+\widehat{V}^{n-2}\right)
\end{split}
\end{align}
We may also forget about the condition $F_1^*(\bm{v})=0$ and use Corollary \ref{Cor: ExpSumTheAverage} to obtain the alternative estimate 
\begin{equation}\label{Eq: E1b.A'}
    A\ll \widehat{C}_1\widehat{C}_2|d_3||r_3|^{n/2+1+\varepsilon}|r_3''|^{1/2}\left(\widehat{V}^n +|r_3|^{n/3}\right),
\end{equation}
where $r_3=r_3'r_3''$ with $r_3'$ given by \eqref{Eq: DefinitionBadCubefull}. In particular, we have shown so far that
\begin{equation}\label{Eq: EstimateA}
A \ll \widehat{C}_1\widehat{C}_2|d_3||r_3|^{n/2+1+\varepsilon}\min\left\{|r_3|^{1/2}\left(\widehat{V}^{n-2}+|r_3|^{n/2-1}\right), |r_3''|^{1/2}\left(\widehat{V}^n+|r_3|^{n/3}\right)\right\}.
\end{equation}
We begin with the contribution from $\widehat{V}\geq |r_3|^{1/2}$ or $\widehat{V}\leq |r_3|^{1/3}$. Since $b_2r_3''\mid (Ndc_2)^{\infty}$, there are at most $O(\widehat{P}^\varepsilon)$ pairs $(b_2, r_3'')$. Moreover, since $r_3$ is cube-full, there are  $O(|r_3|^{1/3})$ available $r_3$ of fixed absolute. One now easily derives
\[
\sum_{\substack{|r|=\widehat{Y}\\ d\mid r}}|b_2r_3|^{1/2} \ll \widehat{Y}^{1+\varepsilon}.
\]
In addition, we also have 
\begin{equation}\label{Eq: SumdC}
    \sum_{|d|\leq \widehat{T}^{-1}\widehat{Y}^{1/2}}\widehat{C}_1\widehat{C}_2|d| = \widehat{Y}\sum_{|d|\leq \widehat{T}^{-1}\widehat{Y}^{1/2}}|d|^{-1} \ll \widehat{Y}^{1+\varepsilon}.
\end{equation}
After employing the estimates $\widehat{\Theta}_1\widehat{\Theta}_2 \ll \widehat{P}^{-5}\widehat{Z}^2$ and $|r_3|\leq \widehat{Y}$, we have in the cases under consideration 
\begin{equation}
    E_2' \ll \widehat{P}^{n-5}\widehat{Y}^{3-n/2}\widehat{Z}^{3-n/2}\left(\widehat{V}^{n-2}+\widehat{Y}^{n/3}\right).
\end{equation}
If the first term dominates, we obtain 
\begin{align*}
    E_2'  \ll \widehat{P}^{-3}\widehat{Y}^{1+n/2+\varepsilon}\widehat{Z}^{1+n/2}\ll \widehat{P}^{5n/6-3+5/3+\varepsilon},
\end{align*}
which is satisfactory if $n>22$. On the other hand, if $\widehat{Y}^{n/3}\geq \widehat{V}^{n-2}$, then
\begin{align*}
    E_2'  \ll \widehat{P}^{n-5}\widehat{Y}^{3-n/6+\varepsilon}\widehat{Z}^{3-n/2}\ll \widehat{P}^{n-5+(3-n/6)+\varepsilon},
\end{align*}
by \eqref{Eq: E1a.lowerBoundY} and because $\widehat{Z}\geq 1$. Thus this contribution is sufficiently small as soon as $n>18$.

Finally, we have to deal with the contribution from $|r_3|^{1/3}<\widehat{V}<|r_3|^{1/2}$. For this, we use that for any  real numbers $A, B>0$ and $0\leq \kappa \leq 1$ that $\min\{A,B\}\leq A^{1-\kappa}B^\kappa$ with $\kappa=1/(n-2)$ to deduce that 
\begin{align*}
    S' & \ll \widehat{C}_1\widehat{C}_2 |d| \widehat{V}^{n(1-\kappa)}\widehat{Y}^{n/2+1+\varepsilon}\sum_{\substack{|r|}=\widehat{Y}}|b_2|^{1/2}|r_3''|^{(1-\kappa)/2}|r_3|^{\kappa(n-1)/2}\\
    &\ll \widehat{C}_1\widehat{C}_2 |d| \widehat{V}^{n(1-\kappa)}\widehat{Y}^{n/2+2+\varepsilon}\sum_{|b_2r_3'r_3''|\leq \widehat{Y}}|b_2|^{-1/2}|r_3'|^{(n-1)/(2n-4)-1}.
\end{align*}
The number of available $b_2$ of fixed absolute value is $O(|b_2|^{1/2})$. Moreover, there are at most $O(\widehat{Y}^{\varepsilon})$ possibilities for $r_3''$. Since $(n-1)/(2n-4)-1<-1/3$ for $n\geq 6$ and the number of $|r'_3|$ of fixed absolute value is $O(|r'_3|^{1/3})$, it follows that the sum above is $O(\widehat{Y}^\varepsilon)$. Therefore, we have 
\begin{equation}
    S' \ll \widehat{C}_1\widehat{C}_2 |d| \widehat{V}^{n(1-\kappa)}\widehat{Y}^{n/2+2+\varepsilon}.
\end{equation}
Therefore, the contribution to $E_2'$ is at most 
\begin{align*}
        \frac{\widehat{P}^n}{\widehat{Y}^n}\widehat{Z}^{1-n/2}\widehat{\Theta}_1\widehat{\Theta}_2 \sum_{d}S' & \ll \widehat{P}^{n-5/3+\varepsilon}\widehat{Y}^{1-n/2}\widehat{Z}^{1-n/2}\widehat{V}^{n(1-\kappa)}\\
        &= \widehat{P}^{-5/3 + n\kappa +\varepsilon}\widehat{Y}^{1+n/2-\kappa n}\widehat{Z}^{1+n/2-\kappa n}\\
        &\ll \widehat{P}^{5n/6 -2\kappa n /3 +\varepsilon},
\end{align*}
where we used $\widehat{\Theta}_1\widehat{\Theta}_2\ll \widehat{P}^{-5/3}\widehat{Y}^{-2}$ and \eqref{Eq: SumdC} to estimate the sum over $d$. One can check that $5n/6- 2n/(3n-6)<n-5$ provided $n\geq 26$, which completes our treatment of $E_2'$ and thus also of $E_{1,a}(Y, \Theta_1, \Theta_2)$. 
\subsubsection{Treatment of $E_{1,b}(Y,\Theta_1,\Theta_2)$} We differ our treatment according to the size of $Y$. When $\widehat{Y}\geq \widehat{P}^{1-\eta}$, where $\eta$ is as in \eqref{Cond: Thetalarge.rsmall} after \eqref{Eq: Decomp.N(P)1}, it is more efficient to estimate the sum over $\bm{v}$ via the same trick that we used to arrive at \eqref{Eq: E_1a(Y)}, whereas when $\widehat{Y}\leq \widehat{P}^{1-\eta}$, we estimate the integral directly. Before doing so, for $W\geq 1$ we focus on the sum 
\[
\Sigma(W)\coloneqq \sum_{\substack{|\bm{v}-\bm{v}_0|<\widehat{W}}}\sum_{|d|\leq \widehat{K}}\sum_{\substack{|r|=\widehat{Y} \\ d\mid r}}|S_{d\underline{c}_0, r, \bm{b}, N}(\bm{v})|,
\]
where $\widehat{K}=\min\{\widehat{Y}, \widehat{Y}^{1/2}\widehat{T}\}$.

To begin with, let us write $r=sr_3$, where $s$ is the cube-free part of $r$ and $r_3$ is cube-full. Moreover, we write $d=ef_1f_2^2d_3$ into pairwise coprime $e,f_1,f_2, d_3\in \O$, where $ef_1f_2$ is square-free, $e$ is the greatest common divisor of $s$ and the square-free part of $d$, $ef_1^2f_2^2\mid s$ and $d_3\mid r_3$. The multiplicativity of $S_{d\underline{c}_0, r, \bm{b},N}(\bm{v})$ together with Proposition~\ref{Prop: S(c,v)estimate.cbad} then imply 
\[
S_{d\underline{c}_0, r,\bm{b}, N}(\bm{v})\ll \widehat{Y}^{\varepsilon}|f_1f_2|^{1/2}|s|^{(n+3)/2}|S_{d_3\underline{c}_0,r_3, \bm{b}, N'}(t\bm{v})|,
\]
for some $N'\mid N$ and $t\in \O$ with $(t, r_3)=1$. Moreover, the number of available $s$ is $O(\widehat{Y}|e_1f_1^2f_2^2r_3|^{-1})$, so that   
\[
S\ll \widehat{Y}^{(n+5)/2+\varepsilon}\sum_{d}\sum_{\bm{v}}\sum_{d=ef_1f_2^2d_2}|e_1f_1|^{-1}|f_2|^{-1/2}\sum_{\substack{|r_3|\leq \widehat{Y} \\ d_3\mid r_3}}\frac{|S_{d_3\underline{c}_0,r_3, \bm{b}, N'}(t\bm{v})|}{|r_3|^{(n+5)/2}}.
\]
Next we factor $r_3$ into $d_3's_1s_2$ into pairwise coprime $d_3', s_1, s_2$, where $d_3'\mid d_3$, $(s_2, N'\Delta_{F_2})=1$ and $\nu_\varpi(s_2)>\nu_\varpi(d_3)$, so that $s_1\mid (N'\Delta_{F_2})^\infty$. Accordingly we shall also write $d_3=d_3'g_1g_2$, so that $g_i\mid s_i$ for $i=1,2$. Let $u\in \O$ and suppose that 
\[
u=\prod b_i^i, \quad (b_i,b_j)=1\text{ if }i\neq j, \quad b_i \text{ square-free}.
\]
We then define the function 
\[
m(u)\coloneqq b_1^3b_2^3\prod_{i\geq 3}b_i^{i+1}.
\]
Note that since $s_3$ is cube-full and $\nu_\varpi(s_2)>\nu_\varpi(g_2)$, we must then have $m(g_2)\mid s_2$. By Lemma \ref{Le: MultiS(v)} we then have for some $t_1, t_2 \in \O$ with $(t_1, d_3's_1)=(t_2,s_2)=1$ and $N''\mid N'$ that 
\[
S_{d_3\underline{c}_0, r_3,\bm{b}, N'}(t\bm{v})=S_{d'_3g_1\underline{c}_0, d_3's_1, \bm{b}, N''}(t_1\bm{v})S_{g_2\underline{c}_0, s_2, \bm{0}, 1}(t_2\bm{v}).
\]
Therefore, it follows from Proposition \ref{Prop: S(c,v)estimate.cbad} that 
\begin{align*}    \sum_{\substack{|r_3|\leq \widehat{Y} \\ d_3\mid r_3}}\frac{|S_{d_3\underline{c}_0,r_3, \bm{b}, N'}(t\bm{v})|}{|r_3|^{(n+5)/2}} & \ll \widehat{Y}^\varepsilon \sum_{d_3=d_3'g_1g_2}\sum_{\substack{|s_1|\leq \frac{\widehat{Y}}{|d_3'|}\\ g_1\mid s_1}}\frac{|S_{d'_3g_1\underline{c}_0, d_3's_1, \bm{b}, N''}(t_1\bm{v})|}{|d_3's_1|^{(n+5)/2}}\sum_{\substack{|s_2|\leq \widehat{Y}\\ m(g_2)\mid s_2}}|g_2||s_2|^{-1}\\
&\ll\widehat{Y}^\varepsilon \sum_{d_3=d_3'g_1g_2}|g_2||m(g_2)|^{-1}\sum_{\substack{|s_1|\leq \frac{\widehat{Y}}{|d_3'|}\\ g_1\mid s_1}}\frac{|S_{d'_3g_1\underline{c}_0, d_3's_1, \bm{b}, N''}(t_1\bm{v})|}{|d_3's_1|^{(n+5)/2}},
\end{align*}
where we used that there are at most $O(|s_2|^{1/3})$ available $s_2$ of fixed absolute value, because $s_2$ is cube-full. Next we change the order of summation and make the sum over $\bm{v}$ the innermost one. We then use Corollary \ref{Cor: ExpSumTheAverage} to deduce
\begin{align*}
\sum_{\substack{|s_1|\leq\frac{\widehat{Y}}{|d_3'|}\\ g\mid s_1}}\sum_{|\bm{v}-\bm{v}_0|<\widehat{W}}\frac{|S_{d'_3g_1\underline{c}_0, d_3's_1, \bm{b}, N''}(t_1\bm{v})|}{|d_3's_1|^{(n+5)/2}} &\ll|d_3'|^{-1/2}|g_1|\sum_{\substack{|s_1|\leq\frac{\widehat{Y}}{|d_3'|}\\ g_1\mid s_1}}|s_1|^{-1}\left(\widehat{W}^n+|d_3's_1|^{n/3}\right)\\
&\ll |d_3'|^{-1/2}\widehat{W}^n+|d_3'|^{-1/2}\widehat{Y}^{n/3+\varepsilon}.
\end{align*}
From what we have shown so far, it follows that 
\[
\Sigma(W)\ll \widehat{Y}^{(n+5)/2+\varepsilon}\sum_{|d|\leq \widehat{K}}\sum_{d=e_1f_1f_2^2d_3}\sum_{d_3=d_3'g_1g_2}|e_1f_1|^{-1}|f_2|^{-1/2}|g_2||m(g_2)|^{-1}|d_3'|^{-1/2}\left(\widehat{W}+\widehat{Y}^{n/3}\right).
\]
Let now $k\in \ZZ_{>0}$ be such that $1/k<\varepsilon$ and write 
\[
d=h'h_1h_2^2\cdots h_k^kh_{k+1}, \quad (h', h_i)=(h_i,h_j)=1, h'\mid (N\Delta_{F_2})^\infty \text{ for }i\neq j,
\]
where $h_i$ is square-free for $i=1,\dots, k$ and $h_{k+1}$ is $(k+1)$th-powerful. Recalling the definition of $m(g_2)$ and that $d_3'$ is cube-full, it is then not hard to see that 
\begin{align*}
    \sum_{|d|\leq \widehat{K}}\sum_{d=e_1f_1f_2^2d_3}\sum_{d_3=d_3'g_1g_2}|e_1f_1|^{-1}|f_2|^{-1/2}|g_2||m(g_2)|^{-1}|d_3'|^{-1/2} & \leq \sum_{|d|\leq \widehat{K}}|h_1\cdots h_k|^{-1}\sum_{d=ef_1f_2^2d_2}\sum_{d_3=d_3'g_1g_2}1\\
    &\ll \sum_{|d|\leq \widehat{K}}|d|^\varepsilon |h_1\cdots h_k|^{-1}\\
    &\ll \widehat{K}^{1/k+\varepsilon}.
\end{align*}
Therefore, the definition of $S$ and our choice of $k$ implies after redefining $\varepsilon$ that 
\begin{equation}\label{Eq: Estimate.Sigma(W)}
    \Sigma(W)\ll \widehat{Y}^{(n+5)/2+\varepsilon}\left(\widehat{W}+\widehat{Y}^{n/3}\right).
\end{equation}
We now apply \eqref{Eq: Estimate.Sigma(W)} to estimate $E_{1,b}(Y, \Theta_1, \Theta_2)$ in two different ways according to the size of $Y$. Let us begin by assuming that $\widehat{Y}\geq \widehat{P}^{1-\eta}$. In this case, we deduce from \eqref{Eq: RelateIntegraltoJ} and Lemma \ref{Le: IntVanishes} that 
\[
E_{1,b}(Y,\Theta_1,\Theta_2)\leq |N|^n\widehat{L}^{-n}\frac{\widehat{P}^n}{\widehat{Y}^n}\int_{|\underline{\theta}|=\underline{\widehat{\Theta}}}\int_{\TT^n}\sum_{\substack{|\bm{v}-\bm{v}_0|<\widehat{V}\widehat{Z}^{-1/2}}}\sum_{|d|\leq \widehat{K}}\sum_{\substack{|r|=\widehat{Y} \\ d\mid r}}|S_{d\underline{c}_0, r, \bm{b}, N}(\bm{v})|\dd\bm{x}\dd \underline{\theta},
\]
where $\bm{v}_0=-r_Mt^{L}(t^{3P} \nabla F_1(\bm{x}_0+t^{-L}\bm{x})+t^{2P}\nabla F_2(\bm{x}_0+t^{-L}\bm{x}))$.
We can now use \eqref{Eq: Estimate.Sigma(W)} with $\widehat{W}=\widehat{V}\widehat{Z}^{-1/2}$ to deduce that 
\begin{align*}
    E_{1,b}(Y,\Theta_1,\Theta_2)&\ll \widehat{P}^{n}\widehat{Y}^{5/2-n/2+\varepsilon}\widehat{\Theta}_1\widehat{\Theta}_2\left(\widehat{V}^n\widehat{Z}^{-n/2}+\widehat{Y}^{n/3}\right)\\
    &= \widehat{Y}^{5/2+n/2+\varepsilon}\widehat{\Theta}_1\widehat{\Theta}_2\widehat{Z}^{n/2}+\widehat{P}^{n}\widehat{Y}^{5/2-n/6+\varepsilon}\widehat{\Theta}_1\widehat{\Theta}_2\\
    &\ll \widehat{P}^{5n/6-5/3}\widehat{Y}^{1/2+\varepsilon}+\widehat{P}^{n-5/3+(1-\eta)(1/2-n/6+\varepsilon)}.
\end{align*}
The first term is $O(\widehat{P}^{5n/6-5/6+\varepsilon})$, which is satisfactory as soon as $n\geq 26$. Moreover, $n-5/3+(1/2-n/6)<n-5$ if $n\geq 24$. In particular, the second term is sufficiently small provided $\varepsilon$ is small. \\

If $\widehat{Y}\leq \widehat{P}^{1-\eta}$ we instead estimate the integral $I_{r_N}(\underline{\theta}, \bm{v})$ directly via Lemma \ref{Le: IntEstimate} to obtain 
\begin{align*}
E_{1,b}(Y,\Theta_1,\Theta_2)&\ll \frac{\widehat{P}^n}{\widehat{Y}^n}\widehat{\Theta}_1\widehat{\Theta}_2\widehat{Z}^{1-n/2-\mu}\Sigma(\widehat{V}),
\end{align*}
where 
\[
\mu = \begin{cases} 1/2 &\text{if }\widehat{P}\widehat{\Theta}_1\ll \widehat{\Theta}_2\\
0 &\text{else.}\end{cases}
\]
We can now apply \eqref{Eq: Estimate.Sigma(W)} with $\widehat{W}=\widehat{V}$ to deduce that 
\begin{align*}
    E_{1,b}(Y, \Theta_1,\Theta_2) & \ll \widehat{P}^n \widehat{Y}^{5/2-n/2+\varepsilon}\widehat{\Theta}_1\widehat{\Theta}_2\widehat{Z}^{1-n/2-\mu}\left(\widehat{V}^n+\widehat{Y}^{n/3}\right)\\
    &= \widehat{Y}^{5/2+n/2+\varepsilon}\widehat{\Theta}_1\widehat{\Theta}_2\widehat{Z}^{1+n/2-\mu}+\widehat{P}^{n-5}\widehat{Y}^{5/2-n/6+\varepsilon}\widehat{Z}^{3-n/2-\mu}.
\end{align*}
The second term is satisfactory by \eqref{Eq: LowerBoundYZ} if $n>15$, so we may assume the first one dominates. Recall that we already dealt with the case when $\widehat{Y}\widehat{\Theta}_1\geq \widehat{P}^{-\delta}$, where $\delta = 8(n-16)/(3n-24)$, since we assume \eqref{Cond: Thetalarge.rsmall} after \eqref{Eq: Decomp.N(P)1} does not hold. So we may suppose the contrary is true. There are now two cases: First, we assume that $\widehat{\Theta}_2\ll \widehat{P}\widehat{\Theta}_1$. In this situation we have $\widehat{Z}\ll \widehat{P}^{3-\delta}\widehat{Y}^{-1}$ and $\mu=0$, so that 
\begin{align*}
    E_{1,b}(Y,\Theta_1,\Theta_2)\ll \widehat{P}^{(3-\delta)(1+n/2)+1-2\delta}\widehat{Y}^{-1/2+\varepsilon}.
\end{align*}
A rather involved computation or a check with a computer algebra program verifies that $(3-\delta)(1+n/2)+1-2\delta<n-5$ if $n\geq 25$, which is satisfactory.\\

The only case that remains is when $\widehat{P}\widehat{\Theta}_1\ll \widehat{\Theta}_2$, in which case $\mu=1/2$ and hence 
\[
E_{1,b}(Y,\Theta_1,\Theta_2)\ll \widehat{P}^{5n/6 +5/6 -5/3}\widehat{Y}^{\varepsilon},
\]
which is satisfactory for $n>25$.

\printbibliography
\end{document}